%% file: origami_slope_gaps_and_hall.tex
\documentclass{amsart}

\input{preamble.tex}

\author[T.~McAdam]{Taylor McAdam}
\address{Department of Mathematics \& Statistics \\
Pomona College\\
610 N. College Avenue\\
Claremont, CA 91711}
\email{\href{mailto:taylor.mcadam@pomona.edu}{taylor.mcadam@pomona.edu}}

\author[X.Yu]{Xiaoxing Yu}
\email{\href{mailto:xyaa2021@mymail.pomona.edu}{xyaa2021@mymail.pomona.edu}}


\title{Origami slope gaps and the Hall distribution}

\begin{document}

\begin{abstract}
In this paper we review the theory of slope gap distributions of translation surfaces and summarize the state-of-the-art for calculating slope gap distributions of Veech surfaces. We then derive the slope gap distribution of a particular 10-tile origami by considering the origami's return times to a Poincar\'e section under the horocycle flow on the moduli space associated with the origami. We show that the resulting distribution is not a sum of scaled Hall distributions, unlike all previously published origami slope gap distributions. More generally, this demonstrates that the slope gap distribution of a branched covering of a translation surface cannot necessarily be represented as a sum of scaled copies of the slope gap distribution of the base surface.
\end{abstract}

\maketitle

\tableofcontents


\input{introduction}
\input{preliminaries}
\input{ksw-algorithm}
\input{sumry-algorithm}
\input{calculations}

\appendix

 \input{code}


 \bibliographystyle{plain}

\bibliography{origami_slope_gaps_and_hall}

\end{document}

%% file: preamble.tex
\usepackage{natbib}

\newcommand{\abs}[1]{\left\lvert {#1} \right\rvert}
\newcommand{\ang}[1]{\left\langle #1 \right\rangle}
\newcommand{\set}[1]{\left\{ #1 \right\}}

\numberwithin{equation}{section}

\theoremstyle{plain}

\newtheorem{thm}{Theorem}[section]

\newtheorem{lem}[thm]{Lemma}

\theoremstyle{definition}

\newtheorem{defn}[thm]{Definition}

\newtheorem{alg}[thm]{Algorithm}

\theoremstyle{remark}

\usepackage{tikz}
\usetikzlibrary{arrows.meta,decorations.markings,backgrounds,positioning,fpu}
\tikzset{
    bullet/.style={inner sep=.5mm,circle,fill},
    open bullet/.style={inner sep=.5mm,circle,fill=white,draw}
}
\usepackage{tikz-cd}
\usepackage{nicematrix}

\tikzset{
    bullet/.style={circle,fill=#1,inner sep=2pt},
    circ/.style={circle,fill=white,draw=#1,inner sep=2pt},
    bullet/.default={black},
    circ/.default={black},
    id/.pic={
        \fill[xshift=3pt] (150:6pt) to[bend right=10] (0:3pt) to[bend right=10] (-150:6pt)--cycle;
    },
    iid/.pic={
        \fill[xshift=0pt] (150:6pt) to[bend right=10] (0:3pt) to[bend right=10] (-150:6pt)--cycle;
        \fill[xshift=6pt] (150:6pt) to[bend right=10] (0:3pt) to[bend right=10] (-150:6pt)--cycle;
    },
    iiid/.pic={
        \fill[xshift=-3pt] (150:6pt) to[bend right=10] (0:3pt) to[bend right=10] (-150:6pt)--cycle;
        \fill[xshift=3pt] (150:6pt) to[bend right=10] (0:3pt) to[bend right=10] (-150:6pt)--cycle;
        \fill[xshift=9pt] (150:6pt) to[bend right=10] (0:3pt) to[bend right=10] (-150:6pt)--cycle;
    },
    ivd/.pic={
        \fill[xshift=-6pt] (150:6pt) to[bend right=10] (0:3pt) to[bend right=10] (-150:6pt)--cycle;
        \fill[xshift=0pt] (150:6pt) to[bend right=10] (0:3pt) to[bend right=10] (-150:6pt)--cycle;
        \fill[xshift=6pt] (150:6pt) to[bend right=10] (0:3pt) to[bend right=10] (-150:6pt)--cycle;
        \fill[xshift=12pt] (150:6pt) to[bend right=10] (0:3pt) to[bend right=10] (-150:6pt)--cycle;
    },
    i/.pic={
        \draw (0,-3pt) -- (0,3pt);
    },
    ii/.pic={
        \foreach \x in {-1.5pt,1.5pt} {
            \draw (\x,-3pt) -- (\x,3pt);
        }
    },
    iii/.pic={
        \foreach \x in {-3pt,0,3pt} {
            \draw (\x,-3pt) -- (\x,3pt);
        }
    },
    iv/.pic={
        \foreach \x in {-4.5pt,-1.5pt,1.5pt,4.5pt} {
            \draw (\x,-3pt) -- (\x,3pt);
        }
    },
}

\usepackage{pgfplots}
\pgfplotsset{compat=1.10}
\usepgfplotslibrary{fillbetween}

\usepackage{amsmath,amsthm,amssymb,amsfonts,amstext}
\usepackage{mathtools}
\usepackage{float}
\usepackage{enumitem}
\usepackage{todonotes}
\usepackage{booktabs}
\usepackage{lipsum}
\usepackage{adjustbox}
\usepackage{hyperref}
\usepackage[frozencache,cachedir=.]{minted}

\renewcommand{\d}{\text{d}}
\DeclareMathOperator{\tr}{tr}
\DeclareMathOperator{\arctanh}{arctanh}

\newcommand{\GL}[2][]{\text{GL}_{#1}(#2)}
\newcommand{\SL}[2][]{\text{SL}_{#1}(#2)}

\newcommand{\PSL}[2][]{\text{PSL}_{#1}(#2)}


%% file: introduction.tex
\section{Introduction}
\label{ch:introduction}

The theory of dynamical systems has been fruitfully applied to model a variety of physical, chemical, ecological, and sociological phenomena, and for this reason it is often considered an important pillar of applied mathematics. However, dynamical approaches can also be a powerful tool in tackling ``pure'' mathematical puzzles, including problems in number theory, analysis, and geometry.

In recent years, dynamical approaches have been applied with great success in the study of translation surfaces and their moduli spaces, a subject with connections to algebraic and complex geometry, Teichm{\"u}ller theory, and homogeneous dynamics, among other fields. In particular, dynamical systems prove a useful tool in computing the distribution of gaps between slopes of saddle connections on a subclass of highly symmetrical translation surfaces known as Veech surfaces (for precise definitions of these terms, see Section \ref{ch:preliminaries}). In much the same way that it is possible to model a real-world system by considering a flow on a lower-dimensional phase space, it is possible to reduce a question about slope gap distributions to analyzing a two-dimensional subspace of a three-dimensional dynamical system.

Slope gap distributions are rare examples of ``naturally'' occurring nonsmooth probability distributions, and are in general quite complicated. However, to our knowledge, until now the subclass of square-tiled surfaces within Veech surfaces have all had slope gap distributions of a simpler form. Namely, all slope gap distributions of square-tiled surfaces calculated previously have been finite sums of scaled copies of a distribution known as the Hall distribution. In this paper, we investigate whether this pattern holds in general for square-tiled surfaces. Using dynamical techniques developed and refined by previous authors (see \cite{athreya-cheung14, athreya-chaika-lelievre15, uyanik-work16, kumanduri-et-al24, alassal-et-al25}) we show that there exists a square-tiled surface whose slope gap distribution is not representable as a finite sum of scaled Hall distributions.\footnote{Since completing this paper, we have learned of unpublished work by Michael Beers and Gabriela Brown, who independently discovered examples of square-tiled surfaces whose slope gap distributions do not appear to be a sum of Hall distributions. Using code they developed, they also compute the slope gap distributions for a large number of square-tiled surfaces with small numbers of tiles. Their work is forthcoming \cite{beers-brown}. }

The paper is organized as follows. We begin in \S\,\ref{sec:prior-research} with a brief introduction to the history of research on slope gap distributions. This section uses more technical terms than elsewhere in this paper, and readers unfamiliar with the field may choose to return to it after reading the remainder of the paper. Next, \S\,\ref{ch:preliminaries} sets up the major definitions and results we use in the rest of the paper. In the following sections, we introduce in turn Kumanduri, Sanchez, and Wang's algorithm for finding slope gap distributions (\S\,\ref{ch:algorithm}) and the SUMRY 2023 research group's algorithm for finding winning vectors (\S\,\ref{ch:determining-winners}). Finally, we present our main calculations and results in \S\,\ref{ch:calculations}.

\subsection{Prior research}
\label{sec:prior-research}

There is a rich connection between sequences of number theoretic interest and geometric objects. The most relevant example to our work here is that of the Farey sequence. The Farey sequence $F_n$ of order $n$ ($n>1$) is the increasing sequence of fractions $a/b$ where $a$ and $b$ are coprime integers and $b\leq n$. In 1970, R.\ R.\ Hall, largely relying on standard arguments from calculus, showed that the distribution of gaps in $F_n$ as $n$ tends to infinity, suitably renormalized, is a piecewise smooth distribution now referred to as the Hall distribution \cite{hall70}.

As early as 1948, Richards recognized a connection between the Farey sequence and the slopes of closed geodesics on the torus in considering the problem of when two periodic waves overlap their pulses \cite{richards48}. In the 2000s, Boca, Cobeli, and Zaharescu developed work connecting the Farey sequence and Hall's distribution of their gaps to lattice points in the plane \cite{boca-cobeli-zaharescu00} and define a dynamical map to study certain sums of Farey fractions \cite{boca-cobeli-zaharescu01}. Boca and Zaharescu further went on to connect the Farey sequence to two-dimensional tori, and bring up an interesting application in the physical sciences: considering the scattering of molecules in a periodic Lorenz gas motivates looking at linear trajectories (specifically, closed geodesics) on the flat torus \cite{boca-zaharescu06}. Marklof and Str\"ombergsson also studied the periodic Lorenz gas in connection to the distribution of lattice points \cite{marklof-strombergsson10}, and their methods served as inspiration for much of the work on slope gap distributions of translation surfaces that would follow.

A separate stream of research was developing in parallel to these advances in understanding the Farey sequence. In the 1970s, roughly contemporaneous to Hall's paper deriving his famous distribution, Thurston \cite{thurston88} and Veech \cite{veech82, veech89} made pioneering advances in the areas of moduli spaces and the distribution of closed curves and geodesic trajectories on surfaces. One of the subclasses of surfaces studied by Thurston and Veech is the one which we treat here, namely, square-tiled surfaces, also called origamis, a term due to Lochak \cite{lochak05}. By the time Boca, Cobeli, and Zaharescu's work started to appear, these essentially geometric problems had been explicitly connected to number-theoretic problems about the statistical distributions of gaps. For example, Elkies and McMullen studied the distribution of gaps in the sequence $\{\sqrt{n}\}\mod1$ using ergodic theory on elliptic curves \cite{elkies-mcmullen04}.

Before gaps between saddle connections on translation surfaces had been extensively studied, much work was done to understand the distribution of saddle connections themselves. In the late 1980s, Masur showed that the number of saddle connections of bounded length grows quadratically with the length for any translation surface \cite{masur88, masur90}. Veech showed that there is an exact quadratic asymptotic for any Veech surface \cite{veech89}, and that the directions of saddle connections equidistribute in $S^1$ for any Veech surface as the length goes to infinity \cite{veech98}. Similarly, Eskin and Masur showed that there is an exact asymptotic for the number of saddle connections of bounded length on almost every surface in a stratum of translation surfaces \cite{eskin-masur01}, and Vorobets showed that the directions of saddle connections equidistribute in $S^1$ for almost every surface in a stratum \cite{vorobets96}. In 2019, Dozier showed that the uniform measure taken over all saddle connections of bounded length converges weakly to the area measure on the surface as the length tends to infinity \cite{dozier19}.

In 2012, using the methods of \cite{marklof-strombergsson10}, Athreya and Chaika turned their attention to the spacing between saddle connections. They showed that the distribution of renormalized gaps between saddle connection directions is the same for almost every surface in a stratum and that this distribution has support at zero, in contrast to the gap distribution of a Veech surface, which never has support at zero. In 2014, Athreya and Cheung \cite{athreya-cheung14} built on this work to re-derive Hall's distribution of Farey gaps in the framework of dynamical systems theory. In their paper, Athreya and Cheung consider the torus as a point in the modular surface $\SL[2]{\mathbb{R}}\mathbin{/}\SL[2]{\mathbb{Z}}$ and subject it to a continuous subgroup action of $\SL[2]{\mathbb{R}}$ known as the horocycle flow. Considering the horocycle flow as a dynamical system, Athreya and Cheung find for it a transversal and a return map, and show that parametrizing the transversal reduces the problem of finding the slope gap distribution to finding areas of subsets ``swept-out'' by level sets of the return time function. 

By framing the problem of slope gaps on the torus in terms of a dynamical system on the modular surface, Athreya and Cheung opened the door to generalizing the study of Farey fractions and the torus to investigating gap distributions of higher-genus translation surfaces by looking at the horocycle flow on the moduli space associated to a given surface. Athreya, Chaika, and Leli\`evre were the first to take take this approach for a translation surface other than the torus in \cite{athreya-chaika-lelievre15}, where they derived the gap distribution for the ``golden L" surface, and a similar method has been used in some form in all subsequent work on the topic. Athreya would later go on to refine and summarize this work in \cite{athreya16}. Uyanik and Work adapt the method of \cite{athreya-chaika-lelievre15} to calculate the slope gap distribution of the regular octagon surface \cite{uyanik-work16}, and they provide an algorithm for computing the slope gap distribution of an arbitrary Veech surface. Berman, McAdam, Miller-Murthy, Uyanik, and Wan use the same algorithm in \cite{berman-et-al21} to calculate the slope gap distributions of regular $2n$-gon translation surfaces. The authors of \cite{berman-et-al21} were able to show that for certain values of $n$, the $2n$-gon had bimodal slope gap distributions, a pattern which had not been observed previously.

Despite its usefulness, Uyanik and Work's algorithm could in certain instances fail to terminate in a finite number of steps. In 2024, Kumanduri, Sanchez, and Wang refined the algorithm by adjusting the parametrization of the transversal so that the slope gap distribution is always calculable, so long as the return time function can be determined at every point on the transversal \cite{kumanduri-et-al24}. Subsequently, Al Assal, Ali, Arengo, McAdam, Newman, Scully, and Zhou (the SUMRY 2023 research group) provided a systematic method for determining the return time function in most cases \cite{alassal-et-al25}. It is Kumanduri, Sanchez, and Wang's modified algorithm for parametrizing the transversal and the SUMRY 2023 research group's method for finding the return time function that we adapt in this paper to calculate the slope gap distributions of square-tiled surfaces.

Finally, we note that although we have summarized some of the literature on translation surfaces and gap distributions most relevant to the work we present in this paper, this list is not exhaustive. A curious reader may also be interested in the the following works: \cite{heersink16,taha18,taha19,sanchez22,work20,osman-southerland-wang25}. 

\subsection{Acknowledgments}

This work began as a thesis project at Pomona College, and the authors are grateful to Pomona College, and especially the Department of Mathematics and Statistics, for providing the infrastructure and support that made it possible. We would also like to thank Stephan Garcia for a close reading and helpful comments on an earlier draft of the manuscript. We are grateful to Aaron Calderon for first communicating the 
``sum of Halls" question to us. We would also like to thank the members of the 2023 Summer Undergraduate Mathematics Research at Yale (SUMRY) slope gap distribution project---namely, Fernando Al Assal, Nada Ali, Uma Arengo, Carson Newman, Noam Scully, and Sophia Zhou---for making their then-unpublished algorithm available to us. Finally, we are grateful to Jordan Grant and Jo O'Harrow of the SUMRY 2021 research group for providing code that we adapted to find generators for the maximal parabolic subgroups of the Veech group of a square-tiled surface.

%% file: preliminaries.tex
\section{Preliminaries}
\label{ch:preliminaries}

\subsection{Translation surfaces}
\label{sec:translation-surfaces}

Let $\Delta_1$, $\Delta_2$ be two polygons in the plane with boundaries oriented counterclockwise, so that an observer traversing the boundary of each polygon according to its orientation sees the interior of the polygon to its left at all times in its journey. If we imagine the polygons as paper cut-outs lying on a flat table, we might wonder what shapes we can make by gluing the polygons together along their sides in a way that respects the orientations of their boundaries. This idea motivates the following definition.

\begin{defn}\label{def:gluing}
Let $\Delta_1$ and $\Delta_2$ be two polygons in the plane with counter\-clockwise-oriented boundaries. Let $s_1$ and $s_2$ be two parallel, same-length sides of $\Delta_1$ and $\Delta_2$, respectively, with opposite orientations. Then there exists a translation $L\colon\mathbb{R}^2\to\mathbb{R}^2$ satisfying $L(s_1)=s_2$.

Let $\sim$ be defined as follows. For any points $\mathbf{x}$, $\mathbf{y}\in s_1\cup s_2$, $\mathbf{x}\sim\mathbf{y}$ if and only if one or more of the following conditions are true:
\begin{itemize}
    \item $\mathbf{x}=\mathbf{y}$;
    \item $\mathbf{x}\in s_1$, $\mathbf{y}\in s_2$, and $L(\mathbf{x})=\mathbf{y}$;
    \item $\mathbf{x}\in s_2$, $\mathbf{y}\in s_1$, and $L(\mathbf{y})=\mathbf{x}$.
\end{itemize}
It is straightforward to check that $\sim$ is reflexive, symmetric, and transitive, so that it defines an equivalence relation on $s_1\cup s_2$. We refer to the natural projection mapping $s_1\cup s_2$ to the quotient space $(s_1\cup s_2)\mathbin{/}\sim$ as \textit{gluing} the segments $s_1$ and $s_2$ together \textit{by translation}, and say that the quotient space is \textit{glued by translation}. 
\end{defn}

Because we do not consider other kinds of gluing in this paper, we refer to gluing by translation as simply \textit{gluing}. In the simple case of gluing two segments $s_1$, $s_2$ together, it is possible to show that the quotient space $(s_1\cup s_2)\mathbin{/}\sim$ is homeomorphic to both $s_1$ and $s_2$. We may therefore say that gluing two segments together is identifying them as one and the same side.

Notice that $\sim$ extends naturally to all points in $\mathbb{R}^2$, which allows gluing together more complicated structures. The following definition is adapted from \cite{masur-06}.

\begin{defn}\label{def:polygonal-representation}
Let $\set{\Delta_1,\Delta_2,\dots,\Delta_n}$ be a finite collection of counter\-clock\-wise-oriented polygons in the plane with sides $\set{s_1,s_2,\dots,s_m}$ such that for every $1\leq i\leq m$, there corresponds to each $s_i$ a unique index $1\leq k\leq m$ selecting a side $s_k$ such that $s_i$ and $s_k$ are parallel, of equal length, and oppositely oriented. Let
\[\Delta=\Delta_1\cup\Delta_2\cup\dots\cup\Delta_n\,,\]
let every $s_i$ and $s_k$ be glued by (in general distinct) translations for $1\leq i\leq m$, and let the equivalence relation on $\Delta$ defined by these gluings be $\sim$. Then the pair $(\Delta,{\sim})$ is a \textit{polygonal representation}.
\end{defn}

Before proceeding, we illustrate this construction with an example. Let $\Delta_1$, $\Delta_2$ be counterclockwise-oriented heptagons and let their edges be as in the upper half of Figure~\ref{fig:general-translation-surface}. The total number of sides is $m=14$. For each $1\leq i\leq 7$, let $s_i$ be glued to $s_{i+7}$. It is easy to check that $s_i$ is oriented oppositely to $s_{i+7}$ for $1\leq i\leq 7$, so $\Delta_1\cup\Delta_2$ satisfies the conditions of Definition~\ref{def:polygonal-representation} and forms a polygonal representation with the appropriate gluing $\sim$.

By imagining an observer traversing this surface, we make some observations about Definition~\ref{def:polygonal-representation}. First, the positions of the $\Delta_i$ in the plane are irrelevant to the observer, who has no way to know its position in $\mathbb{R}^2$. In other words, the intrinsic geometry of the surface does not depend on the exact locations in the plane of its constituent polygons. Second, intrinsically speaking, there is no canonical decomposition of the surface into constituent polygons; distinct decompositions may yield the same intrinsic geometry. This is shown in the lower half of Figure~\ref{fig:general-translation-surface}, where the sides $s_1'$ and $s_8'$ are glued together. An observer moving rightward across $s_1'$ in the lower half of Figure~\ref{fig:general-translation-surface} would find itself emerging onto the same point as it would moving rightward in the corresponding location, not across any edge, in the upper half of Figure~\ref{fig:general-translation-surface}. Similarly, an observer moving rightward across $s_1$ in the upper half of Figure~\ref{fig:general-translation-surface} would emerge onto the same points as it would crawling across the dashed line in the lower half.

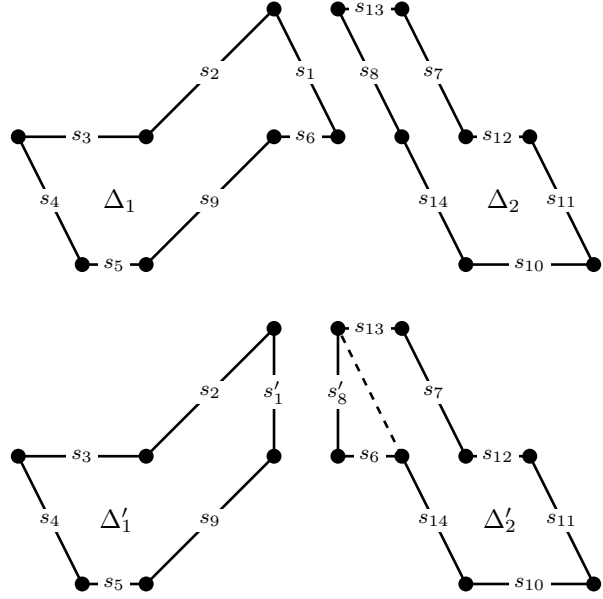
\begin{figure}
    \centering
    \tikzset{side label/.style={fill=white,inner sep=2pt,midway,font=\footnotesize}}
    \begin{tikzpicture}[scale=1.7,line width=1pt]
    \begin{scope}
        \draw (0,0) node[bullet] {}
        -- ++(-.5,1) node[bullet] {} node[side label] {$s_1$}
        -- ++(-1,-1) node[bullet] {} node[side label] {$s_2$}
        -- ++(-1,0) node[bullet] {} node[side label] {$s_3$}
        -- ++(.5,-1) node[bullet] {} node[side label] {$s_4$}
        -- ++(.5,0) node[bullet] {} node[side label] {$s_5$}
        -- ++(1,1) node[bullet] {} node[side label] {$s_9$}
        -- cycle node[side label] {$s_6$};

        \begin{scope}[xslant=-.5]
            \draw (.5,0) node[bullet] {}
            -- ++(0,1) node[bullet] {} node[side label] {$s_8$}
            -- ++(.5,0) node[bullet] {} node[side label] {$s_{13}$}
            -- ++(0,-1) node[bullet] {} node[side label] {$s_7$}
            -- ++(.5,0) node[bullet] {} node[side label] {$s_{12}$}
            -- ++(0,-1) node[bullet] {} node[side label] {$s_{11}$}
            -- ++(-1,0) node[bullet] {} node[side label] {$s_{10}$}
            -- cycle node[side label] {$s_{14}$};
        \end{scope}

        \node at (-1.75,-.5) {$\Delta\mathrlap{_1}$};
        \node at (1.25,-.5) {$\Delta\mathrlap{_2}$};
    \end{scope}

    \begin{scope}[yshift=-2.5cm]
        \draw (-.5,0) node[bullet] {}
        -- ++(0,1) node[bullet] {} node[side label] {$s_1'$}
        -- ++(-1,-1) node[bullet] {} node[side label] {$s_2$}
        -- ++(-1,0) node[bullet] {} node[side label] {$s_3$}
        -- ++(.5,-1) node[bullet] {} node[side label] {$s_4$}
        -- ++(.5,0) node[bullet] {} node[side label] {$s_5$}
        -- cycle node[side label] {$s_9$};

        \draw (.5,0) node[bullet] {}
        -- ++(-.5,0) node[bullet] {} node[side label] {$s_6$}
        -- ++(0,1) node[bullet] {} node[side label] {$s_8'$}
        -- ++(.5,0) node[bullet] {} node[side label] {$s_{13}$}
        -- ++(.5,-1) node[bullet] {} node[side label] {$s_7$}
        -- ++(.5,0) node[bullet] {} node[side label] {$s_{12}$}
        -- ++(.5,-1) node[bullet] {} node[side label] {$s_{11}$}
        -- ++(-1,0) node[bullet] {} node[side label] {$s_{10}$}
        -- cycle node[side label] {$s_{14}$};

        \draw[dashed] (.5,0) -- (0,1);

        \node at (-1.75,-.5) {$\Delta\mathrlap{_1}'$};
        \node at (1.25,-.5) {$\Delta\mathrlap{_2}'$};
    \end{scope}
    \end{tikzpicture}
    \caption{Two polygonal representations that cannot be distinguished by their intrinsic geometries.}
    \label{fig:general-translation-surface}
\end{figure}

To reiterate, distinct collections of polygons with distinct gluings can represent the same surface from the point of view of an observer on the surface. We make this notion precise with the following definition.

\begin{defn}
Let the finite collection of polygons $\set{\Delta_1,\Delta_2,\dots,\Delta_n}$ and the gluing $\sim$ define a polygonal representation. The following three operations are defined for any polygonal representation.
\begin{itemize}
    \item \textit{Cut} a polygon $\Delta_i$ by selecting two points on its boundary and adding a new side between the two points. Treating the new side as the union of two glued sides, it divides $\Delta_i$ into two new polygons.
    \item \textit{Translate} any polygon $\Delta_i$, ensuring its interior remains disjoint from the interior of $\Delta_k$ for $i\neq k$.
    \item \textit{Paste} any two polygons $\Delta_i$ and $\Delta_k$ together along two coinciding glued edges. This builds a new polygon out of the two old polygons and removes a side.
\end{itemize}
We refer to these operations collectively as \textit{cut-translate-paste} operations.
\end{defn}

Cut-translate-paste operations do not change the intrinsic geometry of a polygonal representation, so it is natural to define equivalence classes of polygonal representations related by sequences cut-translate-paste operations.

\begin{defn}
\label{def:translation-surface}
If two polygonal representations $(\Delta,\sim)$ and $(\Delta',{\sim}')$ are related by a (possibly empty) sequence of cut-translate-paste operations, that is, if applying a sequence of cut-translate-paste operations to $(\Delta,\sim)$ yields $(\Delta',{\sim}')$ or vice versa, we say they are equivalent. A cut-translate-paste equivalence class of polygonal representations is a \textit{translation surface}.
\end{defn}

It is straightforward to verify that the relation given in the definition is indeed an equivalence relation: its reflexivity follows from the fact that any polygonal representation is related to itself through an empty sequence of cut-translate paste operations, its symmetry follows from the fact that cut-translate-paste operations are invertible, and its transitivity follows from the fact that sequences of cut-translate-paste operations can be composed.

Now we can restate the equivalence of the two polygonal representations in Figure~\ref{fig:general-translation-surface} in terms of Definition~\ref{def:translation-surface}. Cutting $\Delta_1$ through the left vertices of $s_1$ and $s_6$ on the top polygonal representation, translating the given chunk right so that $s_1$ coincides with $s_8$, and gluing $s_1$ and $s_8$ together yields the bottom polygonal representation.

\subsubsection{Translation surfaces as Riemann surfaces}
\label{subsec:riemann-surfaces}

There is an equivalent definition of translation surfaces as a pair $(X,\omega)$ where $X$ is a Riemann surface and $\omega$ is a holomorphic $1$-form on $X$ \citep[\S\,2.5]{athreya-masur-24}. We do not say more about this construction, but mention it because it is standard in the field to notate translation surfaces $(X,\omega)$, a convention we follow in the rest of this paper. The Riemann surface notation is convenient because it is independent of any particular polygonal representation. This permits a cleaner exposition of how linear transformations on the polygonal representations of translation surfaces induce linear transformations on the underlying translation surfaces themselves.

A reader interested in learning more about the rich theory of translation surfaces and their connection to Riemann surfaces is encouraged to explore the following excellent resources: \cite{zorich06, hubert-schmidt06, wright15, athreya-masur-24}. 

\subsubsection{Linear transformations of translation surfaces}
\label{subsec:linear-transformations}

Linear transformations of the plane naturally induce transformations on polygonal representations. The following definition makes this precise.

\begin{defn}
Let $A\colon\mathbb{R}^2\to\mathbb{R}^2$ be a linear transformation and let $(\Delta,\sim)$ be a polygonal representation with sides $\set{s_i}_{i=1}^n$. Then $A$ induces a linear transformation of $(\Delta,\sim)$ by defining
\[A((\Delta,\sim))=(\Delta',\sim')\,,\]
where $\Delta'=A(\Delta)$ and $\sim'$ glues $A(s_i)$ and $A(s_k)$\footnote{Since linear transformations send parallel lines to parallel lines, these sides can still be glued together using Definition~\ref{def:gluing}.} together if and only if $\sim$ glues $s_i$ and $s_k$ together.
\end{defn}

Linear transformations of polygonal representations in turn induce linear transformations on translation surfaces. This is because if two polygonal representations $(\Delta_1,\sim_1)$ and $(\Delta_2,\sim_2)$ are related by a sequence of cut-translate-paste operations, their images $A((\Delta_1,\sim_1))$ and $A((\Delta_2,\sim_2))$ are also related by a sequence of cut-translate-paste operations. In place of a full proof, we briefly sketch the argument here. Given a sequence of cut-translate-paste operations mapping $(\Delta_1,\sim_1)$ to $(\Delta_2,\sim_2)$, construct a new cut-translate-paste sequence as follows: for every cut between points $p$ and $q$ in the original sequence, insert a cut between points $A(p)$ and $A(q)$ in the new sequence; for every translation through a vector $\mathbf{t}$ in the original sequence, insert a translation through a vector $A\mathbf{t}$ in the new sequence; for every paste between two sides $r$ and $s$ in the original sequence, insert a paste between the two sides $A(r)$ and $A(s)$ in the new sequence. Then the new sequence maps $A((\Delta_1,\sim_1))$ to $A((\Delta_2,\sim_2))$. Hence a linear transformation on a translation surface is well-defined.

Certain subsets of linear transformations on translation surfaces can be endowed with algebraic structure. In particular, the group of invertible $2\times2$ real matrices $\GL[2]{\mathbb{R}}$ defines a group action on the set of translation surfaces. Since a linear transformation may map a polygonal representation to a distinct polygonal representation of the same translation surface, in general translation surfaces have nontrivial stabilizers under this action. Any linear transformation whose action maps a translation surface to itself must preserve volume and orientation. The stabilizer of a translation surface $(X,\omega)$ in $\GL[2]{\mathbb{R}}$ is therefore a subgroup of $\SL[2]{\mathbb{R}}$, the group of unit-determinant $2\times2$ real matrices. This motivates the notation $\SL{X,\omega}$ for the stabilizer of $(X,\omega)$ in $\GL[2]{\mathbb{R}}$.

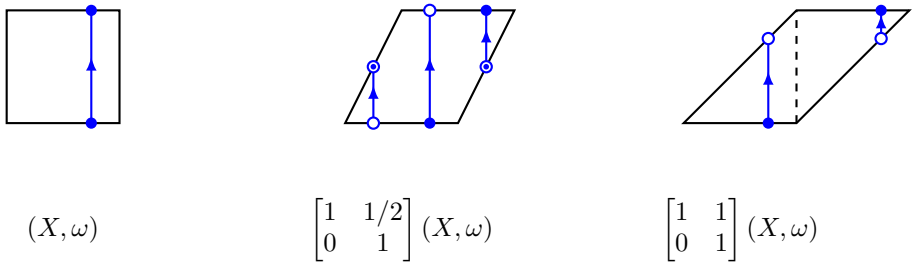
\begin{figure}
    \centering
    \begin{tikzpicture}[scale=1.5,thick]
        \draw (0,0) rectangle +(1,1);
        \draw (3,0) -- +(1,0) -- +(1.5,1) -- +(0.5,1) -- cycle;
        \draw (6,0) -- +(1,0) -- +(2,1) -- +(1,1) -- cycle;

        \node[anchor=base] at (.5,-1) {$(X,\omega)$};
        \node[anchor=base] at (3.5,-1) {$\begin{bmatrix}1 & 1/2 \\ 0 & 1\end{bmatrix}(X,\omega)$};
        \node[anchor=base] at (6.5,-1) {$\begin{bmatrix}1 & 1 \\ 0 & 1\end{bmatrix}(X,\omega)$};

        \draw[blue] (.75,0) node[bullet=blue,inner sep=1.5pt] {} -- pic[midway,sloped,scale={2/3}] {id} +(0,1) node[bullet=blue,inner sep=1.5pt] {};
        \draw[blue] (3.75,0) node[bullet=blue,inner sep=1.5pt] {} -- pic[midway,sloped,scale={2/3}] {id} +(0,1) node[circ=blue,inner sep=1.5pt] {};
        \draw[blue] (3.25,0) node[circ=blue,inner sep=1.5pt] {} -- pic[midway,sloped,scale={2/3}] {id} +(0,.5) node[circ=blue,inner sep=1.5pt] {} node[bullet=blue,inner sep=.75pt] {};
        \draw[blue] (4.25,0.5) node[circ=blue,inner sep=1.5pt] {} node[bullet=blue,inner sep=.75pt] {} -- pic[midway,sloped,scale={2/3}] {id} +(0,.5) node[bullet=blue,inner sep=1.5pt] {};
        \draw[blue] (6.75,0) node[bullet=blue,inner sep=1.5pt] {} -- pic[midway,sloped,scale={2/3}] {id} +(0,.75) node[circ=blue,inner sep=1.5pt] {};
        \draw[blue] (7.75,.75) node[circ=blue,inner sep=1.5pt] {} -- pic[midway,sloped,scale={2/3}] {id} +(0,.25) node[bullet=blue,inner sep=1.5pt] {};

        \draw[dashed] (7,0) -- +(0,1);
    \end{tikzpicture}
    \caption{A horizontal shear by $1$ preserves $(X,\omega)$, while a horizontal shear by $1/2$ does not. Identified points are marked with the same symbol.}
    \label{fig:stabilizing-transformations}
\end{figure}

Preserving volume and orientation is in general a necessary condition but not a sufficient one. Consider the translation surface $(X,\omega)$ consisting of a unit square with opposite edges identified, shown in Figure~\ref{fig:stabilizing-transformations}. Any shear of $(X,\omega)$ preserves its volume and orientation. However, a horizontal shear by $1/2$ does not map $(X,\omega)$ to itself. To see this, notice that $(X,\omega)$ contains a vertical loop of length one, indicated by the blue arrow on the leftmost diagram of Figure~\ref{fig:stabilizing-transformations}. However, the image of $(X,\omega)$ under a shear of $1/2$ does not contain any such loop. Any vertical loop of length $1$ must intersect the sheared surface's lower edge $[0,1]\times0$ at exactly one point, say $(x,0)$. If $x\geq1/2$, then a vertical path of length $1$ passing through this point must end at the point $(x,1)\sim(x-1/2,0)$. If $x\leq1/2$, then a vertical path of length $1$ passing through this point must cross the point $(x,2x)\sim(x+1,2x)$ and end at the point $(x+1,1)\sim(x+1/2,0)$. Since $(x\pm1/2,0)\not\sim(x,0)$ for any $x\in[0,1]$, the vertical path is not a loop in either case. It follows that no vertical loop of length $1$ exists on the sheared surface.

In contrast, a shear by $1$ does map $(X,\omega)$ to itself. Cutting along the dashed line on the right of Figure~\ref{fig:stabilizing-transformations}, translating the triangle containing $(2,1)$ left of the triangle containing $(0,0)$, and pasting the two along the diagonal side recovers $(X,\omega)$. It is also possible to verify that any vertical path of length $1$ forms a loop on this surface.

\subsection{Veech surfaces and Veech groups}
\label{sec:veech-surfaces-groups}

\subsubsection{The hyperbolic upper half-plane}
\label{subsec:hyperbolic-plane}

\begin{defn}
Let $\mathbb{H}^2=\set{x+iy\mid x\,,y\in\mathbb{R}\,,y>0}$ be the upper half of the complex plane endowed with the metric defined by
\[\d{s}^2=\frac{\d{x}^2+\d{y}^2}{y^2}\]
and the ensuing induced volume
\[\d{\mu}=\frac{\d{x}\,\d{y}}{y^2}\,.\]
The resulting metric space is called the \textit{Poincar{\'e} upper half-plane} or the \textit{hyperbolic upper half-plane}.

Also, let $\overline{\mathbb{H}^2}=\set{x+iy\mid x\,,y\in\mathbb{R}\,,y\geq0}\cup\{\infty\}$ be the metric compactification of $\mathbb{H}^2$.
\end{defn}

The group of unit-determinant $2\times2$ real matrices $\SL[2]{\mathbb{R}}$ induces a group action on $\mathbb{H}^2$ through linear-fractional transformations, i.e.\ \textit{M\"obius transformations}. In particular, for any
\[A=\begin{bmatrix}
    a & b \\ c & d
\end{bmatrix}\in\SL[2]{\mathbb{R}}\]
and any $z\in\mathbb{H}^2$, define the action of $A$ on $z$ to be
\[
A\cdot z=\frac{az+b}{cz+d}\,.
\]
It is tedious but straightforward to check that the above equation defines a group action which preserves the metric and the volume element. Since all elements of $\SL[2]{\mathbb{R}}$ have positive determinant $1$, it is a group of orientation-preserving isometries on $\mathbb{H}^2$. The action of $\SL[2]{\mathbb{R}}$ extends continuously to $\overline{\mathbb{H}^2}$.

The action of $\SL[2]{\mathbb{R}}$ on $\overline{\mathbb{H}^2}$ has kernel $\set{I,-I}$, where $I$ is the $2\times2$ identity matrix. This implies that elements of $\SL[2]{\mathbb{R}}$ and linear-fractional transformations of $\overline{\mathbb{H}^2}$ exist in $2$-to-$1$ correspondence: changing the sign of all entries of an element of $\SL[2]{\mathbb{R}}$ does not change its action on $\overline{\mathbb{H}^2}$. We write the quotient group $\SL[2]{\mathbb{R}}\mathbin{/}\set{I,-I}$ as $\PSL[2]{\mathbb{R}}$.

For any discrete subgroup $H\leq\SL[2]{\mathbb{R}}$, it is possible to find a connected subset of $\mathbb{H}^2$ which contains exactly one representative from each orbit under the action of $H$, except possibly on a set of measure zero. Such a subset is called a \textit{fundamental domain} for $H$. A translation surface $(X,\omega)$ whose stabilizer $\SL{X,\omega}$ has a finite-volume fundamental domain is called a \textit{Veech surface}, and $\SL{X,\omega}$ is called its \textit{Veech group}. Informally, Veech surfaces are translation surfaces with a high degree of symmetry---so much so that the group of transformations which fix them is ``big'' enough that every point in $\mathbb{H}^2$ can be carried into a finite-volume subset by transformations in the group.

\subsubsection{The modular group}
\label{subsec:modular-group}

A natural example of a discrete subgroup of $\SL[2]{\mathbb{R}}$ is the multiplicative group $\SL[2]{\mathbb{Z}}$ of $2\times2$ matrices with unit determinant, all of whose entries are integers. It is known that $\SL[2]{\mathbb{Z}}$ is generated by only two elements,
\[S=\begin{bNiceMatrix}[r]
    0 & -1 \\ 1 & 0
\end{bNiceMatrix}\quad\text{and}\quad T=\begin{bmatrix}
    1 & 1 \\ 0 & 1
\end{bmatrix}\,.\]
In terms of their action on $\mathbb{R}^2$, $S$ defines a counterclockwise rotation through a quarter turn, while $T$ defines a horizontal shear one unit to the right. The action of the modular group on the plane restricts naturally to translation surfaces.

The generators $S$ and $T$ define the linear-fractional transformations
\[S\cdot z=-\frac{1}{z}\quad\text{and}\quad T\cdot z=z+1\]
on $\mathbb{H}^2$. Geometrically, $S$ describes an inversion through the unit circle followed by a reflection across the imaginary axis, while $T$ describes a translation one unit to the right. We use this geometric interpretation to find a fundamental domain $\mathcal{F}$ for $\SL[2]{\mathbb{Z}}$. By successive applications of $T$ and its inverse, any point in $\mathbb{H}^2$ can be translated to lie in any infinite vertical strip of width $1$. The standard choice is the strip centered at $0$,
\[\set{z\in\mathbb{H}^2\,\middle\vert\,-\frac12<\Re(z)\leq\frac12}\,.\]
Next, any point lying in the interior of the unit circle can be inverted and reflected to the exterior by an application of $S$.\footnote{The matrix $S$ maps the points on the left half of the unit circle to the right half, so a measure-zero set of points in $\mathcal{F}$ have two representatives.} This defines the fundamental domain
\[\mathcal{F}=\set{z\in\mathbb{H}^2\,\middle\vert\,-\frac12<\Re(z)\leq\frac12\,,\abs{z}\geq1}\,,\]
shown in Figure~\ref{fig:fundamental-domain}. As required, the orbit of $\mathcal{F}$ under the action of $\SL[2]{\mathbb{Z}}$ is all of $\mathbb{H}^2$.

\begin{figure}
    \centering
    \begin{tikzpicture}[scale=1.35,font=\scriptsize]
        \clip (-4,0) rectangle (4,2.5);
    
        \fill[red!10] ({-1/2},2.5) -- ({-1/2},{sqrt(3)/2}) arc[radius=1,start angle=120,delta angle=-60] -- ({1/2},2.5);
        \fill[blue!10] ({1/2},2.5) -- ({1/2},{sqrt(3)/2}) arc[radius=1,start angle=120,delta angle=-60] -- ({3/2},2.5);
        \fill[green!10] ({-1/2},{sqrt(3)/2}) arc[radius=1,start angle=120,delta angle=-60] arc[radius=1,start angle=120,delta angle=60] arc[radius=1,start angle=0,delta angle=60];
        
        \foreach \x in {-4,-3,-2,-1,0,1,2,3,4}{
            \draw ({\x+1/2},0) -- +(0,2.5);
            \draw ({\x-1},0) arc[radius=1,start angle=180,delta angle=-180];
            \draw ({\x},0) arc[start angle=0,delta angle=180,radius={1/3}];
            \draw ({\x},0) arc[start angle=180,delta angle=-180,radius={1/3}];
        }
        
        \draw (-5,0) -- (5,0);

        \node at (0,1.5) {$\mathcal{F}$};
        \node at (0,0.8) {$S\cdot\mathcal{F}$};
        \node at (1,1.5) {$T\cdot\mathcal{F}$};
    \end{tikzpicture}
    \caption{The fundamental domain of the action of $\SL[2]{\mathbb{Z}}$ on $\mathbb{H}^2$ and its images under the actions of $S$ and $T$.}
    \label{fig:fundamental-domain}
\end{figure}
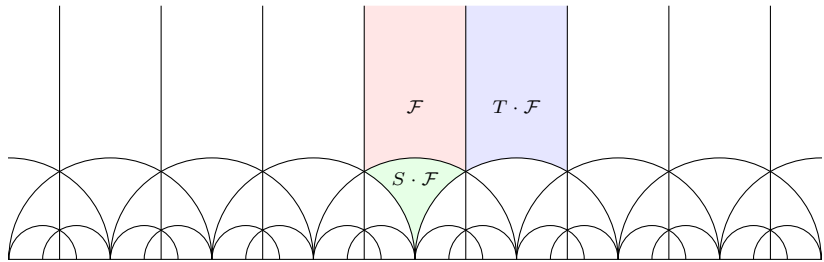

\subsection{Cone points and saddle connections}
\label{sec:cone-points-saddle-connections}

The definition of translation surfaces in terms of polygonal representations ensures that they are flat on the interiors of any constituent polygons as well as on any glued sides. In general, however, translation surfaces contain a finite set of points which do not possess locally flat neighborhoods. As an example, consider the translation surface formed by an octagon with opposite sides glued, shown in Figure~\ref{fig:octagon-cone-pts}. By tracking side identifications, we see that the octagon has just one vertex with eight representatives in the plane, labeled $\bullet$. Each interior angle of the polygon around $\bullet$ has a measure of three eighth-turns, so an observer traversing a small circle around $\bullet$ would make three full turns before returning to its starting location. Informally, the space is locally larger around $\bullet$ than anywhere else on the translation surface.\footnote{In fact, the extra space derives from the nonzero curvature of all neighborhoods of $\bullet$.} The following definitions formalize this notion.

\begin{figure}
    \centering
    \begin{tikzpicture}[scale=1.5]
        \foreach \n in {0,1,2,3,4,5,6,7} {
            \pgfmathsetmacro{\x}{cos(\n*45)}%
            \pgfmathsetmacro{\y}{sin(\n*45)}%
            \draw (\x,\y) -- ({cos((\n+1)*45)},{sin((\n+1)*45)}) node[bullet] {};
            
            \begin{scope}[on background layer]
                \pgfmathsetmacro{\r}{.15}%
                \draw[blue,rotate around={{\n*45}:(\x,\y)},fill=blue!10] ({\x+\r*cos(112.5)},{\y+\r*sin(112.5)}) arc[radius=\r,start angle=112.5,delta angle=135] -- (\x,\y) -- cycle;
            \end{scope}
        }
    \end{tikzpicture}
    \caption{An octagon with opposite sides glued forms a translation surface with a single vertex $\bullet$.}
    \label{fig:octagon-cone-pts}
\end{figure}
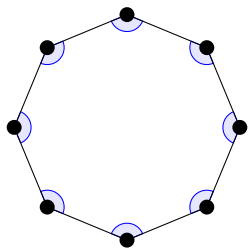

\begin{defn}
Let $(X,\omega)$ be a translation surface with a polygonal representation $(\Delta,\sim)$. Let
\[\Delta=\Delta_1\cup\Delta_2\cup\dots\cup\Delta_n\,,\]
where each $\Delta_i$ is a polygon. Assume without loss of generality that the $\Delta_i$ are all disjoint (if the $\Delta_i$ are not all disjoint, they can be translated to be so). Let $\pi\colon\Delta\to X$ be the natural projection associated to the quotient space $\Delta\mathbin{/}{\sim}=X$. Then each vertex $q$ of each $\Delta_i$ has an interior angle, which we denote $\theta(q)$. For any vertex $p\in X$, let $\set{q_k}_{k=1}^m$ be the preimage of $p$ under $\pi$ and let
\[\theta(p)=\sum_{k=1}^m\theta(q_k)\,.\]
We say $\theta(p)$ is the \textit{angle} of $p$. If $\theta(p)$ measures more than a full turn, we say $p$ is a \textit{cone point}\footnote{The term ``cone point'' comes from the fact that neighborhoods of a cone point are homeomorphic to neighborhoods of the apex of a Euclidean cone \citep[cf.][\S\S\,17.1--17.2]{schwartz11}}.
\end{defn}

In the octagon example above, $X$ contains the unique vertex $\bullet$ whose preimage under the projection contains eight points. The angle at each of these points is three-eighths of a turn, so the total measure of $\theta(p)$ is three turns. It turns out that the angle of a cone point does not depend on the polygonal representation and is always a whole number of turns. The following theorem proves this fact.

\begin{thm}
Let $p$ be a point of a translation surface $(X,\omega)$. Then the following is true:
\begin{enumerate}[label={(\arabic*)}]
    \item $\theta(p)$ is independent of the polygonal representation of $(X,\omega)$;
    \item $\theta(p)$ is a whole number of turns.
\end{enumerate}
\end{thm}

\begin{proof}
Let $\Delta=\Delta_1\cup\dots\cup\Delta_n$, let $(\Delta,\sim)$ be a polygonal representation of $(X,\omega)$, and let $p$ be a point of $X$ with preimage $\set{q_k}_{k=1}^m$ under the natural projection from $\Delta$ to $\Delta\mathbin{/}{\sim}$. Since any two equivalent polygonal representations are related by a sequence of cut-translate-paste operations, to prove (1) it suffices to show that no cut-translate-paste operation changes $\theta(p)$. We assume throughout, without loss of generality, that the $\Delta_i$ are all disjoint.

First, note that in the representation $(\Delta,\sim)$, $\theta$ is given by
\[\theta(p)=\theta(q_1)+\dots+\theta(q_m)\,.\]
Suppose we cut the polygon $\Delta_i$ and translate the resulting pieces into two disjoint polygons $\Delta_{i'}$ and $\Delta_{i''}$. For each point $q_k$ that the cut intersects, two new preimages of $p$ are created, $q_k'$ and $q_k''$. Since the cut does not add or remove any angle, the angles of $q_k'$ and $q_k''$ sum to the angle of $q_k$: symbolically, $\theta(q_k)=\theta(q_k')+\theta(q_k'')$. Translating any polygon $\Delta_i$ preserves its interior angles and so does not change any of the $\theta(q_k)$. Finally, any two vertices $q_k'$ and $q_k''$ glued together by a paste merge into the point $q_k$. Because the paste does not add or remove any angle, $\theta(q_k')+\theta(q_k'')=\theta(q_k)$. Hence no cut-translate-paste operation changes the value of $\theta(p)$, proving (1).

Next, notice that each vertex $q\in\Delta$ is attached to exactly two sides, which make an angle $\theta(q)$. If $\theta(q)$ is measured counterclockwise, call the side it is measured from the \textit{clockwise side} attached to $q$ and the side it is measured to the \textit{counterclockwise side} attached to $q$. Note that clockwise sides are always glued to counterclockwise sides. Relabel the vertices and sides as follows: choose a vertex $q_1$, and let the clockwise side attached to it be $s_1$ and the counterclockwise side attached to it be $s_2$. Then for each side $2\leq k\leq m$, let the side glued to $s_k$ be $s_k'$, the side immediately counterclockwise of $s_k'$ be $s_{k+1}$, and the vertex between $s_k'$ and $s_{k+1}$ be $q_k$. Notice that for $1\leq k\leq m-1$, the angle between $s_k'$ and $s_{k+1}$ is $\theta(q_k)$. Moreover, $s_1$ is the side lying immediately counterclockwise of $s_m$, and the angle between them is $\theta(q_m)$. Now $\theta(p)$ is the sum of all the $\theta(q_k)$ and is also the angle (modulo a full turn) between $s_1$ and itself. But the angle between $s_1$ and itself is $0$, so $\theta(p)$ must be a whole number of turns, proving (2).
\end{proof}

The above theorem shows that cone points and their angles form part of the intrinsic geometry of a translation surface.\footnote{The space of all translation surfaces naturally splits into distinct \textit{strata} based on the number of cone points and their respective angles. Strata are fundamental to the study of translation surfaces, as they are the natural domains over which many results apply. We direct interested readers to \cite{zorich06, hubert-schmidt06, wright15, athreya-masur-24}.} Outside of the cone points, we can obtain a notion of distance on translation surfaces by inheriting a flat metric from $\mathbb{R}^2$. This allows us to consider another feature of the intrinsic geometry of translation surfaces: \textit{geodesics}. A geodesic is a path which locally minimizes distance; an observer moving on a translation surface at constant velocity traces out a geodesic. Since translation surfaces are flat everywhere except at their cone points, a constant-velocity observer on a translation surface who does not encounter a cone point traces out a straight line. It follows that straight lines which do not intersect cone points are geodesics on translation surfaces. Another kind of geodesic begins and ends at (not necessarily distinct) cone points and these are called \textit{saddle connections}. Here is a straightforward way of defining saddle connections.

\begin{defn}
Let $(X,\omega)$ be a translation surface. If $p$ and $q$ are (not necessarily distinct) cone points of $(X,\omega)$ and $\gamma$ is a nontrivial straight-line path connecting $p$ and $q$ and containing no other cone points, we say $\gamma$ is a \textit{saddle connection} on $(X,\omega)$.
\end{defn}

\begin{figure}
    \centering
    \begin{tikzpicture}[scale=1.5,thick]
        \draw[blue] ({-sqrt(3)/2},{-1/2}) -- ({3*sqrt(3)/2},{1/2});
        \draw[blue,dashed] (0,-1) -- ({-sqrt(3)/2},{1/2});
        \draw[blue,dotted] ({sqrt(3)/2},{-1/2}) -- ({3*sqrt(3)/2},{-1/2});
        
        \foreach \n in {0,1,2,3,4,5} {
            \begin{scope}[on background layer]
                \draw ({cos(60*\n+30)},{sin(60*\n+30)}) -- ({cos(60*\n+90)},{sin(60*\n+90)});
                \draw ({sqrt(3)},0) +({-cos(60*\n+30)},{sin(60*\n+30)}) -- +({-cos(60*\n+90)},{sin(60*\n+90)});
            \end{scope}

            \ifodd\n
                \node[bullet] at ({cos(60*\n+30)},{sin(60*\n+30)}) {};
                \node[bullet] at ({sqrt(3)+cos(60*\n+30)},{sin(60*\n+30)}) {};
            \else
                \node[circ] at ({cos(60*\n+30)},{sin(60*\n+30)}) {};
                \node[circ] at ({sqrt(3)+cos(60*\n+30)},{sin(60*\n+30)}) {};
            \fi
        }
        
        \pic[rotate={150}] at ({sqrt(3)/4},{3/4}) {id};
        \pic[rotate={150}] at ({3*sqrt(3)/4},{-3/4}) {id};
        \pic[rotate={210}] at ({-sqrt(3)/4},{3/4}) {iid};
        \pic[rotate={210}] at ({5*sqrt(3)/4},{-3/4}) {iid};
        \pic[rotate={270}] at ({-sqrt(3)/2},0) {iiid};
        \pic[rotate={270}] at ({3*sqrt(3)/2},0) {iiid};
        \pic[rotate={150}] at ({-sqrt(3)/4},{-3/4}) {i};
        \pic[rotate={150}] at ({5*sqrt(3)/4},{3/4}) {i};
        \pic[rotate={210}] at ({sqrt(3)/4},{-3/4}) {ii};
        \pic[rotate={210}] at ({3*sqrt(3)/4},{3/4}) {ii};
    \end{tikzpicture}
    \caption{Saddle connections on the double hexagon. Glued edges are annotated with the same symbol.}
    \label{fig:double-hexagon}
\end{figure}
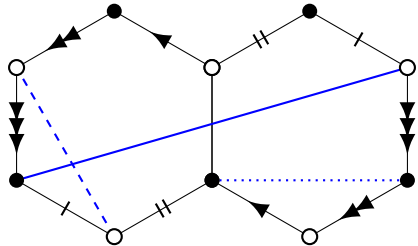

For an example of a saddle connection, consider the translation surface formed by gluing two regular hexagons in Figure~\ref{fig:double-hexagon}. This surface has two cone points $\bullet$ and $\circ$, each of which have an angle of two turns. The three blue paths (solid, dashed, and dotted) each start and end at a cone point and are nontrivial straight lines containing no other cone points, so all three are saddle connections. 

The \textit{holonomy vector} of a saddle connection is the vector whose components record how far the connection travels horizontally and vertically. For example, if the regular hexagons in Figure~\ref{fig:double-hexagon} have circumradius $2$, then the holonomy vectors associated to the solid, dashed, and dotted saddle connections (assuming each is oriented left-to-right) are $\ang{4\sqrt{3},2}$, $\ang{\sqrt{3},-3}$, and $\ang{2\sqrt{3},0}$, respectively. We denote holonomy vector components in angle brackets: $\ang{n,k}$.\footnote{Although this notation for a vector is non-standard, we will find it convenient to be able to visually distinguish holonomy vectors from coordinates of points in a different plane, which we define later in \S\,\ref{ch:algorithm}.} Although holonomy vectors and saddle connections are formally different concepts, we identify them and use the terms interchangeably. This paper studies the distribution of holonomy vector directions on a subset of translation surfaces called square-tiled surfaces.

\subsection{Square-tiled surfaces}

\begin{defn}
A \textit{square-tiled surface} or \textit{origami} is a translation surface which can be constructed from a collection of congruent squares (\textit{tiles}), all of whose sides are either parallel or perpendicular.\footnote{Though this definition may seem somewhat restrictive, the class of square-tiled surfaces displays many interesting properties, and in fact they are known to be dense in every stratum of translation surfaces under the appropriate metric \cite{hubert-schmidt06}.}
\end{defn}
 
Every square-tiled surface is $\GL[2]{\mathbb{R}}$-equivalent to a surface that can be represented as a collection of unit squares with vertices in $\mathbb{Z}^2$. Moreover, we can assume that such a surface is \textit{reduced}, meaning that the additive closure of its set of holonomy vectors (sometimes called the \textit{lattice of periods}) is equal to $\mathbb{Z}^2$. If this is not the case, then it is possible to scale the surface by an element of $\GL[2]{\mathbb{R}}$ to obtain another square-tiled surface with a smaller number of squares, as in Figure~\ref{fig:reduced-sts}. We will henceforth assume that all square-tiled surfaces are reduced and composed of unit squares with vertices in $\mathbb{Z}^2$.\footnote{As it turns out, all translation surfaces in an $\SL[2]{\mathbb{R}}$ orbit share the same slope gap distribution, so scaling a translation surface affects the slope gap distribution in a deterministic, well-understood way \cite{uyanik-work16}.}

Every square-tiled surface with $n$ tiles can be represented by an element of $\mathfrak{S}_n\times\mathfrak{S}_n$, where $\mathfrak{S}_n$ is the symmetric group on $n$ elements, by the following procedure.

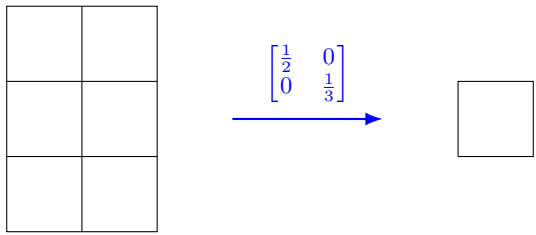
\begin{figure}
    \centering
    \begin{tikzpicture}
        \draw (0,0)--(2,0)--(2,3)--(0,3)--(0,0);
        \draw (1,0)--(1,3);
        \draw (0,1)--(2,1);
        \draw (0,2)--(2,2);
        \draw[-Latex,thick,blue] (3,1.5)--(5,1.5);
        \node[above,blue] at (4,1.6) {\small $\begin{bmatrix}
            \frac{1}{2}&0\\0&\frac{1}{3}
        \end{bmatrix}$};
        \draw (6,1)--(7,1)--(7,2)--(6,2)--(6,1);
    \end{tikzpicture}
    \caption{The unreduced $6$-square-tiled surface on the left can be scaled by an element of $\GL[2]{\mathbb{R}}$ to produce the reduced one-square-tiled surface (the square torus) on the right. In both surfaces, opposite sides are identified.}
    \label{fig:reduced-sts}
\end{figure}

\begin{defn}
Let $(X,\omega)$ be a square-tiled surface with $n$ tiles. Label the tiles of $(X,\omega)$ with the natural numbers $1$ to $n$. Then any element $\sigma\in\mathfrak{S}_n$ defines a bijective mapping from $(X,\omega)$ to itself obtained by mapping the tile with label $n$ to the tile with label $\sigma(n)$.
\begin{enumerate}
    \item The \textit{right} permutation $\sigma_R$ of $(X,\omega)$ is the permutation which induces the mapping that sends each tile to the tile that lies to its right.
    \item The \textit{up} permutation $\sigma_U$ of $(X,\omega)$ is the permutation which induces the mapping that sends each tile to the tile above it.
\end{enumerate}
The \textit{left} and \textit{down} permutations of $(X,\omega)$, $\sigma_L$ and $\sigma_D$, are defined analogously.
\end{defn}

\begin{figure}
    \centering
    \begin{tikzpicture}
        \edef\lab{0}
        \foreach \location in {(-1,1),(0,1),(0,0),(1,0),(2,0)} {
            \pgfmathparse{int(\lab+1)}
            \xdef\lab{\pgfmathresult}
            \draw \location rectangle +(1,1) node[midway] {$\lab$};
        }
        \pic at (-.5,2) {i};
        \pic at (1.5,0) {i};
        \pic at (-.5,1) {ii};
        \pic at (1.5,1) {ii};
    \end{tikzpicture}
    \caption{A $5$-tile origami. Unmarked opposite edges are identified, as are edges with the same marking.}
    \label{fig:permutation-representation}
\end{figure}
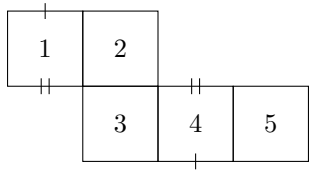

It is straightforward to verify that for any square-tiled surface, $\sigma_L=\sigma_R^{-1}$ and $\sigma_D=\sigma_U^{-1}$. Up to conjugacy (equivalently, relabeling the tiles), every square-tiled surface is uniquely identified by its right and up permutations, which depend only on the intrinsic geometry of the surface. As an example, the $5$-tile origami in Figure~\ref{fig:permutation-representation}, under the labeling given there, has the permutation representation
\[\sigma_R=(1\ 2)(3\ 4\ 5)\,,\quad\sigma_U=(1\ 4)(2\ 3)\,.\]
The permutation representation of any square-tiled surface can be easily calculated by imagining an observer traversing the surface. For instance, if the observer starts on tile~$1$ and moves rightward on the origami in Figure~\ref{fig:permutation-representation}, it next reaches tile~$2$ before returning to tile~$1$, so one cycle in $\sigma_R$ is $(1\ 2)$. If the observer instead moves rightward on tile~$3$, it passes first through tile~$4$ and then through tile~$5$ before returning to tile~$3$. This shows that the other cycle in $\sigma_R$ is $(3\ 4\ 5)$. A similar kind of ``trajectory tracing'' for an upward-moving observer produces $\sigma_U$.

\subsubsection{The square torus}

The simplest square-tiled surface is the one comprising a single tile with opposite edges identified. This surface is called the \textit{square torus}, written $\mathbb{T}^2$, and is illustrated in Figure~\ref{fig:square-torus}.

Since it contains only one tile, the square torus has a permutation representation $\sigma_R=\sigma_U=1$, where $1$ represents the trivial permutation. The square torus has no cone points; its single apparent vertex has an angle of exactly one full turn. For the sake of exposition, we improperly refer to closed geodesics on the torus as ``saddle connections'' on the torus, since these can be thought of as paths connecting a ``cone point'' of the torus to itself.

If we embed the one-square polygonal representation of the square torus in $\mathbb{R}^2$ so that its lower-left corner is at $(0,0)$, then the equivalence relation $\sim$ obtained from gluing its opposite edges identifies $(0,y)$ with $(1,y)$ and $(x,0)$ with $(x,1)$ for all $x$, $y\in[0,1]$. It is straightforward to check that $\mathbb{R}^2\mathbin{/}{\sim}=\mathbb{R}^2\mathbin{/}\mathbb{Z}^2$ (in other words, the quotient obtained from $\sim$ can be given an algebraic structure). We use this algebraic structure to characterize the Veech group of the square torus, $\SL{\mathbb{T}^2}$.

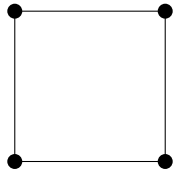
\begin{figure}
    \centering
    \begin{tikzpicture}[scale=2]
        \draw (0,0) rectangle +(1,1);
        \foreach \x in {0,1} {
            \foreach \y in {0,1} {
                \node[bullet] at (\x,\y) {};
            }
        }
    \end{tikzpicture}
    \caption{The square torus, the simplest example of a square-tiled surface.}
    \label{fig:square-torus}
\end{figure}

Let $L\in\SL{\mathbb{T}^2}$ be a linear transformation that fixes the square torus. Since $(1,0)\sim(0,0)$ and $(0,1)\sim(0,0)$, we must have $L((1,0))\sim L((0,0))=(0,0)$, and similarly for $L((0,1))$. But the set of points equivalent to $(0,0)$ is $\mathbb{Z}^2$, so we conclude $L((1,0))$, $L((0,1))\in\mathbb{Z}^2$. It follows that any linear transformation in the Veech group of the square torus is represented by a matrix with all integer entries, that is, that $\SL{\mathbb{T}^2}\leq\SL[2]{\mathbb{Z}}$.

Recall that the matrices $S$ and $T$, representing respectively a counterclockwise rotation through a quarter turn and a horizontal shear by one, generate $\SL[2]{\mathbb{Z}}$. The matrix $S$ maps the unit square to itself (up to translation) and preserves opposite side gluings, so it fixes $\mathbb{T}^2$. We showed in \S\ \ref{subsec:linear-transformations} that a horizontal shear by $1$ preserves $\mathbb{T}^2$; it follows that $\SL[2]{\mathbb{Z}}=\SL{\mathbb{T}^2}$. In other words, the Veech group of the square torus is precisely the modular group. We use this insight in the next section to describe the Veech groups of square-tiled surfaces.

\subsubsection{The modular group and square-tiled surfaces}
\label{subsec:modular-group-origamis}

Observe that an element of $\SL[2]{\mathbb{R}}$ that stabilizes a translation surface must also preserve its set of holonomy vectors and consequently its lattice of periods. Since we assume that all origamis are reduced, this means that any linear transformation that stabilizes an origami must stabilize the lattice $\mathbb{Z}^2$. 
It follows that if $(X,\omega)$ is an origami, its Veech group $\SL{X,\omega}$ must be a subgroup of $\SL[2]{\mathbb{Z}}$. This fact and the fact that every square-tiled surface is a Veech surface \citep{hubert-schmidt06} encourage studying the structure of origami Veech groups through hyperbolic geometry, since actions by subgroups of the modular group on $\mathbb{H}^2$ are particularly well-understood.

\subsubsection{Cusps}
\label{subsec:cusps}

Notice that in $\overline{\mathbb{H}^2}$, points on the real axis and the point $\infty$ are infinitely distant from all points in $\mathbb{H}^2$. For this reason, we refer to the real axis as the \textit{hyperbolic horizon} and the real axis and $\infty$ together as the \textit{points at infinity}. Moreover, the form of the volume element allows visually infinite regions of $\mathbb{H}^2$ to be of finite volume. With these concepts in mind we make the following definition.

\begin{defn}
A \textit{cusp region} is a connected open finite-volume subset of the hyperbolic plane containing exactly one point at infinity. The point at infinity contained in a cusp region is its \textit{cusp}.
\end{defn}

Now let $(X,\omega)$ be a square-tiled surface and consider the quotient $\mathbb{H}^2\mathbin{/}\SL{X,\omega}$. Since that the Veech group of any square-tiled surface is a subgroup of the modular group $\SL[2]{\mathbb{Z}}$, any fundamental domain of $\SL{X,\omega}$ can be represented as a union of fundamental domains of $\SL[2]{\mathbb{Z}}$. The fundamental domain of $\SL{X,\omega}$ has finite hyperbolic volume since $(X,\omega)$ is a square-tiled surface and all square-tiled surfaces are Veech surfaces, so this union must be finite.\footnote{Following this line of reasoning, one can show that $\SL{X,\omega}$ is of finite index in $\SL[2]{\mathbb{R}}$.} Moreover, the fundamental domain must contain a finite number of points at infinity. It follows that such a fundamental domain has the same closure as the union of a finite number of cusp regions.

It turns out there is a correspondence between cusps in the fundamental domain of a Veech group and a certain class of subgroups of the Veech group called \textit{maximal parabolic subgroups}. To build up this connection we define the following notions.

\begin{defn}
An element $A$ of $\SL[2]{\mathbb{Z}}$ is \textit{parabolic} if $\tr{A}=\pm 2$. A subgroup $\Gamma$ of $\SL[2]{\mathbb{Z}}$ all of whose elements are parabolic is a \textit{parabolic subgroup} of $\SL[2]{\mathbb{Z}}$. If $\Gamma$ is not a proper subgroup of any parabolic subgroup of $\SL[2]{\mathbb{Z}}$, it is a \textit{maximal parabolic subgroup} of $\SL[2]{\mathbb{Z}}$.
\end{defn}

Let $A$ be a parabolic element of $\SL[2]{\mathbb{Z}}$ and $a$, $b$ be its upper-left and upper-right entries. Since $\tr{A}=\pm2$, the lower-right entry of $A$ must be $\pm2-a$. The lower-left entry $c$ of $A$ must then satisfy $-(1\mp a)^2=bc$ since $\det{A}=1$. Now recall that
\[A=\begin{bmatrix}
    a & b \\ c & \pm2-a
\end{bmatrix}\]
induces the linear-fractional transformation
\[A\cdot z=\frac{az+b}{cz-a\pm2}\]
on $\mathbb{H}^2$. If $c\neq0$, this transformation has the single real fixed point
\[z=\frac{a\mp1}{c}\,,\]
which lies on the hyperbolic horizon. If $c=0$, this transformation fixes $\infty$ (using the usual definition of linear-fractional transformations on the compactified complex plane). In other words, each parabolic element fixes precisely one point at infinity. It is known that the maximal parabolic subgroup $\Gamma\leq\SL{X,\omega}$ containing $A$ is isomorphic to $\mathbb{Z}\oplus(\mathbb{Z}\mathbin{/}2\mathbb{Z})$ if $-I\in\SL{X,\omega}$, where $-I$ generates the finite cyclic factor; if $-I\notin\SL{X,\omega}$, then $\Gamma$ is isomorphic to $\mathbb{Z}$ (see \cite{uyanik-work16,kumanduri-et-al24}). In either case, it is straightforward to see that the fixed point of the element that generates the infinite cyclic factor is also the unique fixed point of $\Gamma$. The above construction allows us to identify a maximal parabolic subgroup $\Gamma$ with its unique fixed point (cusp) in $\mathbb{H}^2$.

Let $\Gamma$ be a maximal parabolic subgroup of $\SL{X,\omega}$ and let its fixed point at infinity be $a$. If $a\in\mathbb{R}$, it is always possible to conjugate $\Gamma$ in $\SL[2]{\mathbb{Z}}$ by some element $P$ whose associated linear-fractional transformation maps $\infty$ to $a$. The resulting parabolic elements fix $\infty$ and are of the form
\[T^n=\begin{bmatrix}
    1 & n \\ 0 & 1
\end{bmatrix}\]
for $n\in\mathbb{Z}$. The integer $n$ is called the \textit{width} of the cusp that corresponds to $\Gamma$. Geometrically, the width of a cusp is the number of closures of images of the fundamental domain of $\SL[2]{\mathbb{Z}}$ required to cover an associated cusp region.

In the quotiented hyperbolic plane $\mathbb{H}^2\mathbin{/}\SL{X,\omega}$, each cusp has infinitely many representatives, so we must identify a conjugacy class of maximal parabolic subgroups (where conjugation is in $\SL{X,\omega}$) with each cusp. Informally, conjugation in $\SL{X,\omega}$ ``changes'' the fixed point of the subgroup to another representative of the cusp in the hyperbolic plane, performs a translation $T^n$, then ``changes'' the fixed point back to the original representative.

As an example, consider the square-tiled surface with permutation representation
\[\sigma_R=(1\ 2)(3\ 4)\,,\quad\sigma_U=(2\ 3)\,.\]
This surface is drawn in Figure~\ref{fig:four-tile-surface}.
\begin{figure}
    \centering
    \begin{tikzpicture}
        \foreach \LL in {(0,0),(1,0),(1,1),(2,1)}{
            \draw \LL rectangle +(1,1);
        }
        \node at (0.5,0.5) {$1$};
        \node at (1.5,0.5) {$2$};
        \node at (1.5,1.5) {$3$};
        \node at (2.5,1.5) {$4$};
    \end{tikzpicture}
    \caption{A four-tile origami with opposite edges identified}
    \label{fig:four-tile-surface}
\end{figure}
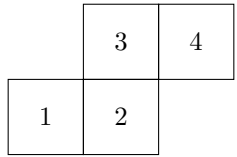
Using techniques discussed in Section~\ref{ch:algorithm} we can calculate that this surface's Veech group has maximal parabolic subgroups
\[\Gamma_\infty=\ang{T^2}\,,\,\Gamma_0=\ang{ST^2S}\,,\,\Gamma_1\ang{(STS)T^2(STS)^{-1}}\,.\]
The fundamental domain of $\Gamma$ is shown in Figure~\ref{fig:cusps} along with its cusps. The subgroup $\Gamma$ has three cusps which can be represented by the points $0$, $1$, and $\infty$. Let $\mathcal{F}$ be the fundamental domain of $\SL[2]{\mathbb{Z}}$ that contains $\infty$ and is symmetric about the imaginary axis. Then the fundamental domain of $\Gamma$ can be decomposed into three cusp regions, $(S\cdot \mathcal{F})\cup (ST\cdot \mathcal{F})$ with cusp $0$, $(TS\cdot\mathcal{F})\cup (TST^{-1}\cdot\mathcal{F})$ with cusp $1$, and $\mathcal{F}\cup (T\cdot\mathcal{F})$ with cusp $\infty$. In this example all three cusps are of width $2$.

\begin{figure}
    \centering
    \begin{tikzpicture}[scale=3]
        \begin{scope}
            \clip (-2,0) rectangle (2,1.5);
            
            \begin{scope}[gray,help lines]
                \foreach \x in {-2,-1,0,1,2}{
                    \draw ({\x+1/2},0) -- +(0,2.5);
                    \draw ({\x-1},0) arc[radius=1,start angle=180,delta angle=-180];
                    \draw ({\x},0) arc[start angle=0,delta angle=180,radius={1/3}];
                    \draw ({\x},0) arc[start angle=180,delta angle=-180,radius={1/3}];
                }
            \end{scope}
            
            \draw[thick,fill=green,fill opacity=.1] (-0.5,2.5) -- (-0.5,{sqrt(3)/6})
                        arc[start angle=120,end angle=0,radius={1/3}]
                        arc[start angle=180,end angle=120,radius=1]
                        arc[start angle=60,end angle=0,radius=1]
                        arc[start angle=180,end angle=60,radius={1/3}]
                        -- (1.5,2.5);
    
            \begin{scope}[font=\footnotesize]
                \node at (0,1.2) {$I$};
                \node at (0,.7) {$S$};
                \node[font=\tiny] at (-.32,.5) {$ST$};
                \node at (1,1.2) {$T$};
                \node at (1,.7) {$TS$};
                \node[font=\tiny] at (1.32,.5) {$TST^{-1}$};
            \end{scope}
        \end{scope}

        \draw (-2,0) -- (2,0);

        \begin{scope}[font=\footnotesize]
            \node[bullet,label={[below,yshift=-1ex]$0$}] at (0,0) {};
            \node[circ,label={[below,yshift=-1ex]$1$}] at (1,0) {};
            \draw[-Latex] (.5,1.4) -- +(0,.2) node[yshift=1ex] {$\infty$};
        \end{scope}
    \end{tikzpicture}
    \caption{Cusps of the Veech group of the origami in Figure~\ref{fig:four-tile-surface}. Each area is labeled by the transformation that maps $\mathcal{F}$ onto it.}
    \label{fig:cusps}
\end{figure}
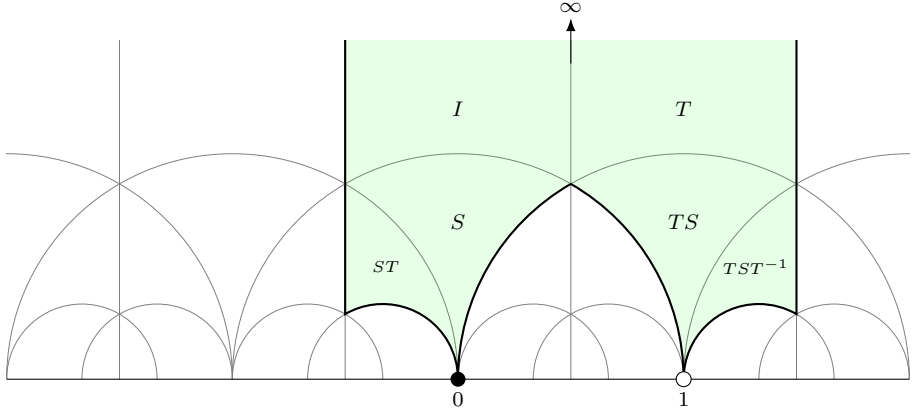

\subsection{Slope gap distributions}
\label{sec:slope-gap-dists}

Recall that a saddle connection on a general translation surface $(X,\omega)$ is a geodesic path connecting two not-necessarily-distinct cone points and the holonomy vector associated to it is the vector in $\mathbb{R}^2$ whose components record how far the connection travels horizontally and vertically. 
It is known that the directions of holonomy vectors on almost every translation surface equidistribute on the circle $S^1$ (\cite{vorobets05}; for the precise sense of ``almost every'' see \cite{kumanduri-et-al24}), and for a Veech surface they always equidistribute (see \cite{veech98}). That is, if $\Lambda(X,\omega)$ denotes the set of holonomy vectors of a translation surface $(X,\omega)$, then for any interval $I\subseteq S^1$, the proportion of vectors in $\Lambda(X,\omega)$ of length at most $R$ with direction in $I$ converges to the angle subtended by $I$ as $R\to\infty$.

Although a na{\"i}ve interpretation of this result suggests that the distribution of holonomy vectors on almost every translation surface is ``truly random'', subtler measures of distribution reveal that this is not the case. In particular, the \textit{gaps} between slopes of holonomy vectors are not distributed as one might expect. We now define some preliminary concepts in order to articulate the concept of holonomy vector slope gaps more precisely.

\begin{defn}
Let $(X,\omega)$ be a translation surface and let $\Lambda(X,\omega)$ be its set of holonomy vectors. If $R$ is a positive real number, the increasing, nonrepeating sequence of $N(R)$ elements
\[\mathcal{S}_R(X,\omega)=\set{\frac{y}{x}\,\middle\vert\,\ang{x,y}\in\Lambda(X,\omega)\,,\,0<x\,,\,0\leq y\leq x\leq R}=\set{s_i}_{i=1}^{N(R)}\]
is the \textit{$R$-slope sequence} of $(X,\omega)$. Furthermore, the sequence
\[\mathcal{G}_R(X,\omega)=\set{R^2(s_{i+1}-s_i)\mid 1\leq i\leq N(R)-1}\]
is the \textit{renormalized $R$-gap sequence} of $(X,\omega)$.
\end{defn}

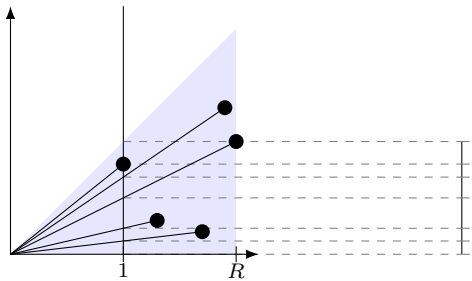
\begin{figure}
    \centering
    \begin{tikzpicture}[font=\footnotesize,scale=1.5]
        \draw[Latex-Latex] (2.2,0) -- (0,0) -- (0,2.2);
        \draw (1,0) pic {i} node[below] {$1$};
        \draw (2,0) pic {i} node[below] {$R$};
        \draw (1,0) -- +(0,2.2);

        \foreach \location in {(2,1),(1,.8),(1.7,.2),(1.3,.3),(1.9,1.3)} {
            \draw (0,0) -- \location node[bullet] {};
        }

        \begin{scope}[on background layer]
            \fill[blue!10] (0,0) -- (2,2) -- (2,0);
            \draw (4,0) -- +(0,1);
            \foreach \y in {0,.5,.8,{.2/1.7},{.3/1.3},{1.3/1.9},1}{
                \draw[dashed,gray] (1,\y) -- (4,\y) pic[rotate=90] {i};
            }
        \end{scope}
    \end{tikzpicture}
    \caption{Holonomy vectors of a translation surface represented in its $R$-slope sequence. The $R$-gap sequence (before renormalization) is represented by segments of the line $x=1$.}
    \label{fig:slope-sequence}
\end{figure}

As defined above, $\mathcal{S}_R(X,\omega)$ is the increasing sequence of unique slopes in $[0,1]$ of holonomy vectors on $(X,\omega)$ with maximum component at most $R$. Notice that the additional requirements on the holonomy vectors stipulate that they lie in a triangular domain shown in Figure~\ref{fig:slope-sequence}. The tips of all holonomy vectors must lie in a discrete set defined by the finite number of cone points of the translation surface \cite{hubert-schmidt06}, so the compactness of the triangular domain guarantees that the $R$-slope sequence is finite. The idea is to examine the limiting distribution of the slopes of \textit{all} holonomy vectors as $R\to\infty$.

We are interested in the spacing of holonomy vector slopes on a translation surface, so we subtract adjacent slopes in the $R$-slope sequence to obtain a gap sequence. In general, as $R\to\infty$, the number of unique holonomy vector slopes $N(R)$ grows quadratically \citep{masur88}. In order to obtain a limiting distribution, the gaps must be renormalized (they would otherwise tend to zero). Since $N(R)$ grows quadratically, a reasonable heuristic renormalization factor is $R^2$. This factor defines the renormalized $R$-gap sequence $\mathcal{G}_R(X,\omega)$. Geometrically, the renormalized $R$-gap sequence can be understood as the renormalized lengths of segments delimited by the projections of holonomy vectors of maximum component $R$ onto the line $x=1$; see Figure~\ref{fig:slope-sequence}.

At this point we are ready to define the limiting distribution we are interested in studying.

\begin{defn}
If there exists a limiting probability density function $f$ such that
\[\lim_{R\to\infty}\frac{\abs{\mathcal{G}_R(X,\omega)\cap(a,b)}}{N(R)}=\int_a^bf(t)dt\,,\]
we say $f$ is the \textit{slope gap distribution} of $(X,\omega)$.
\end{defn}

In other words, the slope gap distribution records the probability that a renormalized holonomy vector slope falls in any interval $(a,b)$. It follows from results of \cite{athreya-chaika12} that such a distribution always exists for a Veech surface.

\subsection{The balanced 10-tile origami}
\label{sec:ten-tile-surface}

In Section~\ref{ch:calculations} we calculate the slope gap distribution for the balanced $10$-tile origami shown in Figure~\ref{fig:ten-tile-origami}. The surface has two cone points each with an angle of two turns.

\begin{figure}
    \centering
    \begin{tikzpicture}
        \edef\lab{0}
        \foreach \x in {1,2,3,4,5,-4,-5,-6,-7,-8} {
            \pgfmathparse{int(\lab+1)}
            \xdef\lab{\pgfmathresult}
            \ifnum\x>0
                \draw (\x,0) rectangle +(1,1) node[midway] {$\lab$};
            \else
                \draw ({-\x},1) rectangle +(1,1) node[midway] {$\lab$};
            \fi
        }
        \foreach \location in {(1,0),(1,1),(6,0),(6,1),(6,2)} {
            \node[bullet] at \location {};
        }
        \foreach \location in {(4,0),(4,1),(4,2),(9,1),(9,2)} {
            \node[circ] at \location {};
        }
    \end{tikzpicture}
    \caption{The balanced $10$-tile origami.}
    \label{fig:ten-tile-origami}
\end{figure}
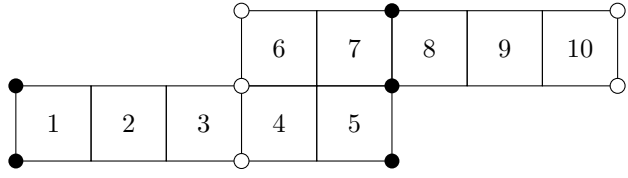

The balanced $10$-tile origami is interesting for a number of reasons. Its relatively high number of tiles gives it a complicated slope gap distribution, but its high degree of symmetry lends it a large Veech group, so that computations are manageable. In particular, the slope gap distribution for every square-tiled surface calculated to date has been a sum of scaled Hall distributions. We show in this paper that the slope gap distribution of the balanced $10$-tile origami cannot be reduced to such a sum.

%% file: ksw-algorithm.tex
\section{The slope gap distribution algorithm}
\label{ch:algorithm}

\subsection{Overview}

We use a slightly adapted version of the algorithm developed in \cite{kumanduri-et-al24} (hereafter the modified KSW algorithm) to calculate slope gap distributions for square-tiled surfaces. The algorithm is as follows:
\begin{alg}[The modified KSW algorithm]\hfill
\begin{enumerate}
    \item Choose a Veech surface $(X,\omega)$, and let the number of its cusps be $n$.
    \item Find the generators $P_i$, $1\leq i\leq n$ of the conjugacy classes of the maximal parabolic subgroups $\Gamma_i$ of $\SL{X,\omega}$. Do this by creating a directed graph of the orbit of $(X,\omega)$ under the action of $\SL[2]{\mathbb{Z}}$. Every path from $(X,\omega)$ to a cusp representative and back corresponds to a parabolic generator $P_i=C_iT^{\alpha_i}C_i^{-1}$, where $\alpha_i$ is the cusp width and $C_i$ is a word in $S$ and $T$.
    \item To each parabolic generator $P_i$ there corresponds a connected component of the Poincar{\'e} section under the horocycle flow, $\Omega_i$. Each $\Omega_i$ can be parametrized as a triangle in the plane (Equation~\ref{eq:poincare-parametrization}).
    \item For each point in each of the $\Omega_i$, determine which short holonomy vector has the least slope. Doing so partitions each $\Omega_i$ into a finite collection of convex polygons.
    \item On each polygonal region, calculate the area swept out by the level sets of the return time function.
    \item Sum the swept-areas of each polygonal region in each $\Omega_i$ to obtain the unnormalized CDF of the slope gap distribution.
    \item Differentiate and normalize the unnormalized CDF to obtain the pdf of the slope gap distribution on $(X,\omega)$.
\end{enumerate}
\end{alg}

The following sections explain the process in more detail. Throughout the section, we refer to the calculation of the slope gap distribution for the three-tile origami in Figure~\ref{fig:three-tile-STS} as a concrete example of the process.

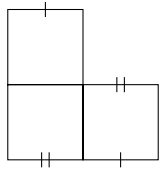
\begin{figure}[b]
    \centering
    \begin{tikzpicture}[scale=1]

    
        \draw (-1,0) rectangle +(1,1) ++(1,0) rectangle ++(1,1) ++(-2,0) rectangle ++(1,1);
        
            
        \draw (.5,.1) -- (.5,-.1)
            (-.5,2.1) -- (-.5,1.9);

        \draw (.45,1.1) -- (.45,.9)
            (.55,1.1) -- (.55,.9)
            (-.55,.1) -- (-.55,-.1)
            (-.45,.1) -- (-.45,-.1);
        
    \end{tikzpicture}
    \caption{The three-tile origami.}
    \label{fig:three-tile-STS}
\end{figure}

\subsection{The horocycle flow}

Let $(X,\omega)$ be a translation surface and $\Lambda(X,\omega)$ be the set of its holonomy vectors. When the translation surface under discussion is clear from context we abbreviate $\Lambda(X,\omega)$ to just $\Lambda$. Recall that the Veech group of $(X,\omega)$, written $\SL{X,\omega}$, is the stabilizer of $(X,\omega)$ in $\SL[2]{\mathbb{R}}$.

Define the \textit{horocycle flow} to be the action of the one-parameter family of matrices
\[h_s=\begin{bmatrix}
    \phantom{-}1 & 0 \\ -s & 1
\end{bmatrix}\]
on the space $\SL[2]{\mathbb{R}}\mathbin{/}\SL{X,\omega}$ by left multiplication. The action of the horocycle flow extends naturally to $\Lambda(X,\omega)$. Geometrically, the horocycle flow shears the plane downwards by $s$. For any vector $\mathbf{v}\in\mathbb{R}^2$, it is straightforward to verify that
\[m(h_s\mathbf{v})=m(\mathbf{v})-s\,,\]
where $m\colon\mathbb{R}^2\to\mathbb{R}$, given by
\[m(\mathbf{v})=\frac{\mathbf{v}\cdot\ang{0,1}}{\mathbf{v}\cdot\ang{1,0}}\,,\]
maps a vector in $\mathbb{R}^2$ to its slope. It follows that the difference in slopes between any two vectors $\mathbf{u}$, $\mathbf{v}\in\mathbb{R}^2$ is preserved under the action of $h_s$: 
\[m(h_s\mathbf{u})-m(h_s\mathbf{v})=m(\mathbf{u})-s-m(\mathbf{v})+s=m(\mathbf{u})-m(\mathbf{v})\,.\]

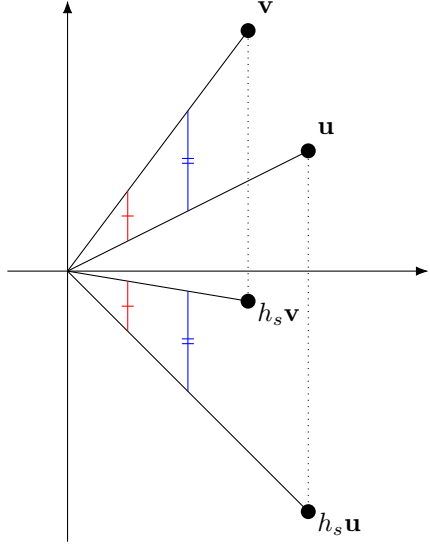
\begin{figure}
    \centering
    \begin{tikzpicture}[scale=.8]
        \draw[-Latex] (-1,0) -- (6,0);
        \draw[-Latex] (0,-4.5) -- (0,4.5);
        
        \draw[-] (0,0) -- (4,2) node[name=u,bullet,label={[above right]$\mathbf{u}$}] {};
        \draw[-] (0,0) -- (3,4) node[name=v,bullet,label={[above right]$\mathbf{v}$}] {};
        
        \draw[-] (0,0) -- (4,-4) node[name=u',bullet,label={[below right]$h_s\mathbf{u}$}] {};
        \draw[-] (0,0) -- (3,-.5) node[name=v',bullet,label={[below right]$h_s\mathbf{v}$}] {};

        \draw[dotted] (u) -- (u');
        \draw[dotted] (v) -- (v');
        
        \draw[red] (1,{-1/6}) -- (1,-1);
        \draw[red] (1,{4/3}) -- (1,{1/2});
        \draw[red] (.9,{-7/12}) -- +(.2,0);
        \draw[red] (.9,{11/12}) -- +(.2,0);

        \draw[blue] (2,{-1/3}) -- (2,-2);
        \draw[blue] (2,{8/3}) -- (2,1);
        \draw[blue] (1.9,{-7/6+.04}) -- +(.2,0);
        \draw[blue] (1.9,{11/6+.04}) -- +(.2,0);
        \draw[blue] (1.9,{-7/6-.04}) -- +(.2,0);
        \draw[blue] (1.9,{11/6-.04}) -- +(.2,0);
        
    \end{tikzpicture}
    \caption{The horocycle flow preserves slope gaps between holonomy vectors. Geometrically, the vertical distance between two holonomy vectors at any given value of $x$ does not change.}
    \label{fig:horocycle-slope-gaps}
\end{figure}

In general, both $\SL[2]{\mathbb{R}}$ and its quotient $\SL[2]{\mathbb{R}}\mathbin{/}\SL{X,\omega}$ are three-dimensional spaces (four real dimensions for each matrix entry with one constraint that the determinant be one). The horocycle flow together with $\SL[2]{\mathbb{R}}\mathbin{/}\SL{X,\omega}$ thus forms a three-dimensional dynamical system. Let $\Omega(X,\omega)$ be the set of surfaces in the $\SL[2]{\mathbb{R}}$-orbit of $(X,\omega)$ which have a horizontal saddle connection of length at most one. That is, let
\begin{equation}\label{eq:poincar'e-section}
\Omega(X,\omega)=\{g\in\SL[2]{\mathbb{R}}\mathbin{/}\SL{X,\omega}\mid g\Lambda\cap((0,1]\times\{0\})\neq\varnothing\}\,.
\end{equation}
We abbreviate $\Omega(X,\omega)$ to $\Omega$ when the Veech surface under discussion is clear from context. We show in the next section that $\Omega$ is a two-dimensional space. It can also be shown that almost every orbit under the horocycle flow intersects $\Omega$ transversely in $\SL[2]{\mathbb{R}}\mathbin{/}\SL{X,\omega}$ \citep{athreya-cheung14}. A lower-dimensional subspace which is transversely intersected by almost every orbit in a higher-dimensional dynamical system is called a \textit{Poincar{\'e} section}---$\Omega$ is a Poincar{\'e} section for the horocycle flow. The Poincar{\'e} section allows us to simplify the continuous dynamical system $(\SL[2]{\mathbb{R}}\mathbin{/}\SL{X,\omega},h_s)$ into a discrete dynamical system by analyzing the \textit{return times} at which an orbit under $h_s$ returns to $\Omega$. Since $s$ represents the changes in vector slopes under the horocycle flow, return times are in correspondence with slope gaps. It follows that the problem of determining the slope gap distribution can be solved by determining return times to $\Omega$.

\subsection{Parabolic generators and \texorpdfstring{$\SL[2]{\mathbb{Z}}$}{SL2(Z)} orbits}
\label{sec:par-gens-sl2z-orbits}

In general, the Veech group of a translation surface with $n$ cusps has $n$ maximal parabolic subgroups $\Gamma_i$, $i=1,\dots,n$. There are two cases: if $-I\in\SL{X,\omega}$, where $I$ is the $2\times2$ identity matrix, then $\Gamma_i\cong\mathbb{Z}\oplus(\mathbb{Z}\mathbin{/}2\mathbb{Z})$. Otherwise, $\Gamma_i\cong\mathbb{Z}$. In this paper, we only concern ourselves with Veech surfaces whose Veech groups contain $-I$; the case when $-I\notin \SL{X,\omega}$ is similar, and details can be found in \cite{kumanduri-et-al24}. When $-I\in \SL{X,\omega}$, it is always possible to select a generator $P_i$ of the infinite cyclic factor of $\Gamma_i$ that satisfies
\[C_iP_iC_i^{-1}=S_i=\begin{bmatrix}
1 & \alpha_i \\ 0 & 1
\end{bmatrix}\]
for some $C_i\in\SL[2]{\mathbb{R}}$ and some $\alpha_i>0$ such that $C_i(X,\omega)$ has a shortest horizontal holonomy vector $\ang{1,0}$.

Recall that the Veech group of a square-tiled surface is a subgroup of $\SL[2]{\mathbb{Z}}$. For a square-tiled surface, it is possible to choose $\alpha_i\in\mathbb{N}$ and $C_i$ such that it is the product of an element of $\SL[2]{\mathbb{Z}}$ and a diagonal element $D$ of $\SL[2]{\mathbb{R}}$. The diagonal matrix $D$ may be needed to ensure the shortest horizontal vector on $C_i(X,\omega)$ has length $1$.
As mentioned previously, $\SL[2]{\mathbb{Z}}$ is generated by the matrices
\[S=\begin{bmatrix} 0 & -1 \\ 1 & \phantom{-}0 \end{bmatrix}
\qquad \text{and} \qquad
T=\begin{bmatrix} 1 & 1 \\ 0 & 1 \end{bmatrix}\,.\]
Recall that in $\mathbb{R}^2$, $S$ is a counterclockwise rotation through a quarter-turn while $T$ is a horizontal shear by $1$. In $\mathbb{H}^2$, $S$ is inversion through the unit circle followed by reflection across the imaginary axis and $T$ is a rightwards translation by $1$. It follows that $C_i$ can be expressed as a word in $S$ and $T$, possibly times a diagonal matrix $D$ in $\SL[2]{\mathbb{R}}$, and that $S_i=T^{\alpha_i}$.

To compute $P_i$, observe that the transformations $S$ and $T$ will always take a square-tiled surface to another square-tiled surface. For $S$, it is straightforward to see that rotating a collection of unit squares with vertices in $\mathbb{Z}^2$ by a quarter turn will result in another collection of unit squares with vertices in $\mathbb{Z}^2$. For $T$, observe that if we view the origami as a collection of disjoint squares, then $T$ will shear each square horizontally by one unit. We can then cut each resulting parallelogram vertically across its center (cf.\ the dashed line in Figure~\ref{fig:stabilizing-transformations}) and glue the diagonal edges together according to the edge identifications in order to recover a collection of unit squares. We can thus create a directed graph with edges labeled $S$ and $T$ showing the orbit of $(X,\omega)$ under $\SL[2]{\mathbb{Z}}$. Such a graph is shown for the three-tile origami in Figure~\ref{fig:directed-orbit-graph}. Parabolic elements of the Veech group then correspond to walks on the graph starting and ending at $(X,\omega)$ of the form $W(S,T)^{-1}T^{\alpha_i}W(S,T)$, where $W(S,T)$ is some word in $S$ and $T$.\footnote{Note that the words $W(S,T)$ are distinct from the $C_i$ and their inverses. The latter may include a diagonal matrix $D$, while the former by definition does not.} Each conjugacy class of maximal parabolic subgroups has a generator $P_i$ of its infinite cyclic factor that corresponds to a walk that does not traverse any edge twice. By finding such walks, we can find the maximal parabolic subgroups.

\begin{figure}
    \centering
    \begin{tikzpicture}[scale=.8]

    
        \draw (-5,0.5) rectangle +(1,1) ++(-1,0) rectangle ++(1,1) ++(1,-1) rectangle ++(1,1);
        
        \draw (-1,0) rectangle +(1,1) ++(1,0) rectangle ++(1,1) ++(-2,0) rectangle ++(1,1);
        
        \draw (3,0) rectangle +(1,1) ++(1,0) rectangle ++(1,1) ++(-2,0) rectangle ++(1,1);

        
        \begin{scope}[-Latex,bend left,font=\footnotesize]
            \draw[blue] (-2.8,1.6) to[edge label={$S$}] (-1.2,1.6);
            \draw[blue] (-1.2,0.6) to[edge label={$S$}] (-2.8,0.6);
            \draw[red] (1.2,1.6) to[edge label={$T$}] (2.8,1.6);
            \draw[red] (2.8,0.6) to[edge label={$T$}] (1.2,0.6);
        \end{scope}

        \begin{scope}[anchor=center,font=\footnotesize]
            \draw[red,-Latex] (-6.2,.5) arc[start angle=315,delta angle=-270,radius={1/sqrt(2)}] node[midway,label={[left]$T$}] {};
            \draw[blue,Latex-] (5.2,.5) arc[start angle=225,delta angle=270,radius={1/sqrt(2)}] node[midway,label={[right]$S$}] {};
        \end{scope}

            
        \draw (-3.5,.6) -- (-3.5,.4)
            (-4.5,1.6) -- (-4.5,1.4);

        \draw (-3.45,1.6) -- (-3.45,1.4)
            (-3.55,1.6) -- (-3.55,1.4)
            (-4.55,.6) -- (-4.55,.4)
            (-4.45,.6) -- (-4.45,.4);
            
        \draw (.5,.1) -- (.5,-.1)
            (-.5,2.1) -- (-.5,1.9);

        \draw (.45,1.1) -- (.45,.9)
            (.55,1.1) -- (.55,.9)
            (-.55,.1) -- (-.55,-.1)
            (-.45,.1) -- (-.45,-.1);
        
    \end{tikzpicture}
    \caption{The $\SL[2]{\mathbb{Z}}$ orbit of the three-tile origami as a directed graph with edges labeled by $S$ and $T$. Unmarked opposite edges are identified.}
    \label{fig:directed-orbit-graph}
\end{figure}
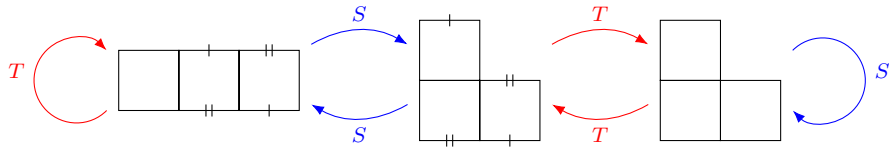

In Figure~\ref{fig:directed-orbit-graph}, the walks starting and ending at the three-tile origami are $S^2$, $T^2$, $TST$ and $STS^{-1}$. We observe that $T^2$ and $STS^{-1}$ are of the required form\footnote{Since $-I$ is in the Veech group, we are actually working with $\SL[2]{\mathbb{Z}}\mathbin{/}-I=\PSL[2]{\mathbb{Z}}$. In this group, $S=S^{-1}$, so we are justified in rewriting $STS$ as $STS^{-1}$.}. It follows that there are exactly two connected components of the three-tile origami's Poincar{\'e} section: the first corresponds to the cusp $STS^{-1}$ of width $\alpha_1=1$, while the second corresponds to the cusp $T^2$ of width $\alpha_2=2$. For the remainder of this section, we denote these connected components as $\Omega_1$ and $\Omega_2$, respectively. The fundamental domain of the three-tile origami's Veech group and its cusps are illustrated on $\mathbb{H}^2$ in Figure~\ref{fig:three-tile-cusps}; the cosets of the fundamental domain tile the entire hyperbolic upper half-plane.

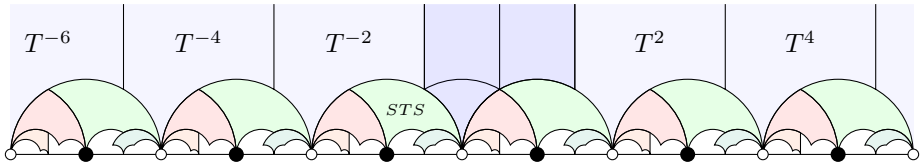
\begin{figure}
    \centering
    \begin{tikzpicture}
        \draw (-6,0) -- (6,0);

        \fill[blue,fill opacity=0.06] ({-1/2},2) -- ({-1/2},{sqrt(3)/2}) arc[start angle=60,delta angle=-60,radius=1] arc[start angle=180,delta angle=-60,radius=1] arc[start angle=120,delta angle=-60,radius=1] -- ({3/2},2);

        \draw ({1/2},{sqrt(3)/2}) arc[start angle=60,delta angle=60,radius=1] arc[start angle=60,delta angle=-60,radius=1] arc[start angle=180,delta angle=-60,radius=1] arc[start angle=120,delta angle=-60,radius=1];
        \draw ({1/2},{sqrt(3)/2}) -- ({1/2},2)
        ({-1/2},{sqrt(3)/2}) -- ({-1/2},2)
        ({3/2},{sqrt(3)/2}) -- ({3/2},2);

        \foreach \n in {-6,-4,-2,0,2,4,6} {
        \begin{scope}
            \clip (-6,0) rectangle (6,2);
            
            \fill[xshift=\n cm,blue,fill opacity=.04] (0,0) arc[start angle=180,end angle=60,radius=1] -- ({3/2},2) -- ({-1/2},2) -- ({-1/2},{sqrt(3)/2})
            arc[start angle=60,end angle=0,radius=1];

            \fill[xshift=\n cm,green,fill opacity=.1] ({1/2},{sqrt(3)/2}) arc[start angle=60,end angle=0,radius=1] arc[start angle=180,delta angle=-120,radius={1/3}] arc[start angle=120,delta angle=-120,radius={1/3}] arc[start angle=0,end angle=120,radius=1];

            \fill[xshift=\n cm,green!60!blue,fill opacity=.1] (2,0) arc[start angle=0,delta angle={180-atan(3*sqrt(3)/13)},radius={1/3}] arc[start angle={180-atan(4*sqrt(3))},end angle=0,radius={1/8}] arc[start angle=180,end angle={atan(4*sqrt(3))},radius={1/8}] arc[start angle={180-atan(5*sqrt(3)/11)},end angle=0,radius={1/5}];

            \fill[xshift=\n cm,red,fill opacity=.1] (0,0) arc[start angle=180,delta angle=-60,radius=1]
            arc[start angle=60,delta angle=-60,radius=1]
            arc[start angle=0,end angle={180-atan(5*sqrt(3)/11)},radius={1/5}]
            arc[start angle={atan(3*sqrt(3)/13)},end angle=180,radius={1/3}];

            \fill[xshift=\n cm,red!30!orange,fill opacity=.1] (0,0) arc[start angle=180,end angle=60,radius={1/3}]
            -- ({1/2},0)
            arc[start angle=0,end angle={180-atan(8*sqrt(3)/13)},radius={1/16}]
            arc[start angle={atan(5*sqrt(3)/37)},end angle=180,radius={1/5}];
            
        \end{scope}
        }
        
        \foreach \n in {-6,-4,-2,0,2,4} {

            \draw[xshift=\n cm] (0,0) arc[start angle=180,delta angle=-180,radius=1]
            ({3/2},{sqrt(3)/2}) -- ({3/2},2);

            \draw[xshift=\n cm] ({1/2},{sqrt(3)/2}) arc[start angle=60,delta angle=-60,radius=1] arc[start angle=180,delta angle=-120,radius={1/3}] arc[start angle=120,delta angle=-120,radius={1/3}];

            \draw[xshift=\n cm] (2,0) arc[start angle=0,delta angle={180-atan(3*sqrt(3)/13)},radius={1/3}] arc[start angle={180-atan(4*sqrt(3))},end angle=0,radius={1/8}] arc[start angle=180,end angle={atan(4*sqrt(3))},radius={1/8}] arc[start angle={180-atan(5*sqrt(3)/11)},end angle=0,radius={1/5}];

            \draw[xshift=\n cm] (0,0) arc[start angle=180,delta angle=-60,radius=1]
            arc[start angle=60,delta angle=-60,radius=1]
            arc[start angle=0,end angle={180-atan(5*sqrt(3)/11)},radius={1/5}]
            arc[start angle={atan(3*sqrt(3)/13)},end angle=180,radius={1/3}];

            \draw[xshift=\n cm] (0,0) arc[start angle=180,end angle=60,radius={1/3}]
            -- ({1/2},0)
            arc[start angle=0,end angle={180-atan(8*sqrt(3)/13)},radius={1/16}]
            arc[start angle={atan(5*sqrt(3)/37)},end angle=180,radius={1/5}];
            
            \ifnum\n=0
                \relax
            \else
                \node[xshift=\n cm] at ({1/2},1.5) {$T^{\n}$};
            \fi
            
            \node[xshift=\n cm,open bullet] at (0,0) {};
            \node[xshift=\n cm,bullet] at (1,0) {};
        }
        \node[open bullet] at (6,0) {};
        
        \node[font=\tiny] at (-.75,.6) {$STS$};
    \end{tikzpicture}
    \caption{The fundamental domain of the maximal parabolic subgroups of the three-tile origami's Veech group and its images in $\mathbb{H}^2$. The two cusps are indicated by open ($STS$) and closed ($T^2$) dots.}
    \label{fig:three-tile-cusps}
\end{figure}

Recall from Equation~\ref{eq:poincar'e-section} that the Poincar{\'e} section of the horocycle flow on $(X,\omega)$ is the collection of all translation surfaces in the $\SL[2]{\mathbb{R}}$ orbit of $(X,\omega)$ with a short horizontal holonomy vector, modulo the transformations which take $(X,\omega)$ to itself (namely, the transformations in its Veech group $\SL{X,\omega}$). However, each surface in the Poincar{\'e} section is uniquely represented by the element of $\SL[2]{\mathbb{R}}$ which carries $C_i(X,\omega)$ to that surface. In particular, the horizontal holonomy vector $\ang{1,0}$ of $C_i(X,\omega)$ must be mapped to a short horizontal holonomy vector $\ang{a,0}$, where $0<\abs{a}\leq1$, under each element of the Poincar{\'e} section. It follows that the Poincar{\'e} section can be parametrized by the two-dimensional family of matrices of the form
\[M_{a,b}=\begin{bmatrix}
    a & b \\ 0 & 1/a
\end{bmatrix}\,,\]
where $0<\abs{a}\leq1$.

Recall that $-I$, $S_i$ are in the Veech group of the transformed surface $C_i(X,\omega)$, so that this parametrization is not one-to-one; we can obtain a one-to-one parametrization by quotienting the $a$-$b$ plane by $-I$ and $S_i$. Quotienting by $-I$ makes it always possible to choose $a$ positive, and so replaces the constraint $0<\abs{a}\leq1$ with the constraint $0<a\leq1$. Notice that the cosets of $\ang{S_i}$ are given by
\[
\begin{bmatrix}
    a & b \\ 0 & 1/a
\end{bmatrix}\begin{bmatrix}
    1 & \alpha_i \\ 0 & 1
\end{bmatrix}^\ell=\begin{bmatrix}
    a & b \\ 0 & 1/a
\end{bmatrix}\begin{bmatrix}
    1 & \ell\alpha_i \\ 0 & 1
\end{bmatrix}=\begin{bmatrix}
    a & b+\ell\alpha_ia \\ 0 & 1/a
\end{bmatrix}
\]
for $\ell\in\mathbb{Z}$. In particular it is possible to choose $\ell$ such that $b$ falls in any fixed interval of length $a\alpha_i$. It turns out to be convenient to choose $b$ to lie in the interval
\[\left[\frac{1-x_0a}{y_0}-\alpha_ia\,,\,\frac{1-x_0a}{y_0}\right)\,,\]
where $y_0>0$ is the shortest vertical component of a saddle connection on $C_i(X,\omega)$ and $x_0$ is the shortest horizontal component of a saddle connection with height $y_0$. This produces the parametrization
\begin{equation}\label{eq:poincare-parametrization}
\Omega_i=\set{(a,b)\in\mathbb{R}^2\,\middle\vert\,0<a\leq 1,\frac{1-x_0a}{y_0}-\alpha_ia\leq b<\frac{1-x_0a}{y_0}}\,,
\end{equation}
which defines a triangle in the $a$-$b$ plane.

In the case of the three-tile origami, we see that on both connected components of the Poincar{\'e} section, $x_0=y_0=1$. The two components are therefore parametrized by the triangles
\begin{align*}
\Omega_1&=\left\{(a,b)\in\mathbb{R}^2\mid0<a\leq1,\,1-2a\leq b<1-a\right\}\,,\\
\Omega_2&=\left\{(a,b)\in\mathbb{R}^2\mid0<a\leq1,\,1-3a\leq b<1-a\right\}\,.
\end{align*}
(Note that we identify the Poincar{\'e} section components with their parametr\-izations by abuse of notation). These triangles and their ``cosets'' in the $a$-$b$ plane are shown in Figure~\ref{fig:poincar'e-section-plane}.

\begin{figure}
    \centering
    \begin{tikzpicture}[scale=.97]
    \begin{scope}[yshift=4.5cm,scale=2.1,font=\footnotesize]
        \clip (-2,-2) rectangle (2,2);
        \begin{scope}[on background layer]
            \draw[-Latex] (-2,0) -- (2,0) node[above left] {$a$};
            \draw[-Latex] (0,-2) -- (0,2) node[below right] {$b$};
        \end{scope}
    
        \draw[very thick,fill=blue,fill opacity=0.1] (1,0) -- (0,1) -- (1,-1) -- cycle;
        \draw (-1,0) -- (0,-1) -- (-1,1) -- cycle;
        \foreach \n in {-2,-1,1,2} {
            \draw (1,{-\n-1}) -- (0,1) -- (1,{-\n}) -- cycle;
            \draw (-1,{\n+1}) -- (0,-1) -- (-1,{\n}) -- cycle;
        }
    
        \begin{scope}[on background layer]
            \node[fill=white] at (0.7,0) {$\Omega_1$};
            \node[fill=white,rotate=-60] at (0.8,-1) {$S_i\Omega_1$};
            \node[fill=white,rotate=-80] at (0.81,-1.8) {$S_i^2\Omega_1$};
            \node[fill=white] at (0.7,0.8) {$S_i^{-1}\Omega_1$};
            \node[fill=white] at (0.7,1.3) {$S_i^{-2}\Omega_1$};
            \node[fill=white] at (-0.7,0) {$-\Omega_1$};
            \node[fill=white,rotate=-60] at (-0.8,1) {$-S_i\Omega_1$};
            \node[fill=white,rotate=-80] at (-0.81,1.8) {$-S_i^2\Omega_1$};
            \node[fill=white] at (-0.7,-0.8) {$-S_i^{-1}\Omega_1$};
            \node[fill=white] at (-0.7,-1.3) {$-S_i^{-2}\Omega_1$};
    
            \node[rotate=90] at (0.5,1.8) {\dots};
            \node[rotate=90] at (0.5,-1.8) {\dots};
            \node[rotate=90] at (-0.5,-1.8) {\dots};
            \node[rotate=90] at (-0.5,1.8) {\dots};
        \end{scope}
    \end{scope}
    \begin{scope}[yshift=-4.5cm,scale=2.1,font=\footnotesize]
        \clip (-2,-2) rectangle (2,2);
        \begin{scope}[on background layer]
            \draw[-Latex] (-2,0) -- (2,0) node[above left] {$a$};
            \draw[-Latex] (0,-2) -- (0,2) node[below right] {$b$};
        \end{scope}
    
        \draw[very thick,fill=red,fill opacity=0.1] (1,0) -- (0,1) -- (1,-2) -- cycle;
        \draw (-1,0) -- (0,-1) -- (-1,2) -- cycle;
        \foreach \n in {-2,-1,1,2} {
            \draw (1,{-2*\n-2}) -- (0,1) -- (1,{-2*\n}) -- cycle;
            \draw (-1,{2*\n+2}) -- (0,-1) -- (-1,{2*\n}) -- cycle;
        }
    
        \begin{scope}[on background layer]
            \node[fill=white] at (0.7,0) {$\Omega_2$};
            \node[fill=white,rotate=-80] at (0.54,-1) {$S_i\Omega_2$};
            \node[fill=white] at (0.6,1) {$S_i^{-1}\Omega_2$};
            \node[fill=white,rotate=60] at (0.5,1.8) {$S_i^{-2}\Omega_2$};
            \node[fill=white] at (-0.7,0) {$-\Omega_2$};
            \node[fill=white,rotate=-80] at (-0.54,1) {$-S_i\Omega_2$};
            \node[fill=white] at (-0.6,-1) {$-S_i^{-1}\Omega_2$};
            \node[fill=white,rotate=60] at (-0.5,-1.8) {$-S_i^{-2}\Omega_2$};
    
        \end{scope}
    \end{scope}
    \end{tikzpicture}
    \caption{The Poincar{\'e} section for the three-tile origami and its $\Gamma_i$-cosets in the $a$-$b$ plane. Note that the cosets only fill up the section of the plane with $\abs{a}\leq1$.}
    \label{fig:poincar'e-section-plane}
\end{figure}
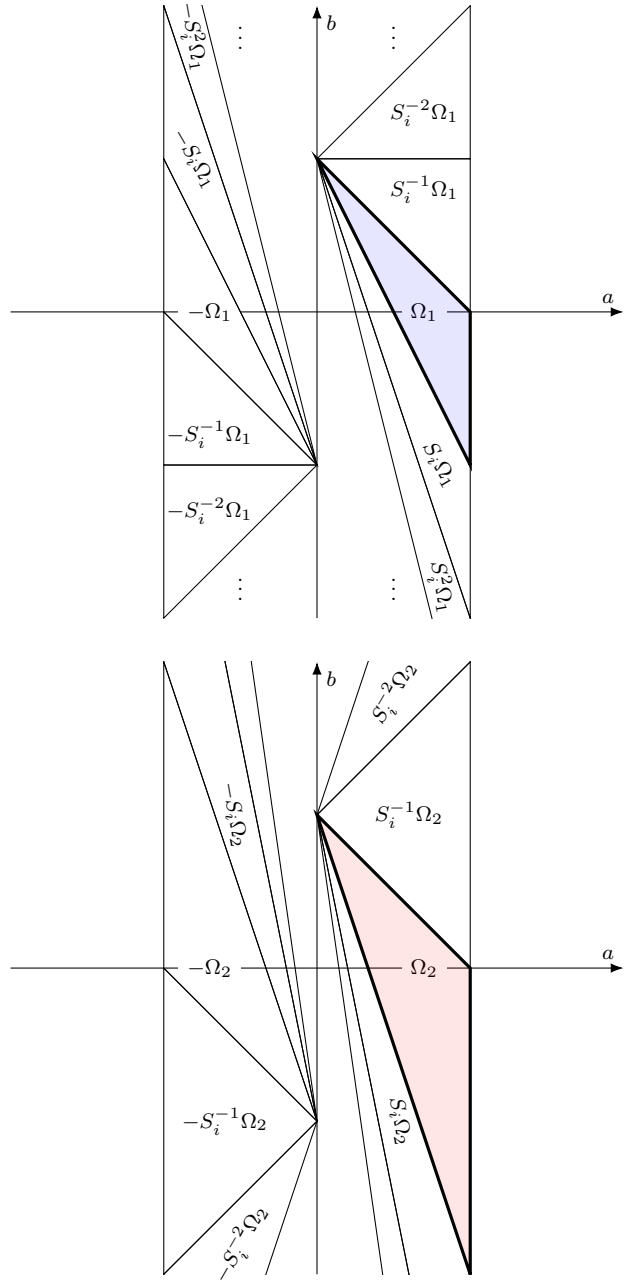

\subsection{Winning saddle connections}
Recall that the horocycle flow shears holonomy vectors vertically downwards by decreasing their slopes. It follows that the return time function to the Poincar{\'e} section, which has a short holonomy vector of slope $0$, is given by the smallest positive slope among all short holonomy vectors, since this is the vector that first reaches the horizontal axis. In particular, the slope of a holonomy vector $M_{a,b}\mathbf{v}=M_{a,b}\ang{x,y}$ on $M_{a,b}C_i(X,\omega)$ is given by the formula
\begin{equation}\label{eq:winners-slope}
\frac{y}{a(ax+by)}=\left(\frac{a^2x}{y}+b\right)^{-1}\,.
\end{equation}
The vector $\mathbf{v}=\ang{x,y}$ that first reaches the horizontal axis is called a \textit{winner}, a \textit{winning vector}, or by abuse of terminology a \textit{winning saddle connection}. The specific choice of parametrization for the Poincar{\'e} section ensures that there are a finite number of winning saddle connections on each connected component $\Omega_i$ \cite{kumanduri-et-al24}.

Since the image of a winner under $M_{a,b}$ must be a short holonomy vector, we require that $0<ax+by\leq1$. This requirement effectively carves ``strips'' out of the Poincar{\'e} section in the $a$-$b$ plane bounded above and below by the parallel lines $b=(1-ax)/y$ and $b=-ax/y$. Notice by comparison with Equation~\ref{eq:winners-slope} that where two such strips intersect, the steeper one represents the winner. The problem then reduces to deciding what regions of $\Omega_i$ each holonomy vector wins on. The winning regions are always convex polygons. To determine the winning saddle connections, we adapt a different algorithm from \cite{alassal-et-al25}. Since the exposition of this algorithm detracts from the main argument here we postpone it to Section~\ref{ch:determining-winners}.

In the case of the three-tile surface, it is possible to see from direct inspection of the shapes of winning strips that in the case of the three-tile origami, $\ang{1,1}$ wins everywhere on $\Omega_1$, while $\ang{1,1}$ and $\ang{2,1}$ partition $\Omega_2$ into two winning triangles separated by the line $b=1-2a$. This is shown in Figure~\ref{fig:winning-regions-three-tile}.

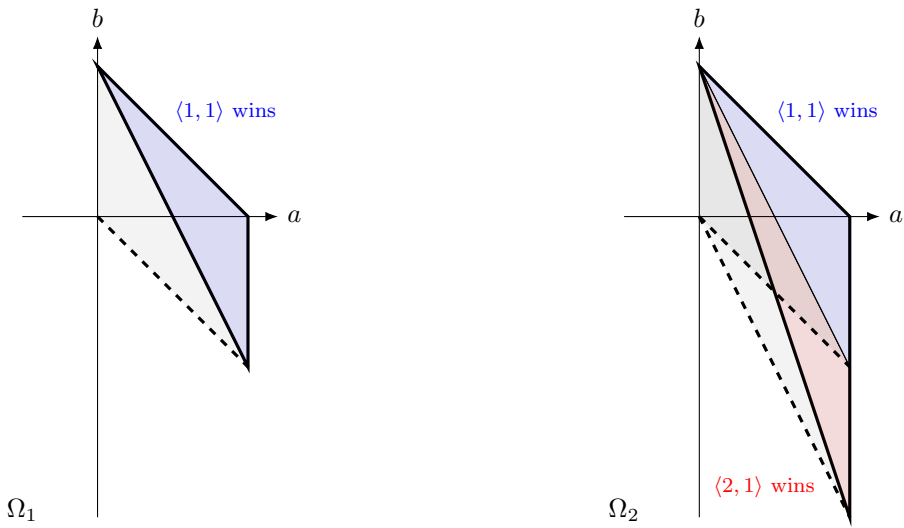
\begin{figure}
    \centering
    \begin{tikzpicture}
    \begin{scope}[xshift=-4cm,scale=2]
        \begin{scope}[on background layer]
            \draw[-Latex] (-.5,0) -- (1.2,0) node[right] {$a$};
            \draw[-Latex] (0,-2) -- (0,1.2) node[above] {$b$};
        \end{scope}
    
        \draw[very thick] (1,0) -- (0,1) -- (1,-1) -- cycle;
        \draw[fill=blue,fill opacity=.1] (1,0) -- (0,1) -- (1,-1) -- cycle;
    
        \begin{scope}
            \clip (-.5,-2) rectangle (1.2,1.2);
            \draw[very thick,dashed] (0,0) -- (1,-1);
            \fill[fill opacity=.05] (0,1) -- (1,0) -- (1,-1) -- (0,0) -- cycle;
        \end{scope}
        
        \node[blue,font=\footnotesize] at (.85,.7) {$\ang{1,1}$ wins};
        \node[anchor=base] at (-.5,-2) {$\Omega_1$};
    \end{scope}
    \begin{scope}[xshift=4cm,scale=2]
        \begin{scope}[on background layer]
            \draw[-Latex] (-.5,0) -- (1.2,0) node[right] {$a$};
            \draw[-Latex] (0,-2) -- (0,1.2) node[above] {$b$};
        \end{scope}
    
        \draw[very thick] (1,0) -- (0,1) -- (1,-2) -- cycle;
        \draw[fill=blue,fill opacity=.1] (1,0) -- (0,1) -- (1,-1) -- cycle;
        \draw[fill=red,fill opacity=.1] (1,-2) -- (0,1) -- (1,-1) -- cycle;
    
        \begin{scope}
            \clip (-.5,-2) rectangle (1.2,1.2);
            \draw[very thick,dashed] (0,0) -- (1,-2);
            \fill[fill opacity=.05] (0,1) -- (1,0) -- (1,-1) -- (0,0) -- cycle;
            \draw[very thick,dashed] (0,0) -- (1,-1);
            \fill[fill opacity=.05] (0,1) -- (0,0) -- (1,-2) -- (1,-1) -- cycle;
        \end{scope}
        
        \node[blue,font=\footnotesize] at (.85,.7) {$\ang{1,1}$ wins};
        \node[red,font=\footnotesize] at (0.43,-1.8) {$\ang{2,1}$ wins};
        \node[anchor=base] at (-.5,-2) {$\Omega_2$};
    \end{scope}
    \end{tikzpicture}
    \caption{Winners for the three-tile origami and their corresponding strips.}
    \label{fig:winning-regions-three-tile}
\end{figure}

\subsection{The return time function}

Recall that the return time function at each point $(a,b)\in\Omega_i$ is given by the slope of the winner $\ang{x,y}$ at that point:
\[t=\frac{y}{a(ax+by)}\,.\]
The level sets of these return time functions in the $a$-$b$ plane are hyperbolae of form
\[b=\frac{1}{at}-\frac{ax}{y}\,.\]
Integrating over the portion of the winning region bounded by the hyperbola corresponding to some positive $t$, once suitably normalized, gives a measure of what proportion of the Poincar{\'e} section has a return time in the interval $[0,t]$.\footnote{It is possible to simply calculate the area bounded by the return time function's level sets because the natural ergodic measure on the Poincar{\'e} section $\mu$ is equivalent to the usual Lebesgue measure on $\mathbb{R}^2$, and because the equidistribution properties of the horocycle flow on $\SL[2]{\mathbb{R}}/\SL{X,\omega}$ carry over to the return time map on the transversal. This is a nontrivial fact proved in for instance \cite{athreya16}.} Since return times are related to slope gaps, it follows that calculating this area for each positive $t$, summing over all areas of different transversal regions, and normalizing produces the CDF of the slope gap distribution.

Since winning regions are always convex polygons, the return time level sets sweep through the vertices, causing the area they bound to change in a nonsmooth way. These vertices appear as points of nonsmoothness in the final distribution.

Consider now the Poincar{\'e} section for the three-tile origami. On the sole winning region of $\Omega_1$, the return-time hyperbola
\[b=\frac{1}{at}-a\]
first touches the $(1,1)$ winning region at time $t=1$, intersecting the region's top edge at $a=1/t$. The area bounded by the hyperbola at this point is given by
\[\int_{1/t}^1\left(1-a-\frac{1}{at}+a\right)dt=1-\frac{1+\ln t}{t}\,.\]
At $t=4$, the hyperbola passes through the lower edge of the region, leading to a point of nonsmoothness. The two intersections it makes with the lower edge are at $(1/2)(1\pm\sqrt{1-4/t})$, so the area for this region is given by the previous expression minus the quantity
\[\int_{\frac{1}{2}\left(1-\sqrt{1-\frac{4}{t}}\right)}^{\frac{1}{2}\left(1+\sqrt{1-\frac{4}{t}}\right)}\left(1-2a-\frac{1}{at}+a\right)dt=\frac{1}{2}\sqrt{1-\frac{4}{x}}-\frac{2}{x}\arctanh\left(\sqrt{1-\frac{4}{x}}\right)\,.\]

Past $t=4$, the hyperbola continues to sweep out area, never capturing the entire Poincar{\'e} section component but approaching the total area of $1/2$. Altogether, the contribution to the CDF for $(1,1)$'s winning region in $\Omega_1$ is given by the piecewise function
\begin{equation}\label{eq:hall-distribution-CDF}
\renewcommand{\arraystretch}{1.5}
\begin{cases}
    0 & 0\leq t<1\,, \\
    1-\dfrac{1+\ln t}{t} & 1\leq t<4\,, \\
    1-\dfrac{1+\ln t}{t} - \dfrac{1}{2}\sqrt{1-\dfrac{4}{t}}+\dfrac{2}{t}\arctanh\left(\sqrt{1-\dfrac{4}{t}}\right) & 4\leq t\,.
\end{cases}
\end{equation}
The CDF and the areas that swept it out are shown in Figure~\ref{fig:sweeping-hyperbolae}.

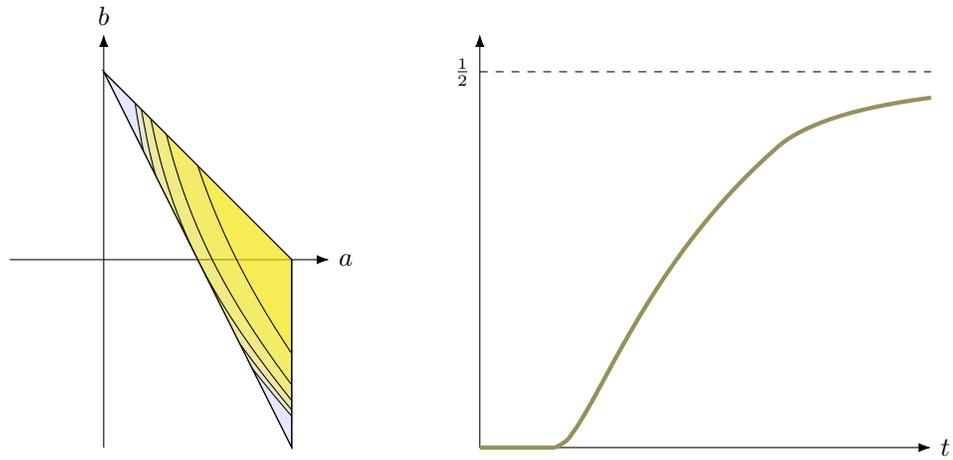
\begin{figure}
    \centering
    \begin{tikzpicture}
    \begin{scope}[xshift=2.5cm,yshift=-2.5cm,xscale=.5,yscale=10]
        \begin{scope}[on background layer]
            \draw[-Latex] (0,0) -- (12,0) node[right] {$t$};
            \draw[-Latex] (0,0) -- (0,.55);
        \end{scope}

        \begin{scope}[ultra thick,yellow!50!black]
            \draw[domain=0:2,smooth,variable=\t] (0,0) plot (\t,0);
            \draw[domain=2:8,smooth,variable=\t,samples=20] (0,0) plot (\t,{1-(2/\t)*(1+ln(\t/2))});
            \draw[domain=8:12,smooth,variable=\t,samples=20] (0,0) plot (\t,{1-(2/\t)*(1+ln(\t/2))-(1/2)*sqrt(1-(8/\t))+(2/\t)*ln((1+sqrt(1-(8/\t)))/(1-sqrt(1-(8/\t))))});
        \end{scope}

        \draw[dashed] (0,.5) node[left,font=\footnotesize] {$\frac12$} -- +(12,0);
    \end{scope}
    \begin{scope}[xshift=-2.5cm,scale=2.5]
        \begin{scope}[on background layer]
            \draw[-Latex] (-.5,0) -- (1.2,0) node[right] {$a$};
            \draw[-Latex] (0,-1) -- (0,1.2) node[above] {$b$};
        \end{scope}

        \draw[fill=blue,fill opacity=.1] (0,1) -- (1,0) -- (1,-1) -- cycle;

        \begin{scope}
        \clip (0,1) -- (1,0) -- (1,-1) -- cycle;
        \foreach \t in {6,5,4,3,2} {
            \draw[fill=yellow,fill opacity=.2,domain={1/\t}:1,smooth,variable=\a] plot ({\a},{1/(\t*\a)-\a}) -- (1,0);
        }
        \end{scope}
        
        \draw (0,1) -- (1,0) -- (1,-1) -- cycle;
    \end{scope}
    \end{tikzpicture}
    \caption{The areas bounded by level sets of $t$ correspond to the CDF of the slope gap distribution.}
    \label{fig:sweeping-hyperbolae}
\end{figure}

Turning now to $\Omega_2$, some consideration shows that the same CDF will be computed for the $(1,1)$ winning region following the same procedure. The $(2,1)$ winning region proceeds similarly, except the return time level-sets are now given by
\[b=\frac{1}{at}-2a\,.\]
A change of coordinates bypasses more tedious computation. Let $b'=b+a$, so that the $a$-$b'$ coordinate system is sheared upward by $1$ from the original plane. This transforms the $(2,1)$ winning region into the $(1,1)$ winning region and changes the return time level-sets to
\[b'=b+a=\frac{1}{at}-2a+a=\frac{1}{at}-a\,.\]
We now recognize that the $(2,1)$ winning region makes the exact same contribution to the total CDF. Since all three of the CDF contributions have the same shape, summing them and then normalizing is equivalent to normalizing just one.\footnote{However, in general the individual CDF contributions must be added together \emph{before} normalization.} We conclude that the total normalized CDF is simply twice the function given in Equation~\ref{eq:hall-distribution-CDF}.

This distribution turns out to be the CDF of the Hall distribution, which is the slope gap distribution of the simplest square-tiled surface, the one-tile square torus. The normalized pdf of the Hall distribution is plotted in Figure~\ref{fig:hall-distribution-pdf}, and its equation is given below. Notice the points of nondifferentiability at $t=1$ and $4$.
\[\begin{cases}
    0 & 0\leq t<1\,, \\
    \dfrac{\ln{t^2}}{t^2} & 1\leq t<4\,, \\
    \dfrac{\ln{t^2}}{t^2} - \dfrac{4}{t^2}\arctanh\left(\sqrt{1-\dfrac{4}{t}}\right) & 4\leq t\,.
\end{cases}\]

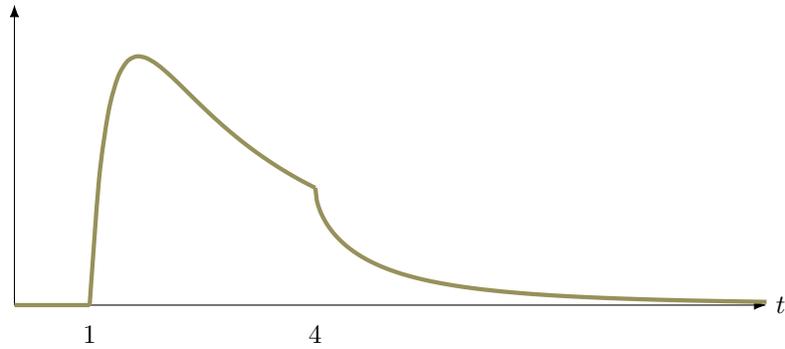
\begin{figure}
    \centering
    \begin{tikzpicture}
    \begin{scope}[on background layer]
        \draw[-Latex] (0,0) -- (10,0) node[right] {$t$};
        \draw[-Latex] (0,0) -- (0,4);
    \end{scope}

    \begin{scope}[ultra thick,yellow!50!black]
        \draw[domain=0:1,smooth,variable=\t] (0,0) plot (\t,0);
        \draw[domain=1:4,smooth,variable=\t] (0,0) plot (\t,{18*(ln(\t)/(\t^2))});
        \draw[domain=4:10,samples=300,variable=\t] (0,0) plot (\t,{18*(ln(\t)/(\t^2)-1/(\t^2)*(ln((1+sqrt(1-4/\t))/(1-sqrt(1-4/\t)))))});
    \end{scope}

    \node[anchor=base] at (1,-.5) {$1$};
    \node[anchor=base] at (4,-.5) {$4$};
    \end{tikzpicture}
    \caption{The pdf of the Hall distribution.}
    \label{fig:hall-distribution-pdf}
\end{figure}

%% file: sumry-algorithm.tex
\section{Determining winners}
\label{ch:determining-winners}

This section details an algorithm used to determine saddle-connection winners developed by \cite{alassal-et-al25}. We now provide a brief overview of the algorithm.
\begin{alg}[The SUMRY algorithm]\hfill
\begin{enumerate}
    \item Let a Poincar{\'e} section component $\Omega$ corresponding to the cusp represented by $W$ and a holonomy vector $\mathbf{v}_0=\ang{x_0,y_0}$ of $W(X,\omega)$ be given.
    \item Parametrize $\Omega$ as a triangle in the $a$-$b$ plane whose top edge lies along the line $b=1$ (Equation~\ref{eq:rotated-parametrization}), and let its top-right corner have coordinates $(a_0,1)$.
    \item For each $k\geq0$, search for vectors that win over $\mathbf{v}_k=\ang{x_k,y_k}$ in the region
    \[\set{(a,b)\in\mathbb{R}^2\,\middle\vert\,\frac{x-1}{a}<y<\frac{xy_k}{x_k}}\,,\]
    or show that none exist.
    \item By repeating step~3, find the vector $\mathbf{v}_{k+1}=\ang{x_1,y_1}$ that wins on the half-open interval
    \[I_k=\left[\frac{x_k-1}{y_k},\frac{x_{k-1}-1}{y_{k-1}}\right)\,,\]
    where by convention we take
    \[\frac{x_{-1}-1}{y_{-1}}=\frac{x_0-1}{y_0}+\alpha\,.\]
    \item By repeating step~4, partition the top edge of $\Omega$ into half-open intervals and their associated winners $(\mathbf{v}_k,I_k)$.
    \item Reconstruct the winning regions by projecting the intervals back onto strips and performing a finite number of comparisons on the interior of $\Omega$.
\end{enumerate}
\end{alg}

\subsection{A rotated perspective}

Recall from Equation~\ref{eq:poincare-parametrization} that the $i$\textsuperscript{th} connected component $\Omega_i$ of the Poincar{\'e} section of $(X,\omega)$ under the horocycle flow can be parametrized by two real parameters $a$ and $b$. Each point in $\Omega_i$ (a translation surface) is of the form $M_{a,b}C_i(X,\omega)$, where
\[M_{a,b}=\begin{bmatrix}
    a & b \\ 0 & 1/a
\end{bmatrix}\]
and $M_{a,b}C_i(X,\omega)$ has a horizontal holonomy vector of length at most $1$. Taking quotients by elements of $\SL{X,\omega}$, we see that $\Omega_i$ can be represented by the set
\[\Omega_i=\set{(a,b)\in\mathbb{R}^2\,\middle\vert\,0<a\leq 1,\frac{1-x_0a}{y_0}-\alpha_ia\leq b<\frac{1-x_0a}{y_0}}\,.\]
In the above equation, $x_0$ and $y_0$ are the components of the holonomy vector with shortest vertical component and $\alpha_i$ is the corresponding cusp width.

An equivalent parametrization, which proves more convenient, is obtained by rotating this triangle counterclockwise by a quarter-turn in the $a$-$b$ plane. This yields
\begin{equation}\label{eq:rotated-parametrization}
\Omega_i'=\set{(a,b)\in\mathbb{R}^2\,\middle\vert\,0<b\leq 1,\frac{x_0b-1}{y_0}\leq a<\frac{x_0b-1}{y_0}+\alpha_ib}\,.
\end{equation}
Figure~\ref{fig:rotated-parametrization} shows the new parametrization. In this figure and for the rest of this section we simplify notation by writing $\Omega$ for $\Omega_i'$, $C$ for $C_i$, and $\alpha$ for $\alpha_i$. 
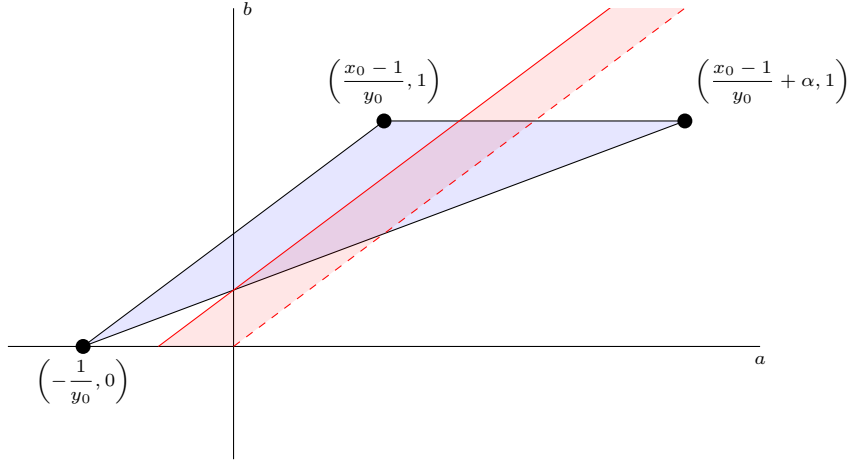
\begin{figure}
    \centering
    \begin{tikzpicture}[xscale=2,yscale=3,font=\scriptsize]
        \draw (-1.5,0) -- (3.5,0)
                (0,-.5)--(0,1.5);
        \useasboundingbox;
        
        \fill[blue,opacity=.1] (-1,0) -- (1,1) -- (3,1) -- cycle;
        \draw (-1,0) node[bullet,label={[below,yshift=-1ex]$\left(-\dfrac{1}{y_0},0\right)$}] {}
            -- (1,1) node[bullet,label={[above]$\left(\dfrac{x_0-1}{y_0},1\right)$}] {}
            -- (3,1) node[bullet,label={[above right]$\left(\dfrac{x_0-1}{y_0}+\alpha,1\right)$}] {}
            -- cycle;
        \node[below] at (3.5,0) {$a$};
        \node[right] at (0,1.5) {$b$};

        \begin{scope}
            \clip (-3.5,-1.5) rectangle (3.5,1.5);
            \fill[red,opacity=.1] (0,0) -- (10,5) -- (9.5,5) -- (-.5,0) -- cycle;
            \draw[red,dashed] (0,0) -- (10,5);
            \draw[red,xshift=-.5cm] (0,0) -- (10,5);
        \end{scope}
    \end{tikzpicture}
    \caption{The rotated parametrization $\Omega_i'$ and the strip $S_\Omega(x,y)$.}
    \label{fig:rotated-parametrization}
\end{figure}

The rotation above amounts to rearranging the parameters' roles in $M_{a,b}$, so that instead we consider $\Omega$ as the set of translation surfaces of form
\[M_{b,-a}C(X,\omega)=\begin{bmatrix}
    b & -a \\ 0 & 1/b
\end{bmatrix}C(X,\omega)\]
with a horizontal saddle connection of length at most $1$.

Now recall that to each vector $\ang{x,y}\in\Lambda(C(X,\omega))$ corresponds a \textit{candidacy strip} on which $\ang{x,y}$ could win. In the unrotated parametrization, this candidacy strip was defined by the set of values $a$, $b$ such that $M_{a,b}\ang{x,y}$ had horizontal component at most $1$ and positive vertical component. This is just the set $\set{(a,b)\in\mathbb{R}^2\mid y>0\,,0<ax+by\leq1}$. In the rotated parametrization we instead impose the constraint that $M_{b,-a}\ang{x,y}$ have horizontal component at most $1$ and positive vertical component. Since
\[\begin{bmatrix}
    b & -a \\ 0 & 1/b
\end{bmatrix}\ang{x,y}=\begin{bmatrix}
    b & -a \\ 0 & 1/b
\end{bmatrix}\begin{bmatrix}
    x \\ y
\end{bmatrix}=\begin{bmatrix}
    bx-ay \\ y/b
\end{bmatrix}\,,\]
this yields the strip
\[S_\Omega(x,y)=\set{(a,b)\in\mathbb{R}^2\mid y>0\,, 0<bx-ay\leq1}\,.\]
Note that although the steeper strip wins where two strips intersect in the unrotated parametrization, in the rotated parametrization it is the shallower strip that wins. This is because the slope of the strip's parallel sides is given by $y/x$, while the slope of a transformed vector is
\[\frac{y}{b(bx-ay)}=\left(\frac{b^2x}{y}-a\right)^{-1}\,.\]
From the preceding equation we see that the greater the slope of the strip's parallel sides, the greater the slope of the transformed vector.

With this knowledge in hand, we now prove a short lemma that significantly simplifies the problem.

\begin{lem}\label{lem:top-edge}
Let
\[E=\left\{(a,b)\,\middle\vert\,\frac{x_0-1}{y_0}\leq a\leq\frac{x_0-1}{y_0}+\alpha,b=1\right\}\]
be the top edge of $\Omega$. Every vector $\ang{x,y}$ that wins somewhere on $\Omega$ wins on an interval of $E$ (i.e., a connected subset of $E$ with nonempty interior), whose left endpoint is
\[p_{\ang{x,y}}=\left(\frac{x-1}{y},1\right)\,.\]
Moreover, if $(x-1)/y=(x_0-1)/y_0$, then $p_{\ang{x,y}}$ lies on the left edge of $\Omega$.
\end{lem}

\begin{proof}
To prove the first part of the lemma, let $\ang{x,y}$ be a winner at point $q\in\Omega$. Notice that the left edge of $S_\Omega(x,y)$ is given by the line $L_{\ang{x,y}}$ with equation
\[b=\frac{y}{x}a+\frac1x\,,\]
and that this intersects with the line $b=1$ at $p_{\ang{x,y}}$. Now suppose that the distinct vector $\ang{x',y'}$ wins at $p_{\ang{x,y}}$. Then $S_\Omega(x',y')$ must contain $p_{\ang{x,y}}$, so that $p_{\ang{x,y}}$ lies below the left edge of $S_\Omega(x',y')$, $L_{\ang{x',y'}}$. Now, $p_{\ang{x,y}}$ lies on $L_{\ang{x,y}}$, but if $S_\Omega(x',y')$ wins over $S_\Omega(x,y)$ at $p_{\ang{x,y}}$, then $L_{\ang{x',y'}}$ must have a shallower slope than $L_{\ang{x,y}}$. Hence the left edges of the two strips intersect to the right of $p_{\ang{x,y}}$.

Now, $S_\Omega(x,y)$ must contain $q$, so $q$ lies under $L_{\ang{x,y}}$. Also, $q$ must be in $\Omega$, so it must lie under $E$. But $L_{\ang{x',y'}}$ lies above $L_{\ang{x,y}}$ everywhere to the left of $p_{\ang{x,y}}$ and lies above $E$ everywhere to its right. It follows that $q$ must lie under $L_{\ang{x',y'}}$.

Similarly, $q$ lies above the right edge of $S_\Omega(x,y)$ since $q\in S_\Omega(x,y)$ and $q$ lies above the horizontal axis since $q\in\Omega$. Now, the right edges of $S_\Omega(x,y)$ and $S_\Omega(x',y')$ intersect at the origin and the right edge of $S_\Omega(x',y')$ has a shallower slope than the right edge of $S_\Omega(x,y)$. Everywhere to the left of the origin, both right edges lie below the horizontal axis. It follows that $q$ must lie above the right edge of $S_\Omega(x',y')$. But now $q$ lies in $S_\Omega(x',y')$, hence $\ang{x,y}$ is not the winner at $q$, contradicting our initial assumptions. This forces us to conclude $\ang{x,y}$ wins at $p_{\ang{x,y}}$.

Next, we show that $\ang{x,y}$ wins on an interval of $E$. Let the set of points in $E$ where $\ang{x,y}$ wins be $E_{\ang{x,y}}$. Within $E$, $\ang{x,y}$ wins at all and only the points in $S_\Omega(x,y)$ that are not also in a candidacy strip $S_\Omega(\mathbf{v})$ for some holonomy vector on $\Omega$ of slope less than the slope of $\ang{x,y}$. Therefore, we can write 
\begin{equation}\label{eq:winning-interval}
E_{\ang{x,y}}=(E\cap S_\Omega(x,y))\setminus\left({\smashoperator\bigcup_{\substack{\mathbf{v}\in\Lambda\\m(\mathbf{v})<y/x}}}S_\Omega(\mathbf{v})\right)\,.
\end{equation}
The right edge of each candidacy strip $S_\Omega(\mathbf{v})$ is a line of slope $m(\mathbf{v})$ passing through the origin. All vectors $\mathbf{v}$ indexed over in the union in Equation~\ref{eq:winning-interval} have slope less than the right edge of $S_\Omega(x,y)$, so each strip $S_\Omega(\mathbf{v})$ covers an interval containing the right endpoint of $E\cap S_\Omega(x,y)$. The union of all the strips $S_\Omega(\mathbf{v})$ in turn covers an interval containing the right endpoint of $E\cap S_\Omega(x,y)$, and comparison with Equation~\ref{eq:winning-interval} shows that $E_{\ang{x,y}}$ is connected.

Next, $E_{\ang{x,y}}$ has empty interior only if it is the singleton set $\{p_{\ang{x,y}}\}$. Consider the left edges of the strips $S_\Omega(\mathbf{v})$: either at least one left edge intersects $E$ at or to the left of $p_{\ang{x,y}}$, in which case $\ang{x,y}$ does not win at $p_{\ang{x,y}}$, or no left edge intersects $E$ to the right of $p_{\ang{x,y}}$, in which case $E_{\ang{x,y}}$ contains more than just the point $p_{\ang{x,y}}$. In neither case does $E_{\ang{x,y}}$ equal $\{p_{\ang{x,y}}\}$, so we conclude $E_{\ang{x,y}}$ is connected with nonempty interior, i.e., it is an interval.

Finally, we show that $p_{\ang{x,y}}$ lies in $\Omega$. The left edge of $\Omega$ is given by the line $L_0$ with equation
\[b=\frac{y_0}{x_0}a+\frac{1}{x_0}\,,\]
which is also the left edge of $S_\Omega(x_0,y_0)$. We show that no strip whose left edge lies to the left of $L_0$ wins on the interior of $\Omega$. First, recall that by definition
\[y_0=\min_{\ang{x,y}\in\Lambda}\{y\}\,,\]
so no strip $S_\Omega(x,y)$, $\ang{x,y}\in\Lambda$, intersects the horizontal axis left of $\Omega$. Hence if $p_{\ang{x,y}}$ lies outside of $\Omega$ then $L_{\ang{x,y}}$ must be steeper than $L_0$, that is, $y/x>y_0/x_0$. But then $\ang{x_0,y_0}$ wins over $\ang{x,y}$ in a neighborhood of $L_0$ since shallower vectors win over steeper ones. Hence $p_{\ang{x,y}}$ is not the left endpoint of a winning interval, contrary to assumption---we conclude that the left endpoints of winning intervals must lie in $\Omega$.
\end{proof}

\begin{figure}
    \centering
    \begin{tikzpicture}[xscale=2,yscale=4]
        \draw (-1.2,0) -- (5.1,0)
                (0,0)--(0,1.2);
        \useasboundingbox;

        \fill[blue,opacity=.1] (-1,0) -- (1,1) -- (5,1) -- cycle;
        \node[below] at (5.1,0) {$a$};
        \node[right] at (0,1.2) {$b$};
        
        \begin{scope}
            \clip (-1.5,0) rectangle (5.5,1.2);
            \fill[red,opacity=.1] (0,0) -- (7,3) -- (6.5,3) -- (-.5,0) -- cycle;
            \draw[red,dashed] (0,0) -- (7,3);
            \draw[red,xshift=-.5cm] (0,0) -- (7,3);
            
            \fill[green,opacity=.1] (0,0) -- (5,2) -- (4.2,2) -- (-.8,0) -- cycle;
            \draw[green!50!black,dashed] (0,0) -- (5,2);
            \draw[green!50!black,xshift=-.8cm] (0,0) -- (5,2);
        \end{scope}

        \node[open bullet,inner sep=1.5pt,label={[below right,font=\scriptsize,xshift=-9pt,yshift=-3pt]{$p_{\ang{x,y}}$}}] at ({11/6},1) {};
        \node[bullet,inner sep=1.5pt,label={[above right,font=\scriptsize]{$q$}}] at ({3/3},{1/2}) {};
    \end{tikzpicture}
    \caption{Diagram of the proof of Lemma~\ref{lem:top-edge}. The steeper, red strip is $S_\Omega(x,y)$ and the shallower, green strip is $S_\Omega(x',y')$.}
    \label{fig:lemma-proof}
\end{figure}
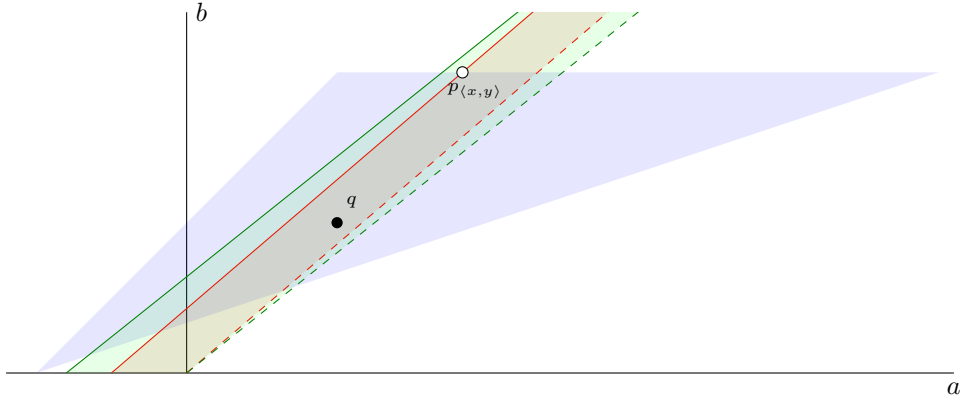

It therefore suffices to consider only which vectors win on the line $b=1$. Once all the intervals and the winners on each interval are found, it is straightforward to retrace the winning strips which correspond to each interval. These winning strips then partition $\Omega$ into winning regions, as required for the next step of the modified KSW algorithm. We now turn to consideration of how to find winning vectors.

\subsection{Bounded and unbounded regions}

As it turns out, it often suffices to have a suspected ``candidate winner'' to find the winner on a region. Just as any vector $\ang{x,y}$ has a winning strip on the $a$-$b$ plane given by the equation $0<bx-ay\leq1$, any point $(a,b)\in\Omega$ defines a ``winning strip'' in the $x$-$y$ plane explicitly given by
\[S_\Lambda(a,b)=\set{\ang{x,y}\in\mathbb{R}^2\,\middle\vert\,\frac{b}{a}x-\frac{1}{a}<y\leq\frac{b}{a}x}\,.\]
The $x$-$y$ plane can be thought of as the space $\Lambda$ of possible holonomy vectors, so $S_\Lambda(a,b)$ is the strip containing all holonomy vectors which could possibly win at the point $(a,b)$. Note that because we always consider points on the top edge $b=1$, the preceding expression simplifies to
\[S_\Lambda(a,1)=\set{\ang{x,y}\in\mathbb{R}^2\,\middle\vert\,\frac{1}{a}(x-1)<y\leq\frac{1}{a}x}\,.\]
That is, the strip containing all possible winners at $(a,1)$ is bounded by two lines of slope $1/a$ in the $x$-$y$ plane intersecting the $x$-axis at $x=0$ and $x=1$, respectively.

Now suppose $\ang{u,v}$ is a candidate winner at some point $(a,b)$ on $\Omega$. Any vector $\ang{u,v}$ which wins over $\ang{x,y}$ must have a steeper slope---that is, we require
\[\frac{v}{u}>\frac{y}{x}\,.\]
But this simply determines another region in the $x$-$y$ plane, namely
\[\set{\ang{x,y}\in\mathbb{R}^2\,\middle\vert\,0<y<\frac{v}{u}x}\,.\]
The intersection of $S_\Lambda(a,b)$ with this region is bounded if and only if the two regions are bounded by lines of different slopes, that is, if
\[\frac{b}{a}\neq\frac{v}{u}\quad\text{or equivalently if}\quad bu-av\neq0\,.\]
In this case, the region looks like a triangle and is bounded. Since the set of holonomy vectors of any translation surface is discrete, any bounded strip contains a finite number of possible winning vectors. It is then a matter of brute-force computation to determine the winning vector.

If instead $bu-av=0$, then the region looks like an unbounded infinite strip. Since $b=1$ for all points on the top edge, we obtain a parallel strip whenever $u=av$. Now let
\[T^{t}=\begin{bmatrix}
    1 & t \\ 0 & 1
\end{bmatrix}\]
for any real $t$. Notice that when $t$ is an integer this is just the $t$\textsuperscript{th} power of the shearing generator $T$ of $\SL[2]{\mathbb{Z}}$. If we shear the surface $C(X,\omega)$ by $T^{-u/v}$, the set of holonomy vectors is also sheared by $T^{-u/v}$:
\[T^{-u/v}\Lambda(C(X,\omega))=\Lambda(T^{-u/v}C(X,\omega))\,.\]
In particular, the parallel strip $S_\Lambda(a,1)$ in $x$-$y$ space is sheared into the vertical strip $0\leq x<1$. In this case, it may happen that the algorithm does not terminate.

In some cases, it is possible at this point to show no other vectors lie in this strip. Recall $C$ can be written as the product of a unit-determinant diagonal matrix $D$ and a word $W$ in $S$ and $T$. In particular, let
\[D=\begin{bmatrix}
    d & 0 \\ 0 & 1/d
\end{bmatrix}\,.\]
Then straightforward computation shows the following diagram commutes, and the set of holonomy vectors of $W(X,\omega)$ is related to the set of holonomy vectors of the sheared surface by the transformation $T^{-u/(vd^2)}=D^{-1}T^{u/v}D$.
\[
\begin{tikzcd}
W(X,\omega) \arrow[r] \arrow[d] & T^{-u/(vd^2)}W(X,\omega) \arrow[d] \\
DW(X,\omega) \arrow[r] & T^{-u/v}DW(X,\omega)
\end{tikzcd}
\]
Properties of the quantity $-u/(vd^2)$ can then help to eliminate possibilities. For instance, if $-u/(vd^2)=n$ so that $T^n\in\SL{W(X,\omega)}=\SL{X,\omega}$ then
\[T^n\Lambda(W(X,\omega))=\Lambda(T^nW(X,\omega))=\Lambda(W(X,\omega))\,.\]
Since $W(X,\omega)$ is just another square-tiled surface, it can have no holonomy vectors of horizontal component shorter than $1$. It follows that the parallel strip must be empty of other holonomy vectors.

\subsection{Partitioning the top edge}

At this point we have a strategy for dividing the top edge into intervals with winning vectors, from which winning regions can be reconstructed. Start with a guess for the winning vector $\mathbf{v}_0=\ang{x_0,y_0}$ in a small neighborhood of the top-right corner of $\Omega$, $(a_0,1)$. Note that $\mathbf{v}_0$ must be an actual holonomy vector of $W(X,\omega)$ to be a viable candidate. Use the above procedure to find the true winner $\mathbf{v}_1=\ang{x_1,y_1}$ at the top-right corner of $\Omega$. The left edge of the strip $S_\Omega(\mathbf{v}_1)$ intersects the top edge at the point with $a$-coordinate $a_1=(x_1-1)/y_1$. Next, using $\mathbf{v}_1$ as the new candidate, find the actual winning vector in a small neighborhood of $(a_1,1)$, say $\mathbf{v}_2$. This then defines a new point of intersection $a_2$. At the $k$\textsuperscript{th} step, given $\mathbf{v}_{k-1}$, compute $a_k$ and the next winning vector $\mathbf{v}_k$ using the above procedure. Repeat until the entire top edge is covered by intervals.

Although the possibility of unbounded strips limits the algorithm's general effectiveness, the procedure described above is sufficient for the next section, in which we calculate the slope gap distribution for a certain 10-tile square-tiled-surface.

%% file: calculations.tex
\section{Calculating the ten-tile origami's slope gap distribution}
\label{ch:calculations}

In this section we calculate the slope gap distribution of the balanced ten-tile origami labeled $(X,\omega)$ in Figure~\ref{fig:ten-tile-STS}. In the next section we prove that the resultant distribution is not a sum of scaled Hall distributions, unlike every other slope gap distributed computed for square-tiled surfaces until now.

\subsection{\texorpdfstring{Poincar{\'e}}{Poincaré} section components}

We automated the method detailed in \S\,\ref{sec:par-gens-sl2z-orbits} using the SageMath package \texttt{surface-dynamics} (see Appendix~\ref{app:code} for the code). The output of the code shows that the ten-tile origami's Veech group has four cusps, represented by the four $\SL[2]{\mathbb{Z}}$ words
\[T^5\,,\quad STS^{-1}\,,\quad (T^3ST^2S)T(T^3ST^2S)^{-1}\,,\quad (T^3ST^2)T^5(T^3ST^2)^{-1}\,.\]
Each cusp corresponds to a separate connected component of the ten-tile origami's Poincar{\'e} section and a separate surface, a ``cusp-relative'' of the original origami, given by $C_i(X,\omega)$ as in the KSW algorithm (see \S\,\ref{sec:par-gens-sl2z-orbits}). For brevity, we denote these cusp-relatives by $(X_1,\omega)$, $(X_2,\omega)$, $(X_3,\omega)$, and $(X_4,\omega)$, and their Poincar{\'e} section components by $\Omega_1$, $\Omega_2$, $\Omega_3$, and $\Omega_4$, respectively (see Figure~\ref{fig:ten-tile-STS}).

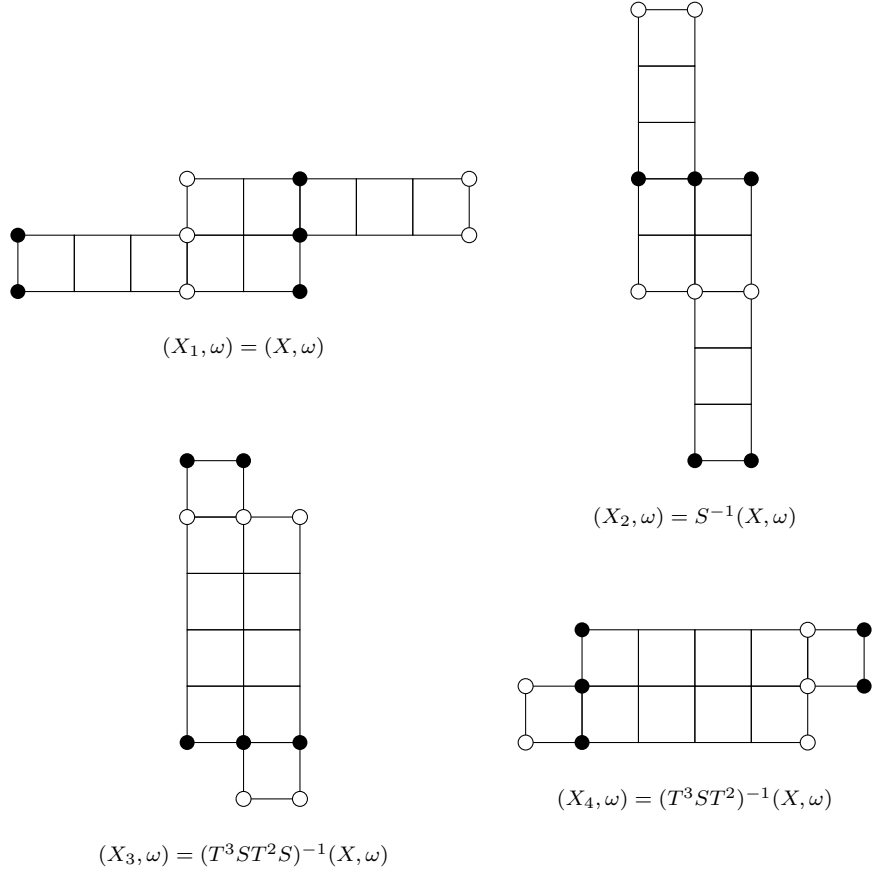
\begin{figure}
    \centering
    \begin{tikzpicture}[font=\footnotesize,scale=1.5]
    \begin{scope}[scale=.5]
    \foreach \x in {1,2,3,4,5,-4,-5,-6,-7,-8} {
        \ifnum\x>0
            \draw (\x,0) rectangle +(1,1);
        \else
            \draw ({-\x},1) rectangle +(1,1);
        \fi
    }
    \foreach \location in {(1,0),(1,1),(6,0),(6,1),(6,2)} {
        \node[bullet] at \location {};
    }
    \foreach \location in {(4,0),(4,1),(4,2),(9,1),(9,2)} {
        \node[circ] at \location {};
    }

    \node at (5,-1) {$(X_1,\omega)=(X,\omega)$};
    \end{scope}

    \begin{scope}[xshift=4cm,scale=.5,rotate around={90:(5,1)}]
    
    \foreach \x in {1,2,3,4,5,-4,-5,-6,-7,-8} {
        \ifnum\x>0
            \draw (\x,0) rectangle +(1,1);
        \else
            \draw ({-\x},1) rectangle +(1,1);
        \fi
    }
    \foreach \location in {(1,0),(1,1),(6,0),(6,1),(6,2)} {
        \node[bullet] at \location {};
    }
    \foreach \location in {(4,0),(4,1),(4,2),(9,1),(9,2)} {
        \node[circ] at \location {};
    }
    \node at (0,1) {$(X_2,\omega)=S^{-1}(X,\omega)$};
    \end{scope}

    \begin{scope}[xshift=1cm,yshift=-4cm,scale=.5,rotate around={90:(3,1)}]
    
    \foreach \x in {1,2,3,4,5,-2,-3,-4,-5,-6} {
        \ifnum\x>0
            \draw (\x,0) rectangle +(1,1);
        \else
            \draw ({-\x},1) rectangle +(1,1);
        \fi
    }
    \foreach \location in {(1,0),(1,1),(6,0),(6,1),(6,2)} {
        \node[circ] at \location {};
    }
    \foreach \location in {(2,0),(2,1),(2,2),(7,1),(7,2)} {
        \node[bullet] at \location {};
    }
    \node at (0,1) {$(X_3,\omega)=(T^3ST^2S)^{-1}(X,\omega)$};
    \end{scope}
    
    \begin{scope}[xshift=4.5cm,yshift=-4cm,scale=.5]
    
    \foreach \x in {1,2,3,4,5,-2,-3,-4,-5,-6} {
        \ifnum\x>0
            \draw (\x,0) rectangle +(1,1);
        \else
            \draw ({-\x},1) rectangle +(1,1);
        \fi
    }
    \foreach \location in {(1,0),(1,1),(6,0),(6,1),(6,2)} {
        \node[circ] at \location {};
    }
    \foreach \location in {(2,0),(2,1),(2,2),(7,1),(7,2)} {
        \node[bullet] at \location {};
    }
    \node at (4,-1) {$(X_4,\omega)=(T^3ST^2)^{-1}(X,\omega)$};
    \end{scope}
    \end{tikzpicture}
    \caption{The balanced ten-tile origami and its cusp-relatives. Opposite edges are identified.}
    \label{fig:ten-tile-STS}
\end{figure}

\subsection{Finding winners}

In this section we calculate the winning vectors on each cusp-relative. The calculations all follow the algorithm outlined in Section~\ref{ch:determining-winners}. Recall that for the sake of exposition, the Poincar{\'e} section components were rotated a quarter-turn counterclockwise in that section from their orientation in the rest of this paper.

The general procedure for each Poincar{\'e} section component proceeds as follows (in the unrotated Section~\ref{ch:introduction} orientation). We determine the winning regions by their intersection with the right edge of the Poincar{\'e} section component $a=1$. Beginning from the lower-right corner, we find the winning vector on that interval. We do so by choosing an initial guess $\ang{u_0,v_0}$ from the set of holonomy vectors on the surface $\Lambda(X,\omega)$, then iterating the SUMRY algorithm until no other vectors in $\Lambda(X,\omega)$ are left in the region of possible winners. Recall that if the current guess is $\ang{u_i,v_i}$, the region in the plane of all possible holonomy vectors is given by
\[\set{\ang{x,y}\in\mathbb{R}^2\,\middle\vert\,x>0,\,y\geq0,\,-\frac{1}{b}(x-1)<y<\frac{v_i}{u_i}x}\]
(where again we are using the ``unrotated'' orientation of the $a$-$b$ plane). Once we identify the winning vector $\ang{u_w,v_w}$, we also know the upper endpoint of its winning interval along the right edge is $b=-(u_w-1)/v_w$, which is also the lower endpoint of the next winning interval. Iterating this process, we proceed up the right edge until we reach the topmost point on the right edge, at which point we have finished partitioning the component.

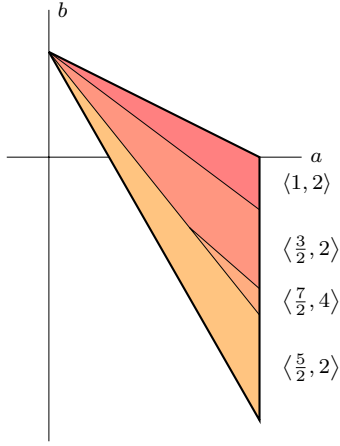
\begin{figure}
    \centering
    \begin{tikzpicture}[scale=2.8,font=\footnotesize]
    \begin{scope}[on background layer]
        \draw[ultra thin] (-.2,0) -- (1.2,0) node[right] {$a$}
         (0,0.7) node[right] {$b$} -- (0,-1.35);
    \end{scope}
    
    \definecolorseries{redseries}{hsb}{last}{red!50!white}{red!10!orange!50!white}
    \foreach \s in {0}{
    \resetcolorseries[5]{redseries}
    \begin{scope}[yslant={-1.25*\s}]
        \begin{scope}[on background layer]
            \fill[redseries!!+] (0,0.5) -- (1,0) -- (1,-.25) -- (0,0.5);
            \fill[redseries!!+] (0,0.5) -- (1,-.25) -- (1,-.625) -- ({2/3},{-1/3}) -- (0,0.5);
            \fill[redseries!!+] ({2/3},{-1/3}) -- (1,-.625) -- (1,-.75) -- ({2/3},{-1/3});
            \fill[redseries!!+] (0,0.5) -- (1,-.75) -- (1,-1.25) -- (0,0.5);
        \end{scope}

        \draw[thick] (0,0.5) -- (1,0) -- (1,-1.25) -- (0,0.5);
        \draw[ultra thin]
            (0,0.5) -- (1,-.75)
            (0,0.5) -- (1,-.25)
            ({2/3},{-1/3}) -- (1,-.625);
        
        \node[right=5pt] at (1,-.125) {$\ang{\pgfmathparse{1+2.5*\s}\pgfmathifisint{\pgfmathresult}{\pgfkeys{/pgf/number format/.cd,frac,frac whole=true}\pgfmathprintnumber{\pgfmathresult}}{\pgfkeys{/pgf/number format/.cd,frac,frac whole=false}\pgfmathprintnumber{\pgfmathresult}},2}$};
        \node[right=5pt] at (1,-.4375) {$\ang{\pgfmathparse{1.5+2.5*\s}\pgfmathifisint{\pgfmathresult}{\pgfkeys{/pgf/number format/.cd,frac,frac whole=true}\pgfmathprintnumber{\pgfmathresult}}{\pgfkeys{/pgf/number format/.cd,frac,frac whole=false}\pgfmathprintnumber{\pgfmathresult}},2}$};
        \node[right=5pt] at (1,-.6875) {$\ang{\pgfmathparse{3.5+5*\s}\pgfmathifisint{\pgfmathresult}{\pgfkeys{/pgf/number format/.cd,frac,frac whole=true}\pgfmathprintnumber{\pgfmathresult}}{\pgfkeys{/pgf/number format/.cd,frac,frac whole=false}\pgfmathprintnumber{\pgfmathresult}},4}$};
        \node[right=5pt] at (1,-1) {$\ang{\pgfmathparse{2.5+2.5*\s}\pgfmathifisint{\pgfmathresult}{\pgfkeys{/pgf/number format/.cd,frac,frac whole=true}\pgfmathprintnumber{\pgfmathresult}}{\pgfkeys{/pgf/number format/.cd,frac,frac whole=false}\pgfmathprintnumber{\pgfmathresult}},2}$};
    \end{scope}
    }
    \end{tikzpicture}
    \caption{Partition of $\Omega_1$}
    \label{fig:x1-omega}
\end{figure}

\begin{table}
    \centering
    \renewcommand{\arraystretch}{1.5}
    \[\begin{array}{@{}ccccc@{}}
    \toprule
    \textbf{Endpoint} & \textbf{Winning vector} & \textbf{Lower boundary} & \textbf{Upper boundary} \\
    \midrule
    -\frac{5}{4} & \ang{\frac{5}{2},2} & \frac{4}{5}(x-1) & \frac{4}{5}x & * \\
    -\frac{3}{4} & \ang{\frac{7}{2},4} & \frac{4}{3}(x-1) & \frac{8}{7}x & \\
    -\frac{5}{8} & \ang{\frac{3}{2},2} & \frac{8}{5}(x-1) & \frac{4}{3}x & \\
    -\frac{1}{4} & \ang{1,2} & 4(x-1) & 2x & \\
    \bottomrule
    \end{array}\]
    \caption{Calculations for winning regions on $\Omega_1$}
    \label{tab:x_1_omega}
\end{table}

Table~\ref{tab:x_1_omega} presents the partition of $\Omega_1$ into winning regions. Notice that the shortest horizontal vector component on $(X_1,\omega)$ is $2$, so that it does not contain a short horizontal saddle connection (one of length at most $1$)---we scale it by the unit-determinant diagonal matrix
\[D_{1/2}=\begin{bmatrix}
    1/2 & 0 \\ 0 & 2
\end{bmatrix}\]
so that the transformed surface $D_{1/2}(X_1,\omega)$ does. This transformation, which ensures the surface contains a short horizontal saddle connection, is necessary to apply the KSW algorithm to $(X_1,\omega)$. Notice that the action of any transformation $A$ on the original surface $(X,\omega)$ corresponds naturally to a transformation $D_{1/2}AD_{1/2}^{-1}$ on the transformed surface. In particular, the horizontal shear $T^s$ transforms into
\[\begin{bmatrix}
    1/2 & 0 \\ 0 & 2
\end{bmatrix}\begin{bmatrix}
    1 & s \\ 0 & 1
\end{bmatrix}\begin{bmatrix}
    2 & 0 \\ 0 & 1/2
\end{bmatrix}=\begin{bmatrix}
    1 & s/4 \\ 0 & 1
\end{bmatrix}=T^{s/4}\,.\]
Further, it is straightforward to verify that
\[\SL{D_{1/2}(X,\omega)}=D_{1/2}\,\SL{X,\omega}\,D_{1/2}^{-1}\,,\]
or in words, the Veech group of the transformed surface is conjugate to the Veech group of original surface by the transformation. Finally, note that the transformation reduces the width of the cusp from $5$ to $5/4$.

Where the rightmost column is unstarred, the region of possible winners is bounded, and manually checking the lattice of possible holonomy vectors shows that no other vector in $\Lambda(X_1,\omega)$ can win over the listed candidate. Where the rightmost column is starred, the region of possible winners is an unbounded strip. However, in each such instance we may transform $D_{1/2}(X_1,\omega)$ by $T^{-5/4}$, which according to the correspondence previously established lies in $\SL{D_{1/2}(X_1,\omega)}$ (recall that $T^{-5}\in\SL{X_1,\omega}$). It follows that under this transformation $\Lambda(D_{1/2}(X_1,\omega))$ is mapped invertibly onto itself, so that no vector which does not appear in the transformed possible-winner region can appear in the untransformed possible-winner region. But the transformed region is
\[\set{\ang{x,y}\in\mathbb{R}^2\,\middle\vert\,0<x<1}\,,\]
so that any vector in this region must have horizontal component $1/2$ (that is, it must cross exactly one tile horizontally). But inspection of $(X_1,\omega)$ shows that no vector can cross exactly one tile horizontally, since the minimum horizontal distance between cone points is two tiles. Hence in each case we have found the winning vector.

\begin{figure}
    \centering
    \begin{tikzpicture}[scale=2.8,font=\footnotesize]
    \begin{scope}
    \begin{scope}[on background layer]
        \draw[ultra thin] (-.2,0) -- (1.2,0) node[right] {$a$}
         (0,1.1) node[right] {$b$} -- (0,-1);
        \fill[blue!70!cyan!50!white] (0,0.5) -- (1,0) -- (1,{-1/3}) -- (0.5,0) -- (0,0.5);
        \fill[blue!50!cyan!50!white] (0.5,0) -- (1,{-1/3}) -- (1,-0.5) -- (0.5,0);
        \fill[blue!30!cyan!50!white] (0,0.5) -- (1,-0.5) -- (1,-1) -- (0,0.5);
    \end{scope}

    \draw[thick] (0,0.5) -- (1,0) -- (1,-1) -- (0,0.5);
    \draw[ultra thin] (0,0.5) -- (1,-0.5);
    \draw[ultra thin] (0.5,0) -- (1,{-1/3});

    \node[right=5pt] at (1,-.25) {$\ang{1,2}$};
    \node[right=5pt] at (1,{-5/12}) {$\ang{2,3}$};
    \node[right=5pt] at (1,-.75) {$\ang{2,2}$};
    \end{scope}

    \begin{scope}[xshift=2.5cm]
    \begin{scope}[on background layer]
        \draw[ultra thin] (-.2,0) -- (1.2,0) node[right] {$a$}
         (0,1.1) node[right] {$b$} -- (0,-1);
        \fill[green!40!white] (0,1) -- (1,0) -- (1,-1) -- (0,1);
    \end{scope}

    \draw[thick] (0,1) -- (1,0) -- (1,-1) -- (0,1);
    \node[right=5pt] at (1,-.5) {$\ang{1,1}$};
    \end{scope}
    \end{tikzpicture}
    \caption{Partitions of $\Omega_2$ (left) and $\Omega_3$ (right)}
    \label{fig:x2-x3-omega}
\end{figure}
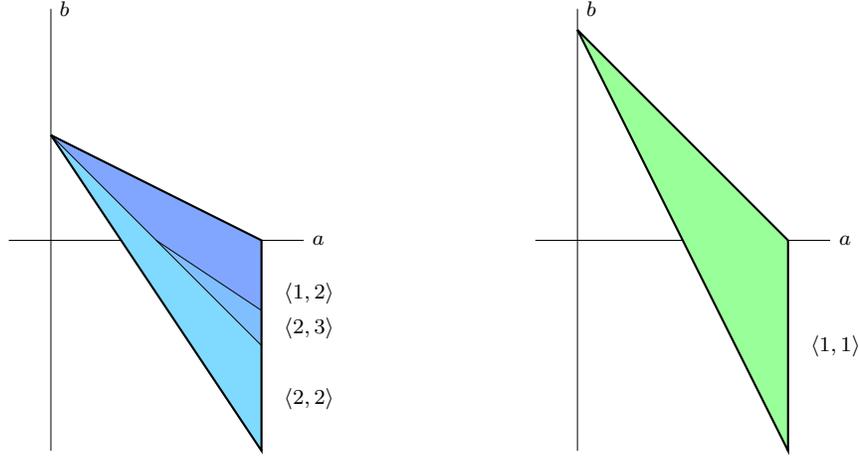

\begin{table}
    \centering
    \renewcommand{\arraystretch}{1.5}
    \[\begin{array}{@{}ccccc@{}}
    \toprule
    \textbf{Endpoint} & \textbf{Winning vector} & \textbf{Lower boundary} & \textbf{Upper boundary} \\
    \midrule
    -1 & \ang{2,2} & x-1 & x & * \\
    -\frac12 & \ang{2,3} & 2(x-1) & \frac32x & \\
    -\frac13 & \ang{1,2} & 3(x-1) & 2x & \\
    \midrule
    -1 & \ang{1,1} & x-1 & x-1 & * \\
    \bottomrule
    \end{array}\]
    \caption{Calculations for winning regions on $\Omega_2$, $\Omega_3$}
    \label{tab:x_2_x_3_omega}
\end{table}

The three remaining cusp-relatives each have a short horizontal saddle connection, so we do not scale them. It follows that, taking each tile to have unit width, every holonomy vector on these surfaces has integer components.

Two of the proposed winners for $(X_2,\omega)$ and $(X_3,\omega)$ have unbounded regions of possible winners. However, it is possible to show no points in the integer lattice $\mathbb{Z}^2$ lie in these regions, hence no holonomy vectors lie in these regions. This is because the region of possible winners always takes the form
\[\frac{p}{q}(x-1)<y<\frac{p}{q}x\,,\]
where $p$, $q$ are coprime integers. Notice that if $\ang{x,y}$ satisfies this inequality, then so does $\ang{x+pn,y+qn}$ for any integer $n$, so we need only consider vectors of horizontal component at most $p$, which gives a bounded region to check. In each case for $(X_2,\omega)$ and $(X_3,\omega)$, this bounded region is empty of holonomy vectors.

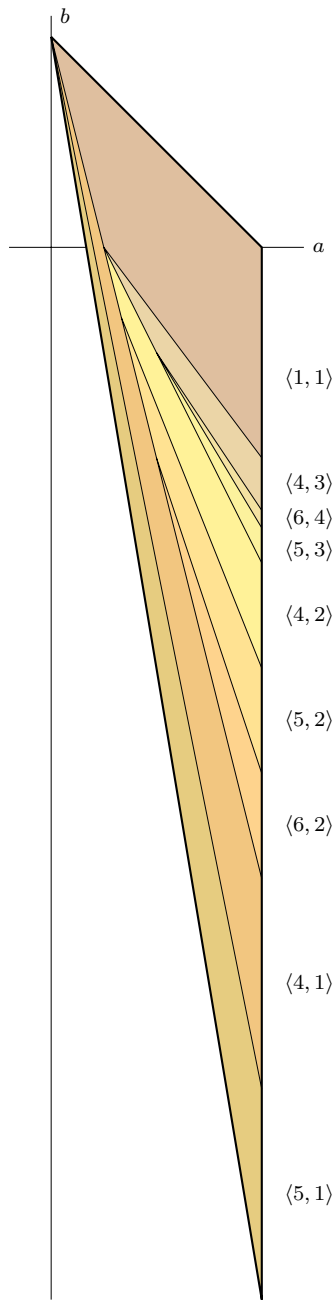
\begin{figure}
    \centering
    \begin{tikzpicture}[scale=2.8,font=\footnotesize]
    \begin{scope}[on background layer]
        \draw[ultra thin] (-.2,0) -- (1.2,0) node[right] {$a$}
         (0,1.1) node[right] {$b$} -- (0,-5);
        \fill[brown!50!white] (0,1) -- (1,0) -- (1,-1) -- (0.25,0) -- (0,1);
        \fill[yellow!30!brown!50!white] (0.25,0) -- (1,-1) -- (1,-1.25) -- (.5,-.5) -- (0.25,0);
        \fill[yellow!60!brown!50!white] (.5,-.5) -- (1,-1.25) -- (1,{-4/3}) -- (.5,-.5);
        \fill[yellow!90!brown!50!white] (.5,-.5) -- (1,{-4/3}) -- (1,-1.5) -- (.5,-.5);
        \fill[yellow!90!orange!50!white] (.25,0) -- (1,-1.5) -- (1,-2) -- ({1/3},{-1/3});
        \fill[yellow!60!orange!50!white] ({1/3},{-1/3}) -- (1,-2) -- (1,-2.5) -- (.5,-1) -- ({1/3},{-1/3});
        \fill[yellow!30!orange!50!white] (.5,-1) -- (1,-2.5) -- (1,-3) -- (.5,-1);
        \fill[orange!90!green!50!white] (0,1) -- (1,-3) -- (1,-4) -- (0,1);
        \fill[orange!80!green!50!white] (0,1) -- (1,-4) -- (1,-5) -- (0,1);
    \end{scope}

    \draw[thick] (0,1) -- (1,0) -- (1,-5) -- (0,1);
    \draw[ultra thin]
        (0,1) -- (1,-3)
        (0.25,0) -- (1,-1)
        (0.25,0) -- (.5,-.5)
        (.5,-.5) -- (1,-1.25)
        (.5,-.5) -- (1,-{4/3})
        (.5,-.5) -- (1,-1.5)
        ({1/3},{-1/3}) -- (1,-2)
        (.5,-1) -- (1,-2.5)
        (0,1) -- (1,-4)
    ;

    \node[right=5pt] at (1,-.625) {$\ang{1,1}$};
    \node[right=5pt] at (1,-1.125) {$\ang{4,3}$};
    \node[right=5pt] at (1,-{31/24}) {$\ang{6,4}$};
    \node[right=5pt,yshift=-2pt] at (1,-{17/12}) {$\ang{5,3}$};
    \node[right=5pt] at (1,-1.75) {$\ang{4,2}$};
    \node[right=5pt] at (1,-2.25) {$\ang{5,2}$};
    \node[right=5pt] at (1,-2.75) {$\ang{6,2}$};
    \node[right=5pt] at (1,-3.5) {$\ang{4,1}$};
    \node[right=5pt] at (1,-4.5) {$\ang{5,1}$};
    \end{tikzpicture}
    \caption{Partition of $\Omega_4$}
    \label{fig:x4-omega}
\end{figure}

\begin{table}
    \centering
    \renewcommand{\arraystretch}{1.5}
    \[\begin{array}{@{}ccccc@{}}
    \toprule
    \textbf{Endpoint} & \textbf{Winning vector} & \textbf{Lower boundary} & \textbf{Upper boundary} \\
    \midrule
    -5 & \ang{5,1} & \frac{1}{5}(x-1) & \frac{1}{5}x & * \\
    -4 & \ang{4,1} & \frac{1}{4}(x-1) & \frac{1}{4}x & * \\
    -3 & \ang{6,2} & \frac{1}{3}(x-1) & \frac{1}{3}x & * \\
    -\frac{5}{2} & \ang{5,2} & \frac{2}{5}(x-1) & \frac{2}{5}x & * \\
    -2 & \ang{4,2} & \frac{1}{2}(x-1) & \frac{1}{2}x & * \\
    -\frac{3}{2} & \ang{5,3} & \frac{2}{3}(x-1) & \frac{3}{5}x & \\
    -\frac{4}{3} & \ang{6,4} & \frac{3}{4}(x-1) & \frac{2}{3}x & \\
    -\frac{5}{4} & \ang{4,3} & \frac{4}{5}(x-1) & \frac{3}{4}x & \\
    -1 & \ang{1,1} & x-1 & x & * \\
    \bottomrule
    \end{array}\]
    \caption{Calculations for winning regions on $\Omega_4$}
    \label{tab:x_4_omega}
\end{table}

A similar observation holds for all but one of the parallel regions in $(X_4,\omega)$: the region of possible winners for $\ang{5,2}$ contains vectors of the form $\ang{3+5n,1+2n}$ for $n\in\mathbb{Z}$. However, inspection reveals that none of these are holonomy vectors on $(X_4,\omega)$: the horizontal spacing between cone points is $1$ and $4$, and these must alternate as a vector travels right from a starting cone point. Hence no vector on $(X_4,\omega)$ can have horizontal component congruent to $2$ or $3$ modulo $5$. In particular, no holonomy vector on $(X_4,\omega)$ has form $\ang{3+5n,1+2n}$, since every such vector has horizontal component congruent to $3$ modulo $5$.

\subsection{Return time calculations}

Integrating over each of the regions above and summing the resulting contributions produces the pdf $f(t)$, given by the equation\footnote{Code implementing this calculation can be found in Appendix~\ref{app:code}.}
\begin{equation}\label{eq:total-pdf}
\begin{gathered}
\frac{33t^2}{8}f(t)=\\
\begin{cases}
    0 & 0\leq t < 1, \\
    4\ln{t} & 1\leq t < 2, \\
    12\ln{t}-8\ln{2} & 2\leq t < 3, \\
    15\ln{t}-8\ln{2}-3\ln{3} & 3\leq t < 4, \\
    16\ln{t}-10\ln{2}-3\ln{3}-8\,\theta\left(\frac{t}{4}\right) & 4\leq t < \frac{16}{3}, \\
    16\ln{t}-10\ln{2}-3\ln{3}-8\,\theta\left(\frac{t}{4}\right)-4\,\theta\left(\frac{3t}{16}\right) & \frac{16}{3}\leq t < 6, \\
    12\ln{t}-8\ln{2}+\ln{3}-8\,\theta\left(\frac{t}{4}\right)-4\,\theta\left(\frac{t}{6}\right) & 6\leq t < 8, \\
    10\ln{t}-2\ln{2}-\ln{3}-8\,\theta\left(\frac{t}{4}\right)-12\theta\left(\frac{t}{8}\right) & 8\leq t < 9, \\
    10\ln{t}-4\ln{2}-\ln{3}-8\,\theta\left(\frac{t}{4}\right)-8\,\theta\left(\frac{t}{8}\right) & 9\leq t < \frac{32}{3}, \\
    10\ln{t}-4\ln{2}-\ln{3}-8\,\theta\left(\frac{t}{4}\right)-8\,\theta\left(\frac{t}{8}\right)-4\,\theta\left(\frac{3t}{32}\right) & \frac{32}{3}\leq t < 12, \\
    9\ln{t}-4\ln{2}-8\,\theta\left(\frac{t}{4}\right)-8\,\theta\left(\frac{t}{8}\right)-4\,\theta\left(\frac{t}{12}\right) & 12\leq t < 16, \\
    9\ln{t}-4\ln{2}-\ln{3}-8\,\theta\left(\frac{t}{4}\right)-8\,\theta\left(\frac{t}{8}\right)-2\,\theta\left(\frac{t}{12}\right) & t>16,
\end{cases}
\end{gathered}
\end{equation}
where
\[\theta(t)=\arctanh\sqrt{1-\frac{1}{t}}\,.\]
We plot this pdf in Figure~\ref{fig:ten-tile-distribution}.

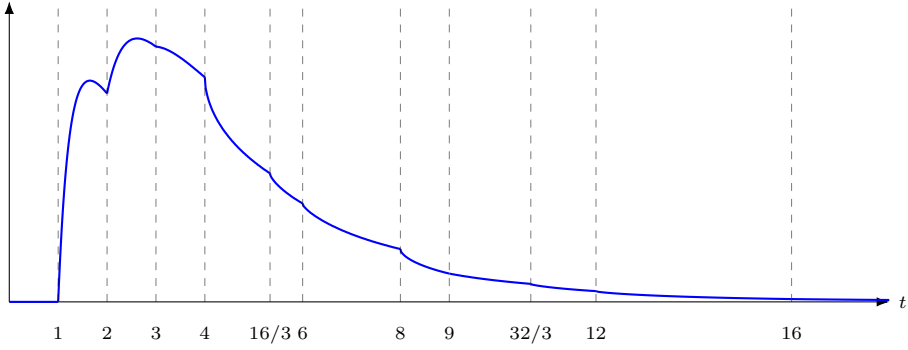
\begin{figure}
    \centering
    \begin{tikzpicture}[
        xscale = 0.65,
        declare function={
            arctanh(\x) = 0.5*ln((1+\x)/(1-\x));
            T(\x) = arctanh(sqrt(1-1/\x));
        },
        font=\scriptsize
    ]
    \begin{scope}[on background layer]
        \draw[-Latex] (0,0) -- (18,0) node[right] {$t$};
        \draw[-Latex] (0,0) -- (0,4);
    \end{scope}
    
    \begin{scope}[thick,blue,yscale=4,samples=100]
        \draw[domain=0:1,smooth,variable=\t] (0,0) plot (\t,0);
        \draw[domain=1:2,smooth,variable=\t] (0,0) plot (\t,{\t^(-2)*4*ln(\t)});
        \draw[domain=2:3,smooth,variable=\t] (0,0) plot (\t,{\t^(-2)*(12*ln(\t)-8*ln(2))});
        \draw[domain=3:4,smooth,variable=\t] (0,0) plot (\t,{\t^(-2)*(15*ln(\t)-8*ln(2)-3*ln(3))});
        \draw[domain=4:{16/3},smooth,variable=\t] (0,0) plot (\t,{\t^(-2)*(16*ln(\t)-10*ln(2)-3*ln(3)-8*T(\t/4))});
        \draw[domain={16/3}:6,smooth,variable=\t] (0,0) plot (\t,{\t^(-2)*(16*ln(\t)-10*ln(2)-3*ln(3)-8*T(\t/4)-4*T(3*\t/16))});
        \draw[domain=6:8,smooth,variable=\t] (0,0) plot (\t,{\t^(-2)*(12*ln(\t)-8*ln(2)+ln(3)-8*T(\t/4)-4*T(\t/6))});
        \draw[domain=8:9,smooth,variable=\t] (0,0) plot (\t,{\t^(-2)*(10*ln(\t)-2*ln(2)-ln(3)-8*T(\t/4)-12*T(\t/8))});
        \draw[domain=9:{32/3},smooth,variable=\t] (0,0) plot (\t,{\t^(-2)*(10*ln(\t)-4*ln(2)-ln(3)-8*T(\t/4)-8*T(\t/8))});
        \draw[domain={32/3}:12,smooth,variable=\t] (0,0) plot (\t,{\t^(-2)*(10*ln(\t)-4*ln(2)-ln(3)-8*T(\t/4)-8*T(\t/8)-4*T(3*\t/32))});
        \draw[domain=12:16,smooth,variable=\t] (0,0) plot (\t,{\t^(-2)*(9*ln(\t)-4*ln(2)-8*T(\t/4)-8*T(\t/8)-4*T(\t/12))});
        \draw[domain=16:18,smooth,variable=\t] (0,0) plot (\t,{\t^(-2)*(9*ln(\t)-4*ln(2)-ln(3)-8*T(\t/4)-8*T(\t/8)-2*T(\t/12))});
    \end{scope}
    
    \foreach\x in {1,2,3,4,{16/3},6,8,9,{32/3},12,16}{
        \draw[dashed,opacity=.5] (\x,0) -- +(0,4);
        \node[anchor=base] at (\x,-.5) {$\x$};
    }
    
    \end{tikzpicture}
    \caption{The pdf of the ten-tile distribution.}
    \label{fig:ten-tile-distribution}
\end{figure}

\subsection{Verification}
In this section we prove some results that independently confirm Equation~\ref{eq:total-pdf} is correct. We begin by proving a result about the holonomy vectors on the ten-tile origami.

\begin{thm}
A vector is a holonomy vector on the ten-tile origami if and only if it takes one of the forms $\ang{5n,k}$, $\ang{5n+2,k}$, and $\ang{5n+3,k}$, where $n$, $k$ are integers.
\end{thm}

\begin{proof}
We may cut the bottom left three tiles in Figure~\ref{fig:ten-tile-origami} and paste them onto the bottom-right of the origami to obtain the diagram in Figure~\ref{fig:pair-of-pants}. Notice that the bottom edges of the three tiles in the top right are not glued to the top edges of the three tiles in the bottom left.

Now, starting from the bottom leftmost point marked with $\circ$ in this representation, notice that the endpoint of any vector with a horizontal component of form $5n$ and $5n+2$ for integer $n$ rests on one of the three dashed blue lines. A cone point lies at every integer vertical distance above the starting point, so any vector of form $\ang{5n,k}$ or $\ang{5n+2,k}$ will be a holonomy vector for the ten-tile origami. A similar argument starting from the bottommost point marked with $\bullet$ in this representation shows that any vector of form $\ang{5n+3,k}$ will be a holonomy vector as well.

Conversely, suppose we have a holonomy vector of the ten-tile origami. Notice that the horizontal distance between any two cone points is of the form $5n$, $5n+2$, or $5n+3$ for some integer $n$, and that the vertical distance between any two cone points is an integer. It follows that the given vector has the desired form.
\end{proof}

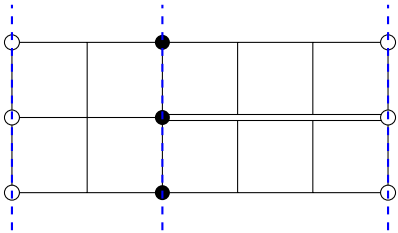
\begin{figure}
    \centering
    \begin{tikzpicture}
    \draw (0,0) grid (5,2);
    \draw[fill=white] (2,.96) rectangle (5,1.04);
    \foreach\y in {0,1,2}{
        \foreach\x in {0,5}{
            \node[circ] at (\x,\y) {};
        }
        \node[bullet] at (2,\y) {};
    }

    \draw[dashed,thick,blue] (0,-.5) -- (0,2.5)
                        (2,-.5) -- (2,2.5)
                        (5,-.5) -- (5,2.5);
    \end{tikzpicture}
    \caption{An equivalent representation of the ten-tiled origami.}
    \label{fig:pair-of-pants}
\end{figure}

Armed with this result, it is straightforward to compute all holonomy vectors with maximum component at most any given $R$ and the accompanying slope gaps. Figure~\ref{fig:congruence-histogram} shows a histogram for $R=5000$. Notice the excellent agreement with the predicted distribution.

\begin{figure}
    \centering
    \includegraphics[width=\linewidth]{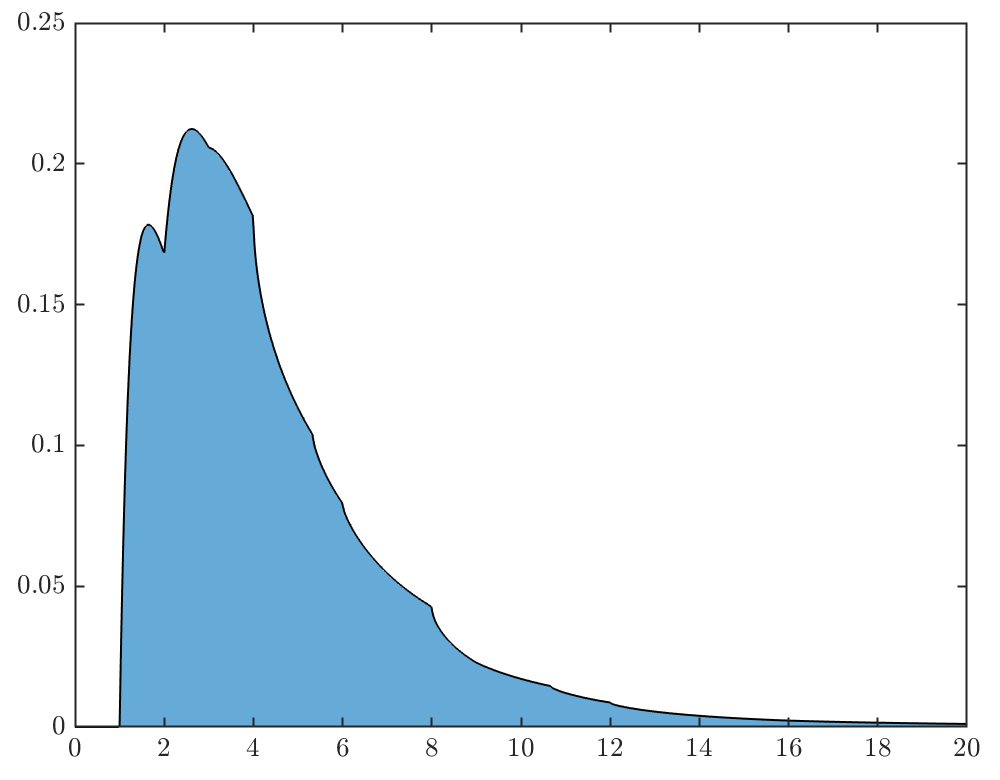}
    \caption{A histogram experimentally confirming the calculated distribution. All vectors with maximum component $R$ were selected; only the range from $t=0$ to $20$ is shown, and $100\,000$ bins are used. See Appendix~\ref{app:code} for the code.}
    \label{fig:congruence-histogram}
\end{figure}

Finally, it turns out that integrating the return time over the Poincar{\'e} section should yield the hyperbolic volume of the Veech group's fundamental domain.\footnote{This is because, outside a set of measure $0$, the entire $\SL[2]{\mathbb{R}}$-orbit of $(X,\omega)$ can be realized as a suspension space with roof function $R(a,b)$ over the Poincar\'e section. For more details, see~\cite{athreya-cheung14}.} Integrating the return time over the Poincar{\'e} section yields\footnote{Code implementing this calculation can be found in Appendix~\ref{app:code}.}
\[\int_\Omega R(a,b)\,dA=2\pi^2=12\left(\frac{\pi^2}{6}\right)\,,\]
which accords with expectation given that the Veech group of the ten-tile origami has index~$12$ in $\SL[2]{\mathbb{Z}}$ and the hyperbolic volume of the modular surface  is known to be $\frac{\pi^2}{6}$ \cite{athreya-cheung14}. The above results provide independent empirical and theoretical confirmation of the calculated distribution.

\subsection{The ten-tile origami and the Hall distribution}

Recall that the Hall distribution has pdf of the form
\[
h(t)=
\begin{cases}
    0 & 0\leq t<1\,, \\
    \dfrac{\ln{t^2}}{t^2} & 1\leq t<4\,, \\
    \dfrac{\ln{t^2}}{t^2} - \dfrac{4}{t^2}\arctanh\left(\sqrt{1-\dfrac{4}{t}}\right) & 4\leq t\,.
\end{cases}\]
A plot is repeated in Figure~\ref{fig:hall-distribution-lines}. The central insight we employ in showing that the distribution of the ten-tile origami cannot be a finite sum of Hall distributions is that the Hall distribution has a ``signature'' in the $1$-$4$ spacing of its nonsmooth points that cannot be erased through adding other Hall distributions. We show this formally below.

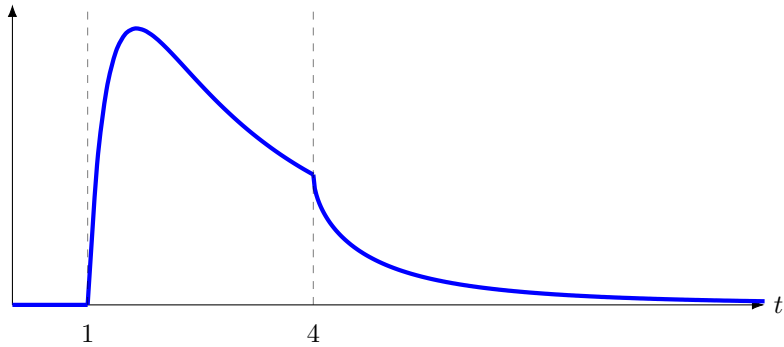
\begin{figure}
    \centering
    \begin{tikzpicture}
    \begin{scope}[on background layer]
        \draw[-Latex] (0,0) -- (10,0) node[right] {$t$};
        \draw[-Latex] (0,0) -- (0,4);
    \end{scope}

    \begin{scope}[ultra thick,blue,yscale=20]
        \draw[domain=0:1,smooth,variable=\t] (0,0) plot (\t,0);
        \draw[domain=1:4,smooth,variable=\t] (0,0) plot (\t,{(ln(\t)/(\t^2))});
        \draw[domain=4:10,samples=300,variable=\t] (0,0) plot (\t,{(ln(\t)/(\t^2)-1/(\t^2)*(ln((1+sqrt(1-4/\t))/(1-sqrt(1-4/\t)))))});
    \end{scope}

    \foreach\x in {1,4}{
        \draw[dashed,opacity=.5] (\x,0) -- +(0,4);
    }

    \node[anchor=base] at (1,-.5) {$1$};
    \node[anchor=base] at (4,-.5) {$4$};
    \end{tikzpicture}
    \caption{The pdf of the Hall distribution.}
    \label{fig:hall-distribution-lines}
\end{figure}

\begin{figure}
    \centering
    \begin{tikzpicture}
    \begin{scope}
        
    \clip (0,0) rectangle (10,4);
    
    \begin{scope}[on background layer]
        \draw[-Latex] (0,0) -- (10,0) node[right] {$t$};
        \draw[-Latex] (0,0) -- (0,4);
    \end{scope}

    \begin{scope}[thick]
        \draw[domain=0:0.5,smooth,variable=\t] (0,0) plot (\t,0);
        \draw[domain=0.5:2,smooth,variable=\t] (0,0) plot (\t,{15*(ln(2*\t)/(4*\t^2))});
        \draw[domain=2:8,smooth,samples=300,variable=\t] (0,0) plot (\t,{15*(ln(2*\t)/(4*\t^2))-15/(4*\t^2)*(ln((1+sqrt(1-2/\t))/(1-sqrt(1-2/\t)))))+20*ln(\t/2)/(\t^2)});
        \draw[domain=8:10,smooth,variable=\t] (0,0) plot (\t,{15*(ln(2*\t)/(4*\t^2))-15/(4*\t^2)*(ln((1+sqrt(1-2/\t))/(1-sqrt(1-2/\t)))))+20*ln(\t/2)/(\t^2)-20/(\t^2)*(ln((1+sqrt(1-8/\t))/(1-sqrt(1-8/\t)))))});
    \end{scope}

    \begin{scope}[ultra thick,red,dotted,yscale=5,xscale=2]
        \draw[domain=0:1,smooth,variable=\t] (0,0) plot (\t,0);
        \draw[domain=1:4,smooth,variable=\t] (0,0) plot (\t,{(ln(\t)/(\t^2))});
        \draw[domain=4:10,samples=300,variable=\t] (0,0) plot (\t,{(ln(\t)/(\t^2)-1/(\t^2)*(ln((1+sqrt(1-4/\t))/(1-sqrt(1-4/\t)))))});
    \end{scope}

    \begin{scope}[ultra thick,blue,dashed,yscale=15,xscale=.5]
        \draw[domain=0:1,smooth,variable=\t] (0,0) plot (\t,0);
        \draw[domain=1:4,smooth,variable=\t] (0,0) plot (\t,{(ln(\t)/(\t^2))});
        \draw[domain=4:20,samples=300,variable=\t] (0,0) plot (\t,{(ln(\t)/(\t^2)-1/(\t^2)*(ln((1+sqrt(1-4/\t))/(1-sqrt(1-4/\t)))))});
    \end{scope}
    \end{scope}

    \foreach\x in {.5,2,8}{
        \draw[dashed,opacity=.5] (\x,0) -- +(0,4);
    }

    \node[anchor=base] at (.5,-.5) {$\tau_2$};
    \node[anchor=base] at (2,-.5) {$\tau_1=4\tau_2$};
    \node[anchor=base] at (8,-.5) {$4\tau_1$};
    \end{tikzpicture}
    \caption{Visual representation of the sum of two Hall distributions (dotted and dashed lines: summands, solid line: sum). No relative scaling removes the nonsmooth points.}
    \label{fig:hall-distribution-sum}
\end{figure}
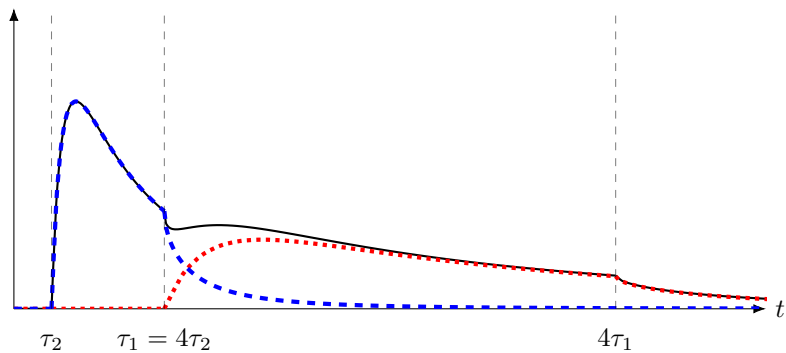

Let $\partial_\pm$ denote the right- and left-handed derivatives (i.e., right- and left-sided limits of the difference quotient) of $h$ at $x$, respectively. Direct computation shows
\begin{align*}
    \partial_-h(1) &= 0\,, & \partial_+h(1) &= 2\,; \\
    \partial_-h(4) &= \frac{1-4\ln{2}}{32}\,, & \partial_+h(4) &= -\infty\,.
\end{align*}
We use this to prove the following theorem.
\begin{thm}
Let the distribution $S$ given by
\[S(t)=\sum_{i=1}^n \alpha_i\,h_i(t)\,,\quad h_i(t)=h\left(\frac{t}{\tau_i}\right)\]
be a nonempty sum of $n$ scaled Hall distributions. Then
\begin{enumerate}
    \item The set $T$ of $t$-values at which $S$ is nonsmooth is nonempty, and
    \item For each $\tau\in T$, $\tau/4\in T$ or $4\tau\in T$.
\end{enumerate}
\end{thm}

\begin{proof}
Notice that we may assume without loss of generality that $\tau_i\neq\tau_j$ for distinct indices $i\neq j$. This is because if $\tau_i=\tau_j$, then $h_i=h_j$ and
\[\alpha_i\,h_i(t)+\alpha_j\,h_j(t)=(\alpha_i+\alpha_j)h_i(t)\,.\]
Proceeding in this way and relabeling the $\alpha_i$ at the end reduces to the case of distinct $\tau_i$, which we assume from this point on.

If $T$ is empty then $S$ is smooth. If $n=1$ this is absurd since $h(t)$ is not smooth; hence assume $n>1$. Since $h_1$ is nonsmooth at $t=\tau_1$, some other term of $S$, say $h_2$, must also be nonsmooth at $t=\tau_1$, or else $S$ itself would be nonsmooth at $t=\tau_1$. Since $h_2$ is nonsmooth at $t=\tau_2$, $4\tau_2$ and by assumption $\tau_1\neq\tau_2$, we conclude $\tau_1=4\tau_2$. By the assumed uniqueness of the $\tau_i$, this in turn implies $\tau_1/\tau_i\neq1$, $4$ for $i\geq3$.

Now if $S$ is smooth at $\tau_1$ then $S'(\tau_1)$ exists, so that $\partial_-S(\tau_1)=\partial_+S(\tau_1)$. Direct computation shows
\begin{align*}
\partial_-S(\tau_1)&=\partial_-h_1(\tau_1)+\partial_-h_2(\tau_1)+\sum_{i=3}^n\partial_-h_i(\tau_1)\\
&=\partial_-h(1)+\partial_-h(4)+\sum_{i=3}^n\partial_-h\left(\frac{\tau_1}{\tau_i}\right)\\
&=\frac{1-4\ln2}{32}+\sum_{i=3}^n\partial_-h\left(\frac{\tau_1}{\tau_i}\right)\\
&>-\infty\,,\end{align*}
where the finiteness of the last sum in the penultimate line follows from the observation that $h$ is smooth everywhere on its domain except for at $1$ and $4$, and that $\tau_1/\tau_i\neq1$, $4$ for $i\geq3$. Similarly,
\begin{align*}
\partial_+S(\tau_1)&=\partial_+h_1(\tau_1)+\partial_+h_2(\tau_1)+\sum_{i=3}^n\partial_+h_i(\tau_1)\\
&=\partial_+h(1)+\partial_+h(4)+\sum_{i=3}^n\partial_+h\left(\frac{\tau_1}{\tau_i}\right)\\
&=2-\infty+\sum_{i=3}^n\partial_+h\left(\frac{\tau_1}{\tau_i}\right)\\
&=-\infty\,,\\
\end{align*}
so $S'(\tau_1)$ does not exist, contrary to the assumption $S$ is smooth. Having arrived at a contradiction we must conclude $T$ is nonempty, proving (1).

The argument above applies equally to any $\tau_i$, showing that nonsmooth points in Hall distributions cannot cancel in the sum. Since every Hall distribution with a nonsmooth point at $t=\tau$ has also a nonsmooth point at either $t=4\tau$ or $t=\tau/4$, this proves (2).

\end{proof}

\subsection{Conclusion}

We have shown that the ten-tile origami has a distribution which is not a finite sum of Hall distributions. This is the first example known to us of such a slope gap distribution, and suggests that there is no simple relationship between the slope gap distribution of an origami and the slope gap distributions of its component tiles. Since every origami is a branched cover of the one-tile square torus, this shows that the slope gap distribution of a translation covering cannot in general be reduced in a trivial way to the distribution of the base surface.

Throughout the course of calculations, most of the non-Hall-like pieces of the distribution seemed to arise from winning regions which did not intercept the vertical axis. It is possible that all non-Hall pieces arise from these winning regions, whose origins are themselves not obvious. Other questions our work raises include whether there is some classification of square-tiled surfaces that predicts whether their slope-gap distribution is a sum of Halls, and whether there are any restrictions on what kinds of slope-gap distributions can exist for square-tiled surfaces. We leave investigation of these issues to future research.

%% file: code.tex
\section{Code}
\label{app:code}

\subsection{Cusp finding script (SageMath)}
The following code was adapted from code written by Jordan Grant and Jo O'Harrow.

\begin{minted}[breaklines]{sage}
from surface_dynamics import *
import numpy as np
text_gens=[] # declare a list to store the generators as their text labels
origami_gens=[] # declare a list to store the generators with their origamis
gens=[] # declare a list to store generators
text_pars=[] # declare a list to store the parabolic generators as their text labels
pars=[] # declare a list to store parabolic generators
s2 = np.matrix([[0, -1], # the 'S' rotation matrix
               [1, 0]])
l = np.matrix([[1, 1], # the 'T' shear matrix
               [0, 1]])

o = Origami('(1,2,3,4,5)(6,7,8,9,10)','(1,9)(2,10)') # define origami
# o = Origami('(1,2)(3,4)','(2,3)')
# o = Origami('(1,2)(3,4)','(1)(2,3)(4)') # define origami
T = o.teichmueller_curve() # teichmueller curve object
S = T.orbit_graph(s2_edges=True, s3_edges=False, l_edges=True, r_edges=False,
            vertex_labels=True) # creates orbit graph with only S and T edges
C = T.cusp_representatives() # find pairs (origami,cusp width) in each cusp
for vert in S.vertices(): # iterating over the vertices of the orbit graph
    if o.is_isomorphic(vert): # check if the original origami is isomorphic to the vertex
        o = vert # if so, rename the original origami to the isomorphic one
S.plot() # plot the orbit graph
for cusp in C: # for each cusp representative
    seq = [] # declare a list to store a 'sequence' of matrices
    text_seq = [] # declare a list to store the names of the matrices
    I = S.all_paths_iterator(starting_vertices=[o], ending_vertices=[cusp[0]], simple=True, max_length=None, trivial=False, use_multiedges=False, report_edges=True, labels=True)
    # iterate over all nonempty paths starting from o and ending on the origami of the cusp representative, never returning to the same vertex, and report its edges with their labels
    J = next(I) # run the iterator and report the edges
    for i in range(len(J)): # for each edge
        if (J[i][-1]) == 'l': # if that edge is labeled 'T'
            seq.append(np.matrix([[1, 1], # add the matrix for T to seq
               [0, 1]]))
            text_seq.append("T") # add the sequence "T" to text_seq
        elif (J[i][-1]) == 's2': # if that edge is labeled 'S'
            seq.append(np.matrix([[0, -1], # add the matrix for S to seq
               [1, 0]]))
            text_seq.append("S") # add the sequence "S"
    mat = np.identity(2) # declare a matrix
    text_mat = '' # declare an empty sequence
    origami_mat = o # initialize an origami to the original origami
    for i in range(len(seq)):
        mat = mat * seq[i] # right-multiply all the matrices onto the identity
        text_mat = text_mat + text_seq[i] # concatenate all the sequences onto the empty sequence
        if text_seq[i] == 'S':
            origami_mat = origami_mat.horizontal_symmetry().mirror()
        elif text_seq[i] == 'T':
            origami_mat = origami_mat.horizontal_twist()
    gens.append(mat) # add this product to the list of generators
    text_gens.append(text_mat) # add the sequence to the text list of generators
    origami_gens.append(origami_mat.to_standard_form()) # add the origami (in standard form) to the list of origamis
for i in range(len(gens)):
    pars.append(gens[i] * np.linalg.matrix_power(l,C[i][1]) * np.linalg.inv(gens[i]))
    text_pars.append('(' + text_gens[i] + ')T^' + str(C[i][1]) + '(' + text_gens[i] + ')^-1')

gamma = o.veech_group()
gamma.ncusps()

print('\n----------------------\n')

for i in range(len(pars)):
    print(text_pars[i])
    print('')
    print(pars[i])
    print('')
    print(origami_gens[i])
    print('\n----------------------\n')

for cusp in C: # C returns pairs (origami,cusp width)
    print(cusp[0]) # prints the origami representative
    print('')
    print(cusp[1]) # prints the cusp width
    print('')
    print(o.is_isomorphic(cusp[0]))
    cusp[0].horizontal_symmetry().show() # draws the cusp; note that surface_dynamics uses 'right' and 'up' permutations by default, but the horizontal_symmetry draws the 'up' permutations as 'down' permutations to account for personal taste
    print('\n----------------\n')
\end{minted}

\subsection{Cone point finder and holonomy vector checker (SageMath)}
\begin{minted}[breaklines]{sage}
# initialize origami
# origami = ('(1,2,3)(4,5)','(2,4)(3,5)')
 origami = ('(1,2,3,4,5)(6,7,8,9,10)','(1,9)(2,10)')
# origami = ('(1,2,3,4,5)(6,7,8,9,10)','(2,6)(3,7)(4,8)(5,9)')
right = Permutation(origami[0]) # right permutation
up = Permutation(origami[1]) # up permutation
left = right.inverse() # left permutation
down = up.inverse() # down permutation

perms = (left,down,right,up)

n = max(right.size(),up.size()) # number of tiles in origami
perms = [SymmetricGroup(n)(perm) for perm in perms] # "normalize" all permutations by adding 1-cycles back in
left = perms[0]
down = perms[1]
right = perms[2]
up = perms[3]
print(perms)

# cone point checker
# the idea is that we only need to check lower-left-corner representatives of cone points
# we start a small distance to the right of the representative, facing upwards (q = 0)
# then we make a quarter-turn around the representative and go one tile to the left, facing left (q = 1)
# another quarter turn brings us one tile down, facing down (q = 2)
# every four turns, check if the tile we're on is one we've been on already
# then, if q > 4, we've found a cone point

t_list = list(range(1,n+1)) # initial list of cone point representatives to check
cone_pts = set() # set of tiles whose lower-left corners are cone points
for t in t_list:
    print('\nPath:')
    t_init = t # initial tile we started on
    t_path = set() # tiles we land on
    q = 0 # number of quarter-turns around the cone pt rep
    while q==0 or not(mod(q,4) == 0 and t == t_init):
        print(t)
        t = perms[mod(q,4)](t) # update the tile we're on according to what direction we're facing
        q += 1
        if mod(q,4) == 0:
            t_path.add(t)
    if q > 4: # if we've made more than a full turn (i.e., we found a cone pt)
        for tile in t_path:
            if tile in t_list:
                t_list.remove(tile) # take this tile out of the list, we don't have to check it again
                cone_pts = cone_pts.union(t_path)
    print('\nCone point?',q > 4,'\nRepresentatives:',t_path)

print('\n','Cone point representatives:',cone_pts)

# holonomy vector checker

# holonomy vector, assumes both components are positive
hol_vec = (7,2)
n = hol_vec[0]
k = hol_vec[1]
m = k/n

# set tile size
w = 1 # tile width
h = 1 # tile height

# initialize tile number
for t in cone_pts:
    t_init = t
    print('\nStarting tile:',t)
    print('------------------------------------')

    # initialize coordinates within tile
    x = 0
    y = 0

    # initialize total distances crossed
    x_tot = 0
    y_tot = 0

    print('x\t y\t x_tot\t y_tot\t t') # print table headings

    # while we are not finished traversing the vector
    while x_tot < n and y_tot < k:
        # ensure x, y are in [0,w) x [0,h)
        x = w*frac(x/w)
        y = h*frac(y/h)
        if x == 0 and y == 0 and t == t_init:
            print('------------------------------------')
        # show current location
        print(x,'\t',y,'\t',x_tot,'\t',y_tot,'\t',t)
        # if the slope is shallow enough to hit a right-edge
        if m < (h-y)/(w-x):
            y_tot += m*(w-x)
            y += m*(w-x)
            x_tot += w-x
            x = w
            t = right(t)
        else:
            x_tot += (h-y)/m
            x += (h-y)/m
            y_tot += h-y
            y = h
            t = up(t)

    print(x,'\t',y,'\t',x_tot,'\t',y_tot,'\t',t)
    x = w*frac(x/w)
    y = h*frac(y/h)

    # check if we're at a lower-left corner
    if x == 0 and y == 0:

        # in general, up(right(t)) and right(up(t)) are not the same
        # but they should both be a tile w a LL cone pt or a tile w/o a LL cone pt
        # check with this code below
        print()
        print(up(right(t)),'\t',right(up(t)))
        print(int(up(right(t))) in cone_pts,'\t',int(right(up(t))) in cone_pts)

        if int(up(right(t))) in cone_pts:
            break
    else:
        print(False)
\end{minted}

\subsection{Histogram generator (MATLAB)}
\begin{minted}[breaklines]{matlab}
clc
clear
close all

%%
r = 5000;

slopes = [];

for n=5:5:r
    disp(n)
    for k = 1:n-1
        slopes = [slopes,k./[n,n-2,n-3]];
    end
    slopes = unique(slopes);
end

slopes = unique([0,slopes,1]);

%%
gaps = diff(slopes);
gaps = sort(gaps*r*r);

close
h = histogram(gaps,1e7,'Normalization','pdf','EdgeAlpha',0);
xlim([0,20])

hold on

t_mesh = linspace(0,20,1000);
s_mesh = 0*t_mesh;
i = 1;
for t = t_mesh
    s_mesh(i) = distr(t);
    i = i + 1;
end
plot(t_mesh,s_mesh,'k')

%%

function density = distr(t)
    if t > 16
        density = 9*log(t)-4*log(2)-log(3)-8*atanh(sqrt(1-4/t)) ...
            -8*atanh(sqrt(1-8/t))-2*atanh(sqrt(1-12/t));
    elseif t > 12
        density = 9*log(t)-4*log(2)-8*atanh(sqrt(1-4/t)) ...
            -8*atanh(sqrt(1-8/t))-4*atanh(sqrt(1-12/t));
    elseif t > 32/3
        density = 10*log(t)-4*log(2)-log(3)-8*atanh(sqrt(1-4/t)) ...
            -8*atanh(sqrt(1-8/t))-4*atanh(sqrt(1-32/3/t));
    elseif t > 9
        density = 10*log(t)-4*log(2)-log(3)-8*atanh(sqrt(1-4/t)) ...
            -8*atanh(sqrt(1-8/t));
    elseif t > 8
        density = 10*log(t)-2*log(2)-log(3)-8*atanh(sqrt(1-4/t)) ...
            -12*atanh(sqrt(1-8/t));
    elseif t > 6
        density = 12*log(t)-8*log(2)+log(3)-8*atanh(sqrt(1-4/t)) ...
            -4*atanh(sqrt(1-6/t));
    elseif t > 16/3
        density = 16*log(t)-10*log(2)-3*log(3)-8*atanh(sqrt(1-4/t)) ...
            -4*atanh(sqrt(1-16/3/t));
    elseif t > 4
        density = 16*log(t)-10*log(2)-3*log(3)-8*atanh(sqrt(1-4/t));
    elseif t > 3
        density = 15*log(t)-8*log(2)-3*log(3);
    elseif t > 2
        density = 12*log(t)-8*log(2);
    elseif t > 1
        density = 4*log(t);
    else
        density = 0;
    end

    density = density*8/33;

    if t > 0
        density = density/t/t;
    end
end
\end{minted}

\subsection{Total pdf (Mathematica)}
\begin{minted}[breaklines]{mathematica}
Subscript[F, 1][t] = 
 Piecewise[{{(3/t^2)*Log[t/2], 
    Inequality[2, Less, t, LessEqual, 4]}, 
       {(1/t^2)*(Log[t^4/32] + Log[1 - Sqrt[1 - 4/t]]), 
    Inequality[4, Less, t, LessEqual, 6]}, 
       {(1/t^2)*(Log[(3*t^3)/16] + Log[1 - Sqrt[1 - 4/t]] - 
       Log[1 + Sqrt[1 - 6/t]]), 
         
    Inequality[6, Less, t, LessEqual, 
     8]}, {(1/t^2)*(Log[t^3/8] + Log[1 - Sqrt[1 - 8/t]] + 
       Log[1 - Sqrt[1 - 4/t]] - 
              4*ArcTanh[Sqrt[1 - 8/t]]), 
    Inequality[8, Less, t, LessEqual, 9]}, 
       {(1/t^2)*(Log[t^3/16] + Log[1 - Sqrt[1 - 8/t]] + 
       Log[1 - Sqrt[1 - 4/t]] - 2*ArcTanh[Sqrt[1 - 8/t]]), t > 9}}, 0]

Subscript[F, 2][t] = 
 Piecewise[{{(2/t^2)*Log[t/2], 
    Inequality[2, Less, t, LessEqual, 3]}, 
       {(1/t^2)*Log[t^3/12], Inequality[3, Less, t, LessEqual, 16/3]}, 
       {(1/t^2)*(Log[t^3/12] - 2*ArcTanh[Sqrt[1 - 16/(3*t)]]), 
    Inequality[16/3, Less, t, LessEqual, 6]}, 
       {(1/t^2)*(Log[t^2/4] + Log[1 - Sqrt[1 - 6/t]]), 
    Inequality[6, Less, t, LessEqual, 8]}, 
       {(1/t^2)*(Log[t^2/8] + Log[1 - Sqrt[1 - 8/t]] - 
       2*ArcTanh[Sqrt[1 - 8/t]]), t > 8}}, 0]

Subscript[F, 3][t] = 
 Piecewise[{{(1/t^2)*Log[t], Inequality[1, LessEqual, t, Less, 4]}, 
       {(1/t^2)*Log[t] - (2/t^2)*ArcTanh[Sqrt[1 - 4/t]], 4 <= t}}]

Subscript[F, 4][t] = 
 Piecewise[{{(3/t^2)*Log[t], Inequality[1, Less, t, LessEqual, 2]}, 
       {(3/t^2)*Log[t^2/2], 
    Inequality[2, Less, t, LessEqual, 3]}, {(1/t^2)*Log[t^8/72], 
         
    Inequality[3, Less, t, LessEqual, 
     4]}, {(1/t^2)*(Log[t^8/72] - Log[1 + Sqrt[1 - 4/t]] - 
              4*ArcTanh[Sqrt[1 - 4/t]]), 
    Inequality[4, Less, t, LessEqual, 16/3]}, 
       {(1/t^2)*(Log[t^8/72] - Log[1 + Sqrt[1 - 4/t]] - 
       4*ArcTanh[Sqrt[1 - 4/t]] - 2*ArcTanh[Sqrt[1 - 16/(3*t)]]), 
         Inequality[16/3, Less, t, LessEqual, 6]}, 
       {(1/t^2)*(Log[t^6/4] - Log[1 + Sqrt[1 - 4/t]] - 
       4*ArcTanh[Sqrt[1 - 4/t]] - 2*ArcTanh[Sqrt[1 - 6/t]]), 
         
    Inequality[6, Less, t, LessEqual, 
     8]}, {(1/t^2)*(Log[(16*t^4)/3] - Log[1 + Sqrt[1 - 4/t]] - 
              2*Log[1 + Sqrt[1 - 8/t]] - 4*ArcTanh[Sqrt[1 - 4/t]] - 
       2*ArcTanh[Sqrt[1 - 8/t]]), 
         
    Inequality[8, Less, t, LessEqual, 
     9]}, {(1/t^2)*(Log[(8*t^4)/3] - Log[1 + Sqrt[1 - 4/t]] - 
              2*Log[1 + Sqrt[1 - 8/t]] - 4*ArcTanh[Sqrt[1 - 4/t]]), 
    Inequality[9, Less, t, LessEqual, 32/3]}, 
       {(1/t^2)*(Log[(8*t^4)/3] - Log[1 + Sqrt[1 - 4/t]] - 
       2*Log[1 + Sqrt[1 - 8/t]] - 4*ArcTanh[Sqrt[1 - 4/t]] - 
              4*ArcTanh[Sqrt[1 - 32/(3*t)]]), 
    Inequality[32/3, Less, t, LessEqual, 12]}, 
       {(1/t^2)*(Log[8*t^3] - Log[1 + Sqrt[1 - 4/t]] - 
       2*Log[1 + Sqrt[1 - 8/t]] - 4*ArcTanh[Sqrt[1 - 4/t]] - 
              4*ArcTanh[Sqrt[1 - 12/t]]), 
    Inequality[12, Less, t, LessEqual, 16]}, 
       {(1/t^2)*(Log[(8*t^3)/3] - Log[1 + Sqrt[1 - 4/t]] - 
       2*Log[1 + Sqrt[1 - 8/t]] - 4*ArcTanh[Sqrt[1 - 4/t]] - 
              2*ArcTanh[Sqrt[1 - 12/t]]), t > 16}}, 0]

F[t] = FullSimplify[
  PiecewiseExpand[
   Subscript[F, 1][t] + Subscript[F, 2][t] + Subscript[F, 3][t] + 
         Subscript[F, 4][t]], {t > 0}, TimeConstraint -> 0.5]
\end{minted}

\subsection{Return-time integral script (Mathematica)}
\begin{minted}[breaklines]{mathematica}
R[x_, y_] = y/(a*(a*x + b*y))

(* Poincaré section component 1 *)

A1 = Integrate[
  Integrate[R[1, 2], {b, (2 - 3*a)/4, (1 - a)/2}], {a, 0, 1}]

A2 = Integrate[
   Integrate[R[3/2, 2], {b, (2 - 5*a)/4, (2 - 3*a)/4}], {a, 0, 2/3}] + 
     Integrate[
   Integrate[R[3/2, 2], {b, (2 - 7*a)/8, (2 - 3*a)/4}], {a, 2/3, 1}]

A3 = Integrate[
  Integrate[R[7/2, 4], {b, (2 - 5*a)/4, (2 - 7*a)/8}], {a, 2/3, 1}]

A4 = Integrate[
  Integrate[R[5/2, 2], {b, (2 - 7 a)/4, (2 - 5 a)/4}], {a, 0, 1}]

(* Poincaré section component 2 *)

B1 = Integrate[
   Integrate[R[1, 2], {b, (1 - 2*a)/2, (1 - a)/2}], {a, 0, 1/2}] + 
     Integrate[
   Integrate[R[1, 2], {b, (1 - 2*a)/3, (1 - a)/2}], {a, 1/2, 1}]

B2 = Integrate[
  Integrate[R[2, 3], {b, (1 - 2*a)/2, (1 - 2*a)/3}], {a, 1/2, 1}]

B3 = Integrate[
  Integrate[R[2, 2], {b, (1 - 3*a)/2, (1 - 2*a)/2}], {a, 0, 1}]

(* Poincaré section component 3 *)

C1 = Integrate[Integrate[R[1, 1], {b, 1 - 2*a, 1 - a}], {a, 0, 1}]

(* Poincaré section component 4 *)

D1 = Integrate[Integrate[R[1, 1], {b, 1 - 4*a, 1 - a}], {a, 0, 1/4}] + 
     Integrate[
   Integrate[R[1, 1], {b, (1 - 4*a)/3, 1 - a}], {a, 1/4, 1}]

D2 = Integrate[
   Integrate[R[4, 3], {b, (1 - 4*a)/2, (1 - 4*a)/3}], {a, 1/4, 1/2}] + 
     Integrate[
   Integrate[R[4, 3], {b, (1 - 6*a)/4, (1 - 4*a)/3}], {a, 1/2, 1}]

D3 = Integrate[
  Integrate[R[6, 4], {b, (1 - 5*a)/3, (1 - 6*a)/4}], {a, 1/2, 1}]

D4 = Integrate[
  Integrate[R[5, 3], {b, (1 - 4*a)/2, (1 - 5*a)/3}], {a, 1/2, 1}]

D5 = Integrate[
   Integrate[R[4, 2], {b, 1 - 4*a, (1 - 4*a)/2}], {a, 1/4, 1/3}] + 
     Integrate[
   Integrate[R[4, 2], {b, (1 - 5*a)/2, (1 - 4*a)/2}], {a, 1/3, 1}]

D6 = Integrate[
   Integrate[R[5, 2], {b, 1 - 4*a, (1 - 5*a)/2}], {a, 1/3, 1/2}] + 
     Integrate[
   Integrate[R[5, 2], {b, (1 - 6*a)/2, (1 - 5*a)/2}], {a, 1/2, 1}]

D7 = Integrate[
  Integrate[R[6, 2], {b, 1 - 4*a, (1 - 6*a)/2}], {a, 1/2, 1}]

D8 = Integrate[Integrate[R[4, 1], {b, 1 - 5*a, 1 - 4*a}], {a, 0, 1}]

D9 = Integrate[Integrate[R[5, 1], {b, 1 - 6*a, 1 - 5*a}], {a, 0, 1}]

(* Total integral *)

FullSimplify[
 Expand[(1/Pi^2)*((A1 + A2 + A3 + A4) + (B1 + B2) + 
     C1 + (D1 + D2 + D3 + D4 + D5 + D6 + D7 + D8 + D9))]]

N[(1/Pi^2)*((A1 + A2 + A3 + A4) + (B1 + B2 + B3) + 
    C1 + (D1 + D2 + D3 + D4 + D5 + D6 + D7 + D8 + D9))]
\end{minted}